\documentclass{amsart}%
\usepackage[T1]{fontenc}
\usepackage{amssymb}
\usepackage{amsfonts}
\usepackage{geometry}
\usepackage{amsmath}
\usepackage{url}
\usepackage{xcolor}
\usepackage{graphicx}%
\newcommand{\rmv}[1]{}
\newcommand{\add}[1]{{#1}}

\newcommand{\hor}{\add{\operatorname*{hor}}}
\newcommand{\ver}{ \add{\operatorname*{vert}}}
\newcommand{\checker}{\add{\operatorname*{check}}}
\newcommand{\err}{\add{\operatorname{Err}}}
\newcommand{\an}[1]{\add{\Lambda_{#1}}}
\newcommand{\bn}[1]{\add{\Pi_{#1}}}

\providecommand{\U}[1]{\protect\rule{.1in}{.1in}}
\newtheorem{theorem}{Theorem}
\theoremstyle{plain}
\newtheorem{acknowledgement}[theorem]{Acknowledgement}

\newtheorem{corollary}[theorem]{Corollary}

\newtheorem{definition}[theorem]{Definition}

\newtheorem{lemma}[theorem]{Lemma}
\newtheorem{notation}[theorem]{Notation}

\newtheorem{proposition}[theorem]{Proposition}
\newtheorem{remark}[theorem]{Remark}

\numberwithin{equation}{section}
\begin{document}
\title{Stability of weighted norm inequalities}
\author[M. Alexis]{Michel Alexis}
\address{Department of Mathematics \& Statistics, McMaster University, 1280 Main Street
West, Hamilton, Ontario, Canada L8S 4K1}
\email{alexism@mcmaster.ca}
\author[J. L. Luna-Garcia]{Jose Luis Luna-Garcia}
\address{Department of Mathematics \& Statistics, McMaster University, 1280 Main Street
West, Hamilton, Ontario, Canada L8S 4K1}
\email{lunagaj@mcmaster.ca}
\author[E.T. Sawyer]{Eric T. Sawyer}
\address{Department of Mathematics \& Statistics, McMaster University, 1280 Main Street
West, Hamilton, Ontario, Canada L8S 4K1 }
\email{Sawyer@mcmaster.ca}
\author[I. Uriarte-Tuero]{Ignacio Uriarte-Tuero}
\address{Department of Mathematics, University of Toronto\\
Room 6290, 40 St. George Street, Toronto, Ontario, Canada M5S 2E4\\
(Adjunct appointment)\\
Department of Mathematics\\
619 Red Cedar Rd., room C212\\
Michigan State University\\
East Lansing, MI 48824 USA}
\email{ignacio.uriartetuero@utoronto.ca}
\thanks{E. Sawyer is partially supported by a grant from the National Research Council
of Canada}
\thanks{I. Uriarte-Tuero has been partially supported by grant MTM2015-65792-P
(MINECO, Spain), and is partially supported by a grant from the National
Research Council of Canada}

\begin{abstract}
We show that while individual Riesz transforms are two weight norm
\textbf{stable} under biLipschitz change of variables on $A_{\infty}$ weights,
they are two weight norm \textbf{unstable} under even rotational change of
variables on doubling weights. More precisely, we show that individual Riesz
transforms are unstable under a set of rotations having full measure, which
includes rotations arbitrarily close to the identity. This provides an
operator theoretic distinction between $A_{\infty}$ weights and doubling weights.

More generally, all iterated Riesz transforms of odd order are rotationally
unstable on pairs of doubling weights, thus demonstrating the need for
characterizations of iterated Riesz transform inequalities using testing
conditions as in \add{\cite{NTV4,LaSaShUr3,SaShUr7, SaShUr10, AlSaUr, LaWi}}, as opposed to the typically stable 'bump' conditions.

\end{abstract}
\maketitle
\tableofcontents

\section{Introduction}

We begin by describing two stability theorems for operator norms given three
decades apart, that motivate the main results of this paper.

\subsection{Previous stability results}

Thirty-five years ago, Johnson and Neugebauer \cite[Theorem 2.10 (a), see also
the preceding Remark 1]{JoNe} characterized the smooth homeomorphisms
$\Phi:\mathbb{R}^{n}\rightarrow\mathbb{R}^{n}$ that preserve Muckenhoupt's
$A_{p}\left(  \mathbb{R}^{n}\right)  $ condition \add{for} a weight $w$ under
pushforward by $\Phi$, as precisely those quasiconformal maps $\Phi$ having
their Jacobian $J=\left\vert \det D\Phi\right\vert $ in the intersection
$\bigcap_{r>1}A_{r}\left(  \mathbb{R}^{n}\right)  $ of the $A_{r}$ classes
over $r>1$. A variant of the one-dimensional case of this beautiful
characterization, see \cite[Theorem 2.7 with $\alpha=1$]{JoNe}, can be
reformulated in terms of \emph{stability} of the `Muckenhoupt' one weight norm
inequality for the Hilbert transform under homeomorphisms of the real line.

\begin{theorem}
Suppose that $w:\mathbb{R\rightarrow}\left[  0,\infty\right)  $ is a
nonnegative weight on the real line $\mathbb{R}$, that $\varphi
:\mathbb{R\rightarrow R}$ is an increasing homeomorphism with $\varphi$ and
$\varphi^{-1}$ absolutely continuous, and that $H$ is the Hilbert transform,
$Hf\left(  x\right)  =\operatorname{p.v.}\int_{-\infty}^{\infty}\frac{f\left(
y\right)  }{y-x}$.\newline For $1<p<\infty$, denote by $\mathfrak{N}%
_{H;p}\left[  w\right]  $ the operator norm of the map $H:L^{p}\left(
w\right)  \mathbb{\rightarrow}L^{p}\left(  w\right)  $, i.e. the best constant
$C$ in the inequality
\[
\int_{\mathbb{R}}\left\vert Hf\left(  x\right)  \right\vert ^{p}w\left(
x\right)  dx\leq C^{p}\int_{\mathbb{R}}\left\vert f\left(  x\right)
\right\vert ^{p}w\left(  x\right)  dx.
\]
\newline Then there is a positive constant $C_{1}$ such that
\[
\mathfrak{N}_{H;p}\left[  \left(  w\circ\varphi\right)  \varphi^{\prime
}\right]  \leq C_{1}\mathfrak{N}_{H,p}\left[  w\right]  ,\ \ \ \ \ \text{for
all weights }w,
\]
\emph{if and only if} $\varphi^{\prime}\in\bigcap_{r>1}A_{r}\left(
\mathbb{R}\right)  $.
\end{theorem}

More recently, Tolsa \cite[see abstract]{Tol} characterized the
`Ahlfors-David' one weight inequality for the Cauchy transform, equivalently
the $1$-fractional vector Riesz transform $\mathbf{R}^{1,2}$ in the plane
$\mathbb{R}^{2}$ (defined in (\ref{def Riesz}) below), in the case $p=2$, namely
\begin{equation}
\int_{\mathbb{R}^{2}}\left\vert \mathbf{R}^{1,2}\left(  f\mu\right)  \left(
x\right)  \right\vert ^{2}d\mu\left(  x\right)  \leq\mathfrak{N}%
_{\mathbf{R}^{1,2};2}^{2}\left(  \mu\right)  \int_{\mathbb{R}^{2}}\left\vert
f\left(  x\right)  \right\vert ^{2}d\mu\left(  x\right)  , \label{Tolsa inequ}%
\end{equation}
in terms of a growth condition and Menger curvature.\ As a consequence, Tolsa
obtained stability of the operator norm $\mathfrak{N}_{\mathbf{R}%
^{1,2}\mathbf{;}2}\left(  \mu\right)  $ under biLipschitz pushforwards of the
measure $\mu$. Even more recently, in papers by D\k{a}browski and Tolsa \cite{DaTo} and
Tolsa \cite{Tol2}, this result was extended to higher dimensions, and as a consequence
they obtained stability of the operator norm $\mathfrak{N}_{\mathbf{R}%
^{1,n};2}\left(  \mu\right)  $ of the $1$-fractional vector Riesz transform
$\mathbf{R}^{1,n}$ under biLipschitz pushforwards of the measure $\mu$ in
$\mathbb{R}^{n}$ \cite[see the comment at the top of page 6]{DaTo},
\cite{Tol2}. As an important application of norm stability, they obtain the
stability of removable sets for Lipschitz harmonic functions under biLipschitz
mappings \cite[see Corollary 1.6 and the discussion surrounding it]{Tol2}.

Here we define the $\alpha$-fractional vector Riesz transform in
$\mathbb{R}^{n}$ by
\begin{equation}
\mathbf{R}^{\alpha,n}f\left(  x\right)  \equiv c_{\alpha,n}\operatorname{p.v.}%
\int_{\mathbb{R}^{n}}\frac{x-y}{\left\vert x-y\right\vert ^{n+1-\alpha}%
}f\left(  y\right)  dy,\ \ \ \ \ x\in\mathbb{R}^{n},0\leq\alpha<n.
\label{def Riesz}%
\end{equation}
Let $\mathbf{R}_{j}^{\alpha,n}=\left(  R_{1}^{\alpha,n},...,R_{2}^{\alpha
,n}\right)  $, where we refer to the components $R_{j}^{\alpha,n}$ as
individual $\alpha$-fractional Riesz transforms in $\mathbb{R}^{n}$. \add{We are primarily concerned with the 
classical case $\alpha=0$ in this paper,} and so we will usually drop the superscript $\alpha$ and
write $\mathbf{R}=\left(  R_{1},...,R_{n}\right)  $ when the dimension
$n$ is understood, and refer to the components $R_{j}$ as Riesz transforms. 

The main problem we consider in this paper is the extent to which the above
theorems hold in the setting of \emph{two weight} norm inequalities, and to
include more general operators in higher dimensions. The complexities inherent
in dealing with two weight norm inequalities - mainly that they are no longer
characterized simply by $A_{p}$-like conditions or more generally by
conditions of `positive nature', but require testing conditions of `singular
nature' as well - suggests that we should limit ourselves to consideration of
biLipschitz maps. Indeed, this much smaller class of maps is much more
amenable to current two weight techniques, and allows for a rich theory where
stability holds in certain `nice' situations, while failing in small
perturbations of these `nice' situations. We also show in the appendix that
any reasonable group of transformations under which the two weight $A_{2}$
condition is stable \add{is} contained in the group of biLipschitz transformations.

Our analysis will be mainly restricted to the case $p=2$ and iterated Riesz
transforms of odd order in $\mathbb{R}^{n}$, where we show that stability of
the two weight norm inequality is sensitive to the distinction between
doubling and $A_{\infty}$ weights, even when the biLipschitz maps are
restricted to rotations of $\mathbb{R}^{n}$.

\subsection{Description of results}

The two weight norm inequality for an operator $T$ with a pair $\left(
\sigma,\omega\right)  $ of positive locally finite Borel measures on
$\mathbb{R}^{n}$ and exponents $1<p\leq q<\infty$ is informally,%
\begin{equation}
\left(  \int_{\mathbb{R}^{n}}\left\vert T\left(  f\sigma\right)  \right\vert
^{q}d\omega\right)  ^{\frac{1}{q}}\leq\mathfrak{N}_{T}\left(  \int
_{\mathbb{R}^{n}}\left\vert f\right\vert ^{p}d\sigma\right)  ^{\frac{1}{p}%
},\ \ \ \ \ f\in L^{p}\left(  \sigma\right)  . \label{two weight norm}%
\end{equation}
See \add{Definition \ref{def bounded} for a formal definition of the two weight norm inequality.}  In the case $p=q=2$, we first
establish a distinction between weighted norm inequalities for positive
operators $T$ in (\ref{two weight norm}), such as the maximal function and
fractional integrals, on the one hand; and singular integral operators $T$ in
(\ref{two weight norm}), such as the individual Riesz transforms and iterated
Riesz transforms, on the other hand. Namely, that the former are two weight
norm stable under biLipschitz change of variables for arbitrary locally finite
positive Borel measures,\emph{\ }while the latter are not in general, even on
pairs of doubling measures.

Our main result, Theorem \ref{stab}, shows that while individual Riesz
transforms are two weight norm \textbf{stable} under biLipschitz change of
variables on pairs of $A_{\infty}$ weights, they are two weight norm
\textbf{unstable} under even a rotational change of variables on doubling
weights. This provides an operator theoretic distinction between $A_{\infty}$
weights and doubling weights\footnote{In 1974, C. Fefferman and B. Muckenhoupt
\cite{FeMu} constructed an example of a doubling weight that was not
$A_{\infty}$ using a self similar construction, on which many subsequent
results have been based.}.

We also show that all iterated Riesz transforms of odd order are rotationally
unstable on pairs of doubling weights, thus demonstrating the need for
characterizations of iterated Riesz transform inequalities using unstable
conditions, such as the testing conditions in \cite{NTV4, LaSaShUr3, SaShUr7, SaShUr10, AlSaUr, LaWi}, as opposed to the
typically stable `bump' conditions\add{, see Section \ref{subsection:sparse_bumps}}.

\subsection{BiLipschitz and rotational stability}
\add{In this subsection, we precisely define stability.}
\begin{definition}
\label{def:transformation_classes}Let $\Phi:\mathbb{R}^{n}\rightarrow
\mathbb{R}^{n}$ be continuous and invertible.

\begin{enumerate}
\item $\Phi$ is \emph{biLipschitz} if%
\[
\left\Vert \Phi\right\Vert _{\operatorname*{biLip}}\equiv\sup_{x,y\in
\mathbb{R}^{n}}\frac{\left\vert \Phi\left(  x\right)  -\Phi\left(  y\right)
\right\vert }{\left\vert x-y\right\vert }+\sup_{x,y\in\mathbb{R}^{n}}%
\frac{\left\vert \Phi^{-1}\left(  x\right)  -\Phi^{-1}\left(  y\right)
\right\vert }{\left\vert x-y\right\vert }<\infty.
\]

\item $\Phi$ is a rotation if $\Phi$ is linear and $\Phi\Phi^{\ast}=I$ and
$\det\Phi=1$.
\end{enumerate}
\end{definition}

Let $\mathcal{X}$ be a group of continuous invertible maps on $\mathbb{R}^{n}%
$, such as the group of biLipschitz or rotation transformations, which we
denote by $\mathcal{X}_{\operatorname*{biLip}}$ and $\mathcal{X}%
_{\operatorname*{rot}}$ respectively\footnote{See Lemma \ref{group} in the
appendix for a justification of considering subgroups of biLipschitz
transformations.}. Denote by $\mathcal{M}$ the space of positive Borel
measures on $\mathbb{R}^{n}$,\ and by $\Phi_{\ast}\mu$ the pushforward of
$\mu\in\mathcal{M}$ by a continuous map $\Phi:\mathbb{R}^{n}\rightarrow
\mathbb{R}^{n}$, i.e.\ $\Phi_{\ast}\mu(B)\equiv\mu(\Phi^{-1}(B))$. We say that
a subclass $\mathcal{S}\subset\mathcal{M}$ of positive Borel measures is
\emph{$\mathcal{X}$-invariant} if $\Phi_{\ast}\mu\in\mathcal{S}$ for all
$\mu\in\mathcal{S}$ and $\Phi\in\mathcal{X}$. Of course $\mathcal{M}$ itself
is $\mathcal{X}$-invariant for the group $\mathcal{X}_{\operatorname*{cont}%
\operatorname*{inv}}$ of all continuous invertible maps, but less trivial
examples of \emph{biLipschitz} invariant classes include,%
\begin{align}
\mathcal{S}_{A_{p}}  &  \equiv\left\{  \mu\in\mathcal{M}:d\mu\left(  x\right)
=u\left(  x\right)  dx\text{ with }u\in A_{p}\right\}  ,\ \ \ \ \ \text{for
}1\leq p<\infty,\label{classes}\\
\mathcal{S}_{A_{\infty}}  &  \equiv\left\{  \mu\in\mathcal{M}:d\mu\left(
x\right)  =u\left(  x\right)  dx\text{ with }u\in A_{\infty}\right\} \add{,} 
\nonumber\\
\mathcal{S}_{\operatorname*{doub}}  &  \equiv\left\{  \mu\in\mathcal{M}%
:\mu\text{ is a doubling measure}\right\}  ,\nonumber\\
\mathcal{S}_{\operatorname*{ADs}}  &  \equiv\left\{  \mu\in\mathcal{M}%
:\mu\text{ is Ahlfors-David regular of degree }s\right\}  ,\nonumber\\
\mathcal{S}_{\operatorname*{lfpB}}  &  \equiv\left\{  \mu\in\mathcal{M}%
:\mu\text{ is a locally finite positive Borel measure}\right\}  .\nonumber
\end{align}
To each of the above classes $\mathcal{S}$ we can associate a functional
$\left\Vert \mu\right\Vert _{\mathcal{S}}$ for which $\mathcal{S}%
\equiv\left\{  \mu\in\mathcal{M}:\left\Vert \mu\right\Vert _{\mathcal{S}%
}<\infty\right\}  $. For example we take
\begin{equation}
\left\Vert \mu\right\Vert _{\mathcal{S}_{A_{\infty}}}=\left[  \mu\right]
_{A_{\infty}}\equiv\sup_{Q}\left(  \frac{1}{\left\vert Q\right\vert }\int
_{Q}\mu\right)  \exp\left(  \frac{1}{\left\vert Q\right\vert }\int_{Q}\ln
\frac{1}{\mu}\right)  \,, \label{A infinity char}%
\end{equation}
and $\left\Vert \mu\right\Vert _{\mathcal{S}_{\operatorname*{doub}}%
}=C_{\operatorname*{doub}}\left(  \mu\right)  $ as in Definition
\ref{doubling def}. In the case that $\mathcal{S}=\mathcal{S}%
_{\operatorname{lfpB}}$, there is no `natural' choice of $\Vert\cdot
\Vert_{\mathcal{S}}$ that measures the `size' of the measure $\mu$ and so
instead we\add{, may, for instance,} define
\[
\left\Vert \mu\right\Vert _{\mathcal{S}_{\operatorname*{lfpB}}}=%
\begin{cases}
1 & \text{ if }\mu\in\mathcal{S}_{\operatorname*{lfpB}}\\
\infty & \text{ otherwise }%
\end{cases}
\,.
\]
We also define
\[
\left\Vert \Phi\right\Vert _{\mathcal{X}}=\left\{
\begin{array}
[c]{ccc}%
\left\Vert \Phi\right\Vert _{\operatorname*{biLip}} & \text{ if } &
\mathcal{X}=\mathcal{X}_{\operatorname*{biLip}}\\
1 & \text{ if } & \mathcal{X}=\mathcal{X}_{\operatorname*{rot}}\text{ and
}\Phi\in\mathcal{X}_{\operatorname*{rot}}\\
\infty & \text{ if } & \mathcal{X}=\mathcal{X}_{\operatorname*{rot}}\text{ and
}\Phi\not \in \mathcal{X}_{\operatorname*{rot}}%
\end{array}
\right.  .
\]
Here is the main stability definition for a function $\mathcal{F}$ on measure
pairs, a group $\mathcal{X}\in\left\{  \mathcal{X}_{\operatorname*{biLip}%
},\mathcal{X}_{\operatorname*{rot}}\right\}  $ and an $\mathcal{X}$-invariant
class $\mathcal{S}$ (or to be precise, for $(\mathcal{S},\Vert\cdot
\Vert_{\mathcal{S}}$)).

\begin{definition}\label{defn:stability}
Let $\mathcal{X}\in\left\{  \mathcal{X}_{\operatorname*{biLip}},\mathcal{X}%
_{\operatorname*{rot}}\right\}  $, $\mathcal{S}\subset\mathcal{M}$ be
$\mathcal{X}$-invariant, and let $\mathcal{F}:\mathcal{S}\times\mathcal{S}%
\rightarrow\left[  0,\infty\right]  $ be a nonnegative extended real-valued
function on the product set $\mathcal{S}\times\mathcal{S}$. We say that the
function $\mathcal{F}$ is \emph{$\mathcal{X}$-stable on }$\mathcal{S}$ if
there is a function $\mathcal{G}:[0,\infty)^{4}\rightarrow\lbrack0,\infty)$
which maps bounded subsets of $[0,\infty)^{4}$ to bounded subsets of
$[0,\infty)$, such that%
\begin{align}
\mathcal{F}\left(  \Phi_{\ast}\sigma,\Phi_{\ast}\omega\right)   &
\leq\mathcal{G}\left(  \left\Vert \Phi\right\Vert _{\mathcal{X}}%
,\mathcal{F}\left(  \sigma,\omega\right)  ,\left\Vert \sigma\right\Vert
_{\mathcal{S}},\left\Vert \omega\right\Vert _{\mathcal{S}}\right)
,\label{FG}\\
\text{for all }\sigma,\omega &  \in\mathcal{S}\text{ such that }%
\mathcal{F}(\sigma,\omega)<\infty\text{ and all }\Phi\in\mathcal{X}.\nonumber
\end{align}

\end{definition}

Note that to check \add{that} $\mathcal{G}$ maps bounded sets to bounded sets, it
suffices to show for instance that $\mathcal{G}$ is continuous. Typically, we
will take $\mathcal{F}$ to be an operator norm on weighted spaces, in which case we say an operator $T$ is (un)stable on a class of measures $\mathcal{S}$ if its two weight operator norm is (un)stable on $\mathcal{S}$. One
may also take $\mathcal{F}$ to be a common \add{two weight} bump condition. 

A simple example of a biLipschitz stable function on the class $\mathcal{S}%
_{\operatorname{lfpB}}$ is the classical two weight $A_{2}$ characteristic for
a pair of measures, namely
\[
\mathcal{F}\left(  \sigma,\omega\right)  =A_{2}\left(  \sigma,\omega\right)
=\sup_{\text{cubes }Q\text{ in }\mathbb{R}^{n}}\frac{\left\vert Q\right\vert
_{\sigma}}{\left\vert Q\right\vert }\frac{\left\vert Q\right\vert _{\omega}%
}{\left\vert Q\right\vert }.
\]
Indeed,
\[
\frac{\left\vert Q\right\vert _{\Phi_{\ast}\sigma}}{\left\vert Q\right\vert
}\frac{\left\vert Q\right\vert _{\Phi_{\ast}\omega}}{\left\vert Q\right\vert
}=\frac{\left\vert \Phi^{-1}Q\right\vert _{\sigma}}{\left\vert Q\right\vert
}\frac{\left\vert \Phi^{-1}Q\right\vert _{\omega}}{\left\vert Q\right\vert
}\approx\frac{\left\vert \Phi^{-1}Q\right\vert _{\sigma}}{\left\vert \Phi
^{-1}Q\right\vert }\frac{\left\vert \Phi^{-1}Q\right\vert _{\omega}%
}{\left\vert \Phi^{-1}Q\right\vert },
\]
since $\Phi^{-1}$ is biLipschitz, and now observe that there is a cube $P$
such that $P\subset\Phi^{-1}Q\subset\rho P$ for some $\rho>1$ by
quasiconformality of $\Phi$ \cite[Lemma 3.4.5]{AsIwMa}, where $\rho$ depends
only on \add{$\left \| \Phi\right \| _{\operatorname*{biLip}}$}. Thus we have%
\[
\frac{\left\vert Q\right\vert _{\Phi_{\ast}\sigma}}{\left\vert Q\right\vert
}\frac{\left\vert Q\right\vert _{\Phi_{\ast}\omega}}{\left\vert Q\right\vert
}\lesssim\frac{\left\vert \rho P\right\vert _{\sigma}}{\left\vert \rho
P\right\vert }\frac{\left\vert \rho P\right\vert _{\omega}}{\left\vert \rho
P\right\vert }\leq A_{2}\left(  \sigma,\omega\right)  ,
\]
and by taking supremums over cubes gives
\begin{equation}
A_{2}\left(  \Phi_{\ast}\sigma,\Phi_{\ast}\omega\right)  \leq\mathcal{G}%
\left(  \left\Vert \Phi\right\Vert _{\operatorname*{biLip}},A_{2}\left(
\sigma,\omega\right)  ,\left\Vert \sigma\right\Vert _{\mathcal{S}%
_{\operatorname{lfpB}}},\left\Vert \omega\right\Vert _{\mathcal{S}%
_{\operatorname{lfpB}}}\right)  =\mathcal{G}\left(  \left\Vert \Phi\right\Vert
_{\operatorname*{biLip}},A_{2}\left(  \sigma,\omega\right)  ,1,1\right)
\label{A2 stab}%
\end{equation}
for $\mathcal{G}(w,x,y,z)=cw^{4n}x$, where $c>0$ is independent of $\Phi$,
$\sigma$ and $\omega$.

The reader can also check that all of the usual `Orlicz bump' conditions
\begin{align*}
&  \sup_{Q\text{ a ball}}\left\Vert u^{\frac{1}{p}}\right\Vert _{A,Q}%
\left\Vert v^{-\frac{1}{p}}\right\Vert _{B,Q}<\infty,\\
\text{where }  &  \left\Vert f\right\Vert _{A,Q}\equiv\inf\left\{
\lambda>0:\frac{1}{\left\vert B\right\vert }\int_{B}A\left(  \frac{\left\vert
f\left(  x\right)  \right\vert }{\lambda}\right)  dx\right\}  ,
\end{align*}
on a pair of absolutely continuous measures $\sigma\left(  x\right)  dx$ and
$\omega\left(  x\right)  dx$ on $\mathbb{R}^{n}$ as in the conjecture of
Cruz-Uribe and Perez \cite{CrPe} (proved by Lerner - see \cite{Ler}), are
biLipschitz stable on any biLipschitz invariant subclass $\mathcal{S}$, e.g., Neugebauer's bump condition,%
\[
A_{2,r}\left(  \sigma,\omega\right)  =\sup_{\text{cubes }Q\text{ in
}\mathbb{R}^{n}}\left(  \frac{1}{\left\vert Q\right\vert }\int_{Q}%
\sigma\left(  x\right)  ^{r}dx\right)  ^{\frac{1}{r}}\left(  \frac
{1}{\left\vert Q\right\vert }\int_{Q}\omega\left(  x\right)  ^{r}dx\right)
^{\frac{1}{r}},
\]
where $1<r<\infty$. \add{See Appendix \ref{subsection:sparse_bumps}.}

More recently, additional variants of bump condition, such as entropy bumps
and separated bumps, have arisen in work of Treil and Volberg \cite{TrVo2}, Lacey and Spencer \cite{LaSp}
to mention just a few. The sufficiency of these bump conditions for two weight
singular integral inequalities all go through the boundedness of sparse
operators - see Lerner \cite{Ler} for a proof of the optimal result to date,
and a history of this fascinating subject. \add{In Appendix \ref{subsection:sparse_bumps},} we show that no
such bump conditions can characterize the two weight norm inequality for an
iterated Riesz transform $T$ of odd order even when the measures are doubling
(or for any Calder\'{o}n-Zygmund operator $T$ that is biLipschitz unstable on
doubling measures).

We mention in passing that the following form of the two weight $A_{p}$
condition on the real line,%
\[
\widetilde{A}_{p}\left(  v,w\right)  \equiv\sup_{I\text{ an interval}}\left(
\frac{1}{\left\vert I\right\vert }\int_{I}w\right)  \left(  \frac
{1}{\left\vert I\right\vert }\int_{I}\frac{1}{v^{p^{\prime}-1}}\right)
^{p-1},
\]
has been proved stable under an increasing homeomorphic change of variable
$\varphi$ (with both $\varphi$ and $\varphi^{-1}$ absolutely continuous) if
and only if $\varphi^{\prime}\in A_{1}\left(  \mathbb{R}\right)  $, see
\cite[Corollary 4.4]{JoNe}, but this condition is no longer equivalent to
boundedness of the Hilbert transform for two weights, and moreover, the definitions of
stability of $\widetilde{A}_{2}\left(  v,w\right)  $ and $A_{2}\left(  \sigma,\omega\right)  $ considered above are a priori different since
composition and pushforward don't commute, e.g. when $p=2$, $\Phi_{\ast}%
v\neq\left(  \Phi_{\ast}v^{-1}\right)  ^{-1}$ in general.

\subsubsection{Main results}

Our main result below on both \emph{stability} and \emph{instability} involves
Riesz transforms and doubling measures, as well as Stein elliptic
Calder\'{o}n-Zygmund operators. Recall that if $K$ is a Calder\'{o}n-Zygmund
kernel, i.e. satisfies%
\begin{align}
\left\vert K\left(  x,y\right)  \right\vert  &  \leq C_{\operatorname*{CZ}%
}\left\vert x-y\right\vert ^{-n},\label{size and smoothness}\\
\left\vert \nabla_{x}K\left(  x,y\right)  \right\vert +\left\vert \nabla
_{y}K\left(  x,y\right)  \right\vert  &  \lesssim C_{\operatorname*{CZ}%
}\left\vert x-y\right\vert ^{-n-1},\nonumber
\end{align}
and if $T$ is a bounded linear operator on unweighted $L^{2}\left(
\mathbb{R}^{n}\right)  $, we say that $T$ is \emph{associated with} the kernel
$K$ if%
\[
Tf\left(  x\right)  =\int K\left(  x,y\right)  f\left(  y\right)
dy,\ \ \ \ \ \text{for all }x\in\mathbb{R}^{n}\smallsetminus
\operatorname*{supp}f,
\]
and we refer to such operators as \emph{Calder\'{o}n-Zygmund operators}. Note
in particular that a Calder\'{o}n-Zygmund operator $T$ is bounded on
unweighted $L^{2}\left(  \mathbb{R}^{n}\right)  $ by definition. Following
\cite[(39) on page 210]{Ste2}, we say that a Calder\'{o}n-Zygmund operator $T$
is \emph{elliptic in the sense of Stein} if there is a unit vector
$\mathbf{u}_{0}\in\mathbb{R}^{n}$ and a constant $c>0$ such that%
\[
\left\vert K\left(  x,x+t\mathbf{u}_{0}\right)  \right\vert \geq c\left\vert
t\right\vert ^{-n},\ \ \ \ \ \text{for all }t\in\mathbb{R},
\]
where $K\left(  x,y\right)  $ is the kernel of $T$.

\add{Note that a function $\mathcal{F}$ being $\mathcal{X}$-stable means the estimate (\ref{FG}) holds across \emph{all} measure pairs and \emph{all} functions in the class $\mathcal{X}$, while to show a function $\mathcal{F}$ is \emph{not} $\mathcal{X}$-stable, it suffices to construct \emph{ a sequence} of measure pairs and \emph{ a sequence } of functions in $\mathcal{X}$ for which the arguments of $\mathcal{G}$ in (\ref{FG}) remain bounded, but $\mathcal{F}$ diverges to $\infty$, i.e., (\ref{FG}) fails for any choice of $\mathcal{G}$. For this last point, in this paper we will always prove instability via this last strategy. In this paper, we consider norms as in (\ref{two weight norm}) for $p=q=2$.}

\begin{theorem}
\label{stab}The two weight operator norms for individual Riesz transforms
$R_{j}$, and more generally any Stein elliptic Calder\'{o}n-Zygmund operator,
are biLipschitz stable on $\mathcal{S}_{A_{\infty}}$. The individual Riesz
transforms, as well as iterated Riesz transforms of odd order, are not even
\emph{rotationally} stable on $\mathcal{S}_{\operatorname*{doub}}$, and even
when the measures \add{are restricted to }have doubling constants $C_{\operatorname*{doub}}$
arbitrarily close to $2^{n}$.
\end{theorem}

In fact, we can prove the following stronger rotational instability for
iterated Riesz transforms of odd order, which in particular shows that
instability can hold for rotations arbitrarily close to the identity.

\begin{theorem}
\label{prop:perturbations_rotation} Iterated Riesz transforms of odd order are
unstable on $\mathcal{S}_{\operatorname*{doub}}$ under a set of rotations
having full measure.
\end{theorem}

In contrast to the instability assertions in these theorems, most positive
operators, such as maximal functions and fractional integral operators, are
easily seen to be biLipschitz stable on $\mathcal{S}_{A_{p}}$, $\mathcal{S}%
_{A_{\infty}}$, $\mathcal{S}_{\operatorname*{doub}}$ and $\mathcal{S}%
_{\operatorname*{lfpB}}$.

For example, if $T=I_{\alpha}$ is the fractional integral of order
$0<\alpha<n$, and if $\Phi:\mathbb{R}^{n}\rightarrow\mathbb{R}^{n}$ is
biLipschitz, then%
\begin{align*}
\left\Vert T_{\Phi_{\ast}\sigma}f\right\Vert _{L^{2}\left(  \Phi_{\ast}%
\omega\right)  }^{2}  &  =\int_{\mathbb{R}^{n}}\left\vert \int_{\mathbb{R}%
^{n}}\left\vert x-y\right\vert ^{\alpha-n}f\left(  y\right)  d\Phi_{\ast
}\sigma\left(  y\right)  \right\vert ^{2}d\Phi_{\ast}\omega\left(  x\right) \\
&  =\int_{\mathbb{R}^{n}}\left\vert \int_{\mathbb{R}^{n}}\left\vert \Phi
^{-1}x-\Phi^{-1}y\right\vert ^{\alpha-n}f\left(  \Phi^{-1}y\right)
d\sigma\left(  y\right)  \right\vert ^{2}d\omega\left(  x\right) \\
&  \approx\int_{\mathbb{R}^{n}}\left\vert \int_{\mathbb{R}^{n}}\left\vert
x-y\right\vert ^{\alpha-n}\left(  f\circ\Phi^{-1}\right)  \left(  y\right)
d\sigma\left(  y\right)  \right\vert ^{2}d\omega\left(  x\right)  =\left\Vert
T_{\sigma}\left(  f\circ\Phi^{-1}\right)  \right\Vert _{L^{2}\left(
\omega\right)  }^{2}%
\end{align*}
and%
\[
\left\Vert f\right\Vert _{L^{2}\left(  \Phi_{\ast}\sigma\right)  } ^\add{2}%
=\int_{\mathbb{R}^{n}}\left\vert f\left(  y\right)  \right\vert ^{2}%
d\Phi_{\ast}\sigma\left(  y\right)  =\int_{\mathbb{R}^{n}}\left\vert f\left(
\Phi^{-1}y\right)  \right\vert ^{2}d\sigma\left(  y\right)  =\left\Vert
f\circ\Phi^{-1}\right\Vert _{L^{2}\left(  \sigma\right)  }^{2}\ .
\]
A similar proof holds for the case when $T$ is a fractional maximal operator
of order $0\leq\alpha<n$.

\subsection{History of stability\label{history subsection}}

The class of Calder\'{o}n-Zygmund kernels $K\left(  x,y\right)  $ satisfying
(\ref{size and smoothness}) has long been known to be invariant under
biLipschitz change of variable $x=\Phi\left(  u\right)  $. For example, if
$K_{\Phi}\left(  u,v\right)  =K\left(  \Phi\left(  u\right)  ,\Phi\left(
v\right)  \right)  $, then the chain rule gives
\[
\left\vert \nabla_{u}K_{\Phi}\left(  u,v\right)  \right\vert =\left\vert
D\Phi\left(  u\right)  \left(  \nabla_{x}K\right)  \left(  u,v\right)
\right\vert \lesssim\left\Vert D\Phi\right\Vert _{\infty}C_{\operatorname*{CZ}%
}\left\vert u-v\right\vert ^{-n-1}\leq\left\Vert \Phi\right\Vert
_{\operatorname*{biLip}}C_{\operatorname*{CZ}}\left\vert u-v\right\vert
^{-n-1}.
\]
It follows that if a Calder\'{o}n-Zygmund operator $T$ associated with such a
kernel $K$ satisfies the two weight norm inequality (\ref{two weight norm}),
then the pullback $T_{\Phi}$ with kernel $K_{\Phi}$ is also a
Calder\'{o}n-Zygmund operator (by a simple change of variables using that the
Jacobian of $\Phi$ is bounded between two positive constants), and satisfies
the inequality (\ref{two weight norm}) with the pair of measures $\left(
\sigma,\omega\right)  $ replaced by the pair of pushforwards $\left(
\Phi_{\ast}\sigma,\Phi_{\ast}\omega\right)  $. This raises the question of
when $T$ itself satisfies (\ref{two weight norm}) with the pair of
pushforwards $\left(  \Phi_{\ast}\sigma,\Phi_{\ast}\omega\right)  $ when
$\Phi$ is biLipschitz. Roughly speaking, our results show that the answer is
\textbf{yes} if the measures $\sigma,\omega$ are $A_{\infty}$ weights, but
\textbf{no} in general if the measures $\sigma,\omega$ are just doubling.

In \cite{LaSaUr}, it was mentioned that the two weight norm inequality for the
Hilbert transform is \textquotedblleft unstable,\textquotedblright\ in the
sense that for $\omega$ equal to the Cantor measure, and $\sigma$ an
appropriate choice of weighted point masses in each removed middle third, the
norm of the operator could go from finite to infinite with just arbitrarily
small perturbations of the locations of the point masses, while the
$\mathcal{A}_{2}$ condition remained in force. In the appendix, we use this
example to show that the Hilbert transform is two weight norm \emph{unstable}
under biLipschitz pushforwards of arbitrary measure pairs, and this
instability extends to Riesz transforms in higher dimensions in a
straightforward way. Thus the Riesz transforms in higher dimensions are
biLipschitz \emph{unstable} on arbitrary weight pairs, something which already
shows that the more familiar bump-type conditions, e.g. \cite[Theorem 3]{Neu},
cannot characterize the two-weight problem for Riesz transforms alone.

On the other hand, we show below that Riesz transforms are biLipschitz stable
under pairs of $A_{\infty}$ weights. So on one hand, for pairs of arbitrary
measures we have instability, and on the other hand for pairs of $A_{\infty}$
weights, we have stability. This begs the question, what side-conditions on
the weights in our weight pairs will give stability/instability for Riesz
transforms? Now it is trivial that $A_{\infty}$ weights are doubling weights,
but it wasn't until the famous construction of Fefferman and Muckenhoupt in
\cite{FeMu} that one knew the two classes were in fact different. Because of
this, doubling is often considered to be the next more general condition on a
weight than $A_{\infty}$.

The main result of this paper is that individual Riesz transforms are
biLipschitz - and even \emph{rotationally} - unstable for pairs of doubling
weights. This provides an operator-theoretic means of distinguishing
$A_{\infty}$ weights from doubling weights, sharpening the result of Fefferman
and Muckenhoupt, by showing that stability differentiates the two classes.

\subsubsection{Our methods and their history}

In 1976, Muckenhoupt and Wheeden showed in \cite{MuWh} that the two-weight
norm inequality for the maximal function $M$ implies the one-tailed
$\mathcal{A}_{2}$ condition, and conjectured that it was sufficient. Then in
1982, the third author disproved that conjecture in \cite{Saw1} by starting
with a pair of simple radially decreasing weights $V,U$ constructed by
Muckenhoupt in \cite{Muc}, that were essentially constant on dyadic intervals
$I_{k}=[2^{-k-1},2^{-k}]$ and failed the two weight inequality for $M$. Then
the weights were \emph{disarranged} into weights $v,u$, i.e. dilates and
translates of the weights restricted to the dyadic intervals $I_{k}$ were
essentially redistributed onto the unit interval $\left[  0,1\right]  $
according to a self-similar \textquotedblleft
transplantation\textquotedblright\ rule. The resulting weights satisfied the
one-tailed $\mathcal{A}_{2}$ condition on $\left[  0,1\right]  $ but failed
the two-weight norm inequality for $M$.\footnote{The reader can easily check
that for a discretized version of these weights, the dyadic square function
defined in Section \ref{section:dyadic_prelims} also has infinite two-weight
norm.} However, such weights were not doubling, as follows from calculations in \cite{Saw1}. This significant obstacle remained until the pioneering work of
Nazarov \cite{Naz} and \cite{NaVo}, to which we now turn.

Some years later, Treil and Volberg showed in \cite{TrVo} that the two-weight
norm inequality for the Hilbert transform $H$ implies the two-tailed
$\mathcal{A}_{2}$ condition, and Sarason conjectured the two-tailed condition was also sufficient
\cite[s. 7.9]{HaNi}. Shortly after that, Nazarov disproved the conjecture in
\cite{Naz} (which we were unable to locate till very recently, using the
references in \cite{KaTr}), even using \emph{doubling} weights, in a beautiful
proof involving the Bellman technique and a brilliant supervisor, or
remodeling, argument - see also \cite{NaVo} for the details. This use of
doubling weights here turns out to be crucial for our purposes. More
specifically, Nazarov's method consisted of first using the Bellman technique
in a delicate argument to construct a weight pair $\left(  V,U\right)  $ on
$\mathbb{T}$ that failed to satisfy the two weight inequality for the discrete
Hilbert transform, but satisfied both dyadic doubling, with constant
arbitrarily close to that of Lebesgue, and dyadic $A_{2}$. Then he
transplanted highly oscillating functions according to a certain self-similar
`supervisor' rule having roots in \cite{Bou}, that resulted in a pair of
weights $\left(  v,u\right)  $ on $\mathbb{T}$ that satisfied the two-tailed
$\mathcal{A}_{2}$ condition, with doubling constant arbitrarily close to that
of Lebesgue measure, and for which the testing condition was increasingly
unbounded. Nazarov's argument requires the clever use of highly oscillatory
functions in order to deal with the singularity of the Hilbert transform, and
the use of holomorphic function theory to prove weak convergence results for
these increasingly oscillatory functions.

Very recently, it has come to our attention that Kakaroumpas and Treil
extended Nazarov's results to $p\neq2$ using a non-Bellman and `remodeling'
construction \cite{KaTr}. More precisely, Kakaroumpas and Treil first began
with a pair of discretized weights with the $A_{p}$ condition under control, a
bilinear form involving the Haar shift having increasingly large norm, but
doubling constant just as large. They then apply an iterative
\emph{disarrangement} of these weights to then construct new weights for which
the $A_{p}$ condition and the norm of the bilinear form remain essentially
unchanged, but the dyadic doubling constant of the weights is much closer to
that of Lebesgue measure. This clever disarrangement is \add{one of the} innovative idea\add{s}
which replaces Nazarov's Bellman construction, and provides weights for which
one can compute explicit quantities. It is possible that our Riesz transform
results can be proved using the Haar shift scheme of Kakaroumpas and Treil in
place of the square function scheme of Nazarov, but we have not checked the details.

Note that the rotational stability problem is only significant in 
dimension \add{two or higher,} since in one dimension the only rotation is reflection about the
origin, and that preserves the Hilbert transform. Our proof of rotational
instability in higher dimensions begins by using the Bellman construction in
\cite{NaVo}, and is then inspired by Nazarov and Volberg's supervisor argument
with highly oscillatory functions. In particular, we extend Nazarov's
supervisor/remodelling construction to higher dimensions, which we call
\textquotedblleft transplantation\textquotedblright,\ and which makes explicit how $v,u$ are constructed by \emph{transplanting\ averages} of $V,U$. 

We also need to extend Nazarov's weak convergence results to higher
dimensions, where holomorphic function theory is no longer available. This
requires \add{the} new arguments \add{in Section \ref{section:action_of_Riesz}}, comprising much of the technical difficulty of the
present paper. We must also prove\ that testing conditions hold at all scales
for one of the Riesz transforms, something not considered in \cite{NaVo}.
Finally, in the Appendix, we provide proofs of those portions of the supervisor argument required for our theorem that not detailed in
\cite{NaVo}; one may also consult \cite{KaTr} for additional arguments.

\begin{remark}
In our construction, we show that a given iterated Riesz transform $T_{0}$ of
order $N=2m+1$ fails one of the testing conditions, while all other iterated
Riesz transforms $T$ of order $N=2m+1$ satisfy both testing conditions. Thus
at this point, we have doubling measures satisfying the $A_{2}$ condition with doubling constant arbitrarily close to that of Lebesgue measure and
both testing conditions for $T$. We now need to conclude that $T$ is two
weight bounded. Since the doubling constants can be taken arbitrarily close to that of Lebesgue measure, then the
$A_{2}$ condition implies the classical energy condition \cite{Gri}, and so
\add{one can apply the T1 theorem of \cite{SaShUr10}; see Theorem \ref{thm:starting_T1_thm} below for a more precise statement and proof.}
\end{remark}

\subsection{Proof of Stability}

We present here a simple proof of stability in Theorem \ref{stab}, using a few
classical facts on weights from \cite{Neu} and \cite{CoFe}. The case of
$A_{\infty}$ weights in Lemma \ref{Neugebauer_A2_iff} below is folklore from
decades ago, but seems to have first been recorded in Hyt\"{o}nen and Lacey
\cite{HyLa}, where they also prove a sharp dependence on the characteristics
using much deeper tools. We begin with the following lemma of Neugebauer.

\begin{lemma}
[{\cite[Theorem 3]{Neu}}]Let $\left(  u,v\right)  $ be a pair of nonnegative
functions. Then there exists $W\in A_{p}$ with $c_{1}u\leq W\leq c_{2}v$ if
and only if there is $r>1$ such that%
\[
\sup_{Q}\left(  \frac{1}{\left\vert Q\right\vert }\int_{Q}u^{r}\right)
\left(  \frac{1}{\left\vert Q\right\vert }\int_{Q}v^{r\left(  1-p^{\prime
}\right)  }\right)  ^{p-1}<\infty.
\]

\end{lemma}

Recall that a weight $w$ is a weak $A_{\infty}$ weight, written $w\in
\operatorname*{weak}A_{\infty}$, if any of the following equivalent conditions
hold for all cubes $Q$ and subsets $E$ (see e.g.\cite{Saw}):%
\begin{align}
\exists R &  <\infty\text{ and }\phi\left(  t\right)  \nearrow\text{ with
}\lim_{t\searrow0}\phi\left(  t\right)  =0\text{ such that }\frac{\left\vert
E\right\vert _{w}}{\left\vert RQ\right\vert _{w}}\leq\phi\left(
\frac{\left\vert E\right\vert }{\left\vert Q\right\vert }\right)
,\label{weak A}\\
\forall R &  >1,\exists C,\varepsilon>0\text{ such that }\frac{\left\vert
E\right\vert _{w}}{\left\vert RQ\right\vert _{w}}\leq C\left(  \frac
{\left\vert E\right\vert }{\left\vert Q\right\vert }\right)  ^{\varepsilon
},\nonumber\\
\exists r &  >1\text{ such that }\left(  \int_{Q}w^{r}\right)  ^{\frac{1}{r}%
}\leq\frac{1}{\left\vert 2Q\right\vert }\int_{2Q}w.\nonumber
\end{align}

\begin{lemma}
[{\cite[see 1.2 Theorem]{HyLa}}]\label{Neugebauer_A2_iff} Suppose that $T$ is
a sufficiently regular\footnote{see 6.13 on page 221 of \cite{Ste2} for
definitions, and for the nature of the `sufficiently regular' assumption.}
Calder\'{o}n-Zygmund operator, and that both $\omega$ and $\sigma$ are
$\operatorname*{weak}A_{\infty}$ weights. Then $T$ satisfies the two weight
norm inequality%
\[
\left\Vert T_{\sigma}f\right\Vert _{L^{2}\left(  \omega\right)  }^{2}\leq
C\left\Vert f\right\Vert _{L^{2}\left(  \sigma\right)  }^{2},
\]
if $A_{2}\left(  \sigma,\omega\right)  <\infty$.
\end{lemma}

\begin{proof}
Since $\sigma$ and $\omega$ each satisfy a weak reverse H\"{o}lder condition
(the third line in (\ref{weak A})) for some $r>1$, we have
\[
A_{2,r}\left(  \sigma,\omega\right)  \equiv\sup_{Q}\left(  \frac{1}{\left\vert
Q\right\vert }\int_{Q}\omega^{r}\right)  ^{\frac{1}{r}}\left(  \frac
{1}{\left\vert Q\right\vert }\int_{Q}\sigma^{r}\right)  ^{\frac{1}{r}}%
\lesssim\sup_{Q}\left(  \frac{1}{\left\vert 2Q\right\vert }\int_{2Q}%
\omega\right)  \left(  \frac{1}{\left\vert 2Q\right\vert }\int_{2Q}%
\sigma\right)  =A_{2}\left(  \sigma,\omega\right)  .
\]
Now we apply Neugebauer's lemma with $p=2$ to the weight pair $\left(
u,v\right)  =\left(  \omega,\sigma^{-1}\right)  $ to obtain that there exists
$W\in A_{2}$ with $c_{1}\omega\left(  x\right)  \leq W\left(  x\right)  \leq
c_{2}\sigma\left(  x\right)  ^{-1}$. Then the extension of the weighted
inequality of Coifman and Fefferman \cite{CoFe} for Calder\'{o}n-Zygmund
operators given in \cite[6.13 on page 221]{Ste2} shows that%
\begin{align*}
&  \left\Vert T_{\sigma} f\right\Vert _{L^{2}\left(  \omega\right)  }^{2}\leq c_1^{-1} \left\Vert
T_{\sigma} f\right\Vert _{L^{2}\left(  W\right)  }^{2}\leq C c_1 ^{-1} \left\Vert f \sigma \right\Vert
_{L^{2}\left(  W\right)  }^{2}\leq C c_{1}^{-1} c_2 \left\Vert f \sigma \right\Vert _{L^{2}\left(
\sigma^{-1}\right)  }^{2},\\
&  \ \ \ \ \ \ \ \ \ \ \ \ \ \ \text{i.e. }\left\Vert T_{\sigma}f\right\Vert
_{L^{2}\left(  \omega\right)  }^{2}\leq C c_1^{-1} c_2  \left\Vert f\right\Vert
_{L^{2}\left(  \sigma\right)  }^{2}\ ,
\end{align*}
for all Calder\'{o}n-Zygmund operators $T$.
\end{proof}

\begin{remark}
Define a measure pair $\left(  \sigma,\omega\right)  $ to be \emph{universal}
(for boundedness of smooth Stein-elliptic Calder\'{o}n-Zygmund operators) if a smooth Stein-elliptic Calder\'{o}n-Zygmund operator $T$ is bounded from $L^{2}\left(  \sigma\right)
$ to $L^{2}\left(  \omega\right)  $ if and only if all such operators are so
bounded. Lemma \ref{Neugebauer_A2_iff} above shows that pairs of $A_{\infty}$
weights are universal, and Theorem \ref{stab} above shows that not all pairs
of doubling measures are universal.
\end{remark}

\begin{proof}
[Proof of stability in Theorem \ref{stab}] Suppose the norm inequality $\left\Vert T_{\sigma}f\right\Vert _{L^{2}\left(
\omega\right)  }^{2}\leq\mathfrak{N}_{T}\left(  \sigma,\omega\right)
^{2}\left\Vert f\right\Vert _{L^{2}\left(  \sigma\right)  }^{2}$ holds for a
Calder\'{o}n-Zygmund operator $T$ associated with a kernel $K$, and a pair of
$A_{\infty}$ weights $\left(  \sigma,\omega\right)  $. Since (\ref{A2 stab})
implies the biLipschitz stability of $A_{2}\left(  \sigma,\omega\right)  $,
and since the $A_{\infty}$-characteristics $[\sigma]_{A_{\infty}}$ and
$[\omega]_{A_{\infty}}$ are easily seen to be biLipschitz stable as well (in
fact they are stable under the more general class of quasiconformal change of
variables \cite[Theorem 2]{Uch}), we conclude that the norm inequality also
holds for the Calder\'{o}n-Zygmund operator $T_{\Phi}$ with kernel%
\[
K_{\Phi}\left(  x,y\right)  \equiv K\left(  \Phi\left(  x\right)  ,\Phi\left(
y\right)  \right)  .
\]
As mentioned at the beginning of Subsection \ref{history subsection},
$T_{\Phi}$ is a Calder\'{o}n-Zygmund operator whenever $T$ is, i.e. satisfies
\add{(\ref{size and smoothness})} and is bounded on unweighted $L^{2}\left(
\mathbb{R}^{n}\right)  $. Thus we conclude from Lemma \ref{Neugebauer_A2_iff}
that $T$ is bounded on the weight pair $\left(  \Phi_{\ast}\sigma,\Phi_{\ast
}\omega\right)  $.

We can also be more precise in our proof of stability, since \cite[Theorem
1.2]{HyLa} implies that the function
\[
\mathcal{G}(w,x,y,z)\equiv Cw^{\alpha_{\mathcal{X}}}x\left(  y^{\beta
_{\mathcal{X}}}+z^{\beta_{\mathcal{X}}}\right)  \,
\]
satisfies (\ref{FG}) for the functional $\mathcal{F}=\mathfrak{N}_{T}\left(
\sigma,\omega\right)  $, where $\alpha_{\mathcal{X}}$ and $\beta_{\mathcal{X}%
}$ are appropriately chosen \add{exponents}.
\end{proof}

\begin{remark}
Let $T$ be a strongly elliptic vector of Calder\'{o}n-Zygmund operators as in
\cite[see Theorem 2.6.]{SaShUr7}. Then two weight boundedness of $T$ implies
the two weight $A_{2}$ condition \cite[Lemma 4.1]{SaShUr7}. Thus if $\sigma$
and $\omega$ are $\operatorname*{weak}A_{\infty}$ weights, then Lemma
\ref{Neugebauer_A2_iff} shows that the two weight norm inequality for $T$
holds if and only if the $A_{2}$ condition holds. It follows that $T$ is
biLipschitz stable on
\[
\mathcal{S}_{\operatorname*{weak}A_{\infty}}\equiv\left\{  \mu\in
\mathcal{M}:d\mu\left(  x\right)  =u\left(  x\right)  dx\text{ with }%
u\in\operatorname*{weak}A_{\infty}\right\}  .
\]
We do not know if all Stein elliptic Calder\'{o}n-Zygmund operators are
biLipschitz stable on $\mathcal{S}_{\operatorname*{weak}A_{\infty}}$.
\end{remark}

The proof of instability in Theorem \ref{stab} is much more complicated.
\begin{itemize}
    \item In Section \ref{section:dyadic_prelims}, we show there exist dyadically doubling weights $U,V$ on $[0,1]^n$ which fail a square function testing condition.
    \item In Section \ref{section:sprvsr_trans}, we describe Nazarov's ``supervisor'' disarrangement of the weights $U,V$ into doubling weights $u,v$ on $[0,1]^n$, and we see how the weights $u,v$ are a linear combination of the oscillatory functions $s_k ^{\hor, P}$.
    \item In Section \ref{section:action_of_Riesz}, we study how the Riesz transforms interact with these oscillatory functions.
    \item Then in Section \ref{section:full_proof}, we show that the norm inequality for $R_1$ fails on the weights $(v,u)$ by showing the testing condition on $[0,1]$ is at least as large as the square function testing condition for $(V,U)$, while the dyadic testing conditions for $R_2$ holds for the weight pair $(u,v)$. We then extend $u,v$ to all of $\mathbb{R}^n$, and using that $u,v$ are doubling with doubling constant close to that of Lebesgue measure, we get dyadic testing for $R_2$ implies the norm inequality for $R_2$.
    \item In Section
\ref{section:iterated_Riesz}, we then extend our results to show that
individual iterated Riesz transforms of odd order are rotationally unstable.
\end{itemize}   
 
\subsection{Open Problems}

The question of stability of operator norms for singular integrals on weighted
spaces is in general wide open. Here are two instances that might be more accessible.

\begin{enumerate}
\item Only iterated Riesz transforms of \emph{odd} order are treated in
Theorem \ref{stab}. Are Riesz transforms of even order, such as the real and
imaginary parts of the Beurling transform, stable under rotations, or more
generally biLipschitz pushforwards?

\item While the individual Riesz transforms $R_{j}$ are unstable under
rotations of $\mathbb{R}^{n}$, the vector Riesz transform $\mathbf{R}=\left(
R_{1},R_{2},...,R_{n}\right)  $ is clearly rotationally stable since it is
invariant under rotations. In fact, as mentioned at the beginning of the
paper, D\k{a}browski and Tolsa \cite[see the top of page 6]{DaTo}, \cite{Tol2}
have demonstrated biLipschitz stability in the Ahlfors-David one weight
setting for the $1$-fractional vector Riesz transform $\mathbf{R}^{1,n}$. This
motivates the question of whether or not the vector Riesz transform
$\mathbf{R}$ of fractional order $0$ is biLipschitz stable on $\mathcal{S}%
_{\operatorname*{doub}}$ in the two weight setting.
\end{enumerate}

\begin{acknowledgement}
We thank D. Cruz-Uribe, K. Moen and X. Tolsa for valuable comments. \add{We also thank two referees for helpful feedback, which helped make the final version of this paper much more accessible.}
\end{acknowledgement}

\section{Preliminaries: grids, doubling, telescoping identities and dyadic
testing}

\label{section:dyadic_prelims}

We begin by introducing some notation, Haar bases and the telescoping
identity. Then we recall the beautiful Bellman construction used in
\cite{NaVo} to obtain the dyadic weights $V,U$.

\subsection{Notation for grids and cubes}

\add{Given a cube $J$,} let $\mathcal{D}(J)$ denote the collection of dyadic subcubes of $J$, and
for each $m\geq0$ let $\mathcal{D}_{m}(J)$ denote the dyadic subcubes $I$ of
$J$ satisfying $\ell(I)=2^{-m}\ell(J)$. Let $\mathcal{P}(J)$ denote the
collection of subcubes of $J$ with sides parallel to the coordinate axes, and
$\mathcal{P}^{0}\equiv\mathcal{P}([0,1]^{n})$. Unless otherwise specified, any
cube mentioned in this paper is assumed to be axis-parallel, and we denote the collection of such cubes in $\mathbb{R}^n$ by $\mathcal{P}^n$. We also define
$\mathcal{D}^{0}\equiv\mathcal{D}([0,1]^{n})$.

Given a cube $I\subset\mathbb{R}^{n}$, we will use the notational convention
\[
I=I_{1}\times I_{2}\times...\times I_{n}\,.
\]

Given a cube $I \subset\mathbb{R}^{n}$, we let $\mathfrak{C}^{(k)} (I)$ denote
the $k$th generation dyadic grandchildren of $I$, and $\mathfrak{C} (I)
\equiv\mathfrak{C}^{(1)} (I)$. And given a dyadic grid $\mathcal{D}$ and a cube $I$ in the grid, we let
$\pi_{\mathcal{D}} I$ denote the parent of $I$ in $\mathcal{D}$. The same
notation extends to arbitrary grids $\mathcal{K}$, like in Section
\ref{section:sprvsr_trans}, where $\pi_{\mathcal{K}} I$ denote the
$\mathcal{K}$-parent of $I$.

In dimension $1$, given an interval $I \subset\mathbb{R}$, let $I_{-}$ denote the left half and
$I_{+}$ denote the right half; for convenience, given a cube $I \subset
\mathbb{R}^{n}$, we also let $I_{\pm} \equiv\left(  I_{1} \right)  _{\pm}\times I_{2}
\times\ldots\times I_{n}$.

It will also be useful to keep track of the location of the children of $I$ in higher dimensions. In  $\mathbb{R}^{n}$, let $\Theta$ denote the $2^{n}$ locations a dyadic child
cube can be \add{in} relative to its parent. For instance, when $n=2$ we can take
$\Theta\equiv\left\{  \operatorname{NW},\operatorname{NE},\operatorname{SW}%
,\operatorname{SE}\right\}  $ the set of four locations of a dyadic square $Q$
within its $\mathcal{D}$-parent $\pi_{\mathcal{D}}Q$, where $\operatorname{NW}%
$ stands for Northwest, $\operatorname{NE}$ denotes Northeast, etc\dots Given
a cube $I$ and $\theta\in\Theta$, we adopt the notation that $I_{\theta}$
denotes the dyadic child of $I$ at location $\theta$.

As usual we let $\left\vert J\right\vert _{\mu}\equiv\int_{J}d\mu$ for any
positive Borel measure in $\mathbb{R}^{n}$\add{. If $\mu$ is not specified in the subscript, then $|J|$ denotes the Lebesgue measure of $J$. Also define the expectation} $E_{J}\mu\equiv\frac{1}%
{|J|}\int\limits_{J}d\mu$. Given a locally integrable function $U$ in
$\mathbb{R}^{n}$, we often abbreviate the absolutely continuous measure
$U\left(  x\right)  dx$ by $U$ as well. \add{We call $U$ a \emph{weight} if $0<U\left(  x\right)
<\infty$ for all $x\in\mathbb{R}^{n}$.}

\subsection{Doubling}

We say that two distinct cubes $Q_{1}$ and $Q_{2}$ in $\mathbb{R}^{n}$ are
\emph{adjacent} if there exists a cube $Q$ for which $Q_{1}$ and $Q_{2}$ are
dyadic children of $Q$.

\begin{definition}
\label{doubling def}Recall a measure $\mu$ on $\mathbb{R}^{n}$ is
\emph{doubling} if there exists a constant $C$ such that
\[
\mu(2Q)\leq C\mu(Q)\,\text{ for all cubes }Q\,.
\]
The smallest such $C$ is called the doubling constant for $\mu$, denoted
$C_{\operatorname*{doub}}$. Equivalently, if $\mu$ is a doubling measure, then there exists $\lambda\geq1$
such that for any two dyadic children $I$ and $J$ of an arbitrary cube
$K$
\[
\frac{E_I\mu}{E_J\mu}\in(\lambda^{-1},\lambda)\,.
\]
The smallest such $\lambda$, denoted $\lambda_{\operatorname*{adj}}$ is
referred to as the \emph{doubling ratio} or \emph{adjacency constant }of $\mu
$.

One may also consider the \emph{dyadic} adjacency constant $\lambda
_{\operatorname*{adj}}^{\operatorname*{dyad}}$ for a measure $\mu$, which is
defined as above except that we that we additionally restrict $I,J$ to belong to a fixed dyadic grid $\mathcal{D}$, the last of which will be clear from context.

Given $\tau\in(0,1)$, we say a doubling measure $\mu$ is $\tau$-flat if its
adjacency constant $\lambda$ satisfies $\lambda,\lambda^{-1}\in(1-\tau
,1+\tau)$. One can make a similar definition in the dyadic setting.
\end{definition}

\add{For a doubling measure $\mu$ on $\mathbb{R}^{n}$, the closer doubling ratio of $\mu$ is to $1$, then the closer $C_{\operatorname{doub}}$ is to $2^n$: more precisely, for every $\epsilon > 0$, there exists a $\delta >0$ such that for all doubling measures $\mu$ on $\mathbb{R}^n$, if $|\lambda_{\operatorname{adj}} (\mu) - 1| < \delta$, then $| C_{\operatorname{doub}}  - 2^n | < \epsilon$. }

One can make similar definitions replacing $\mathbb{R}^{n}$ by an open subset,
and modifying the definitions accordingly.

\subsection{Telescoping identity }

\subsubsection{\add{Working in the plane}}
We begin by discussing the telescoping identity in the plane where matters can
easily be made more explicit. For each square $Q$ in the plane define the
$1$-dimensional projection $\mathbb{E}_{Q}$ by%
\[
\mathbb{E}_{Q}f\equiv\left(  E_{Q}f\right)  \mathbf{1}_{Q}%
\]
where $E_{Q}f\equiv\frac{1}{\left\vert Q\right\vert }\int_{Q}f$ is the average
of $f$ on $Q$. Denote the four dyadic children of a square $Q$ in the plane by
$Q_{\operatorname{NW}},Q_{\operatorname{NE}},Q_{\operatorname{SW}%
},Q_{\operatorname{SE}}$ where $\operatorname{NW}$ stands for the northwest
child, etc\add{\dots} Then define an orthonormal Haar basis $\left\{  h_{Q}%
^{\hor},h_{Q}^{\ver},h_{Q}%
^{\checker}\right\}  $ associated with $Q$ by
\begin{align*}
\sqrt{\left\vert Q\right\vert }h_{Q}^{\hor}  &
\equiv+\mathbf{1}_{Q_{\operatorname{NW}}}-\mathbf{1}_{Q_{\operatorname{NE}}%
}+\mathbf{1}_{Q_{\operatorname{SW}}}-\mathbf{1}_{Q_{\operatorname{SE}}}\ , \quad \sqrt{\left\vert Q\right\vert }h_{Q}^{\ver}  
\equiv-\mathbf{1}_{Q_{\operatorname{NW}}}-\mathbf{1}_{Q_{\operatorname{NE}}%
}+\mathbf{1}_{Q_{\operatorname{SW}}}+\mathbf{1}_{Q_{\operatorname{SE}}}\ ,\\
\sqrt{\left\vert Q\right\vert }h_{Q}^{\checker}  &
\equiv+\mathbf{1}_{Q_{\operatorname{NW}}}-\mathbf{1}_{Q_{\operatorname{NE}}%
}-\mathbf{1}_{Q_{\operatorname{SW}}}+\mathbf{1}_{Q_{\operatorname{SE}}}\ ,
\end{align*}
where we associate the three matrices $\left[
\begin{array}
[c]{cc}%
+ & -\\
+ & -
\end{array}
\right]  ,\left[
\begin{array}
[c]{cc}%
- & -\\
+ & +
\end{array}
\right]  ,\left[
\begin{array}
[c]{cc}%
+ & -\\
- & +
\end{array}
\right]  $ with $h_{Q}^{\hor},h_{Q}%
^{\ver},h_{Q}^{\checker}$, which change sign \emph{horizontally}, \emph{vertically} and in a \emph{checkerboard} pattern, respectively. Thus we also refer to these three matrices as the horizontal matrix, vertical
matrix and checkerboard matrix. Let $\bigtriangleup_{Q}$ denote
Haar projection onto the $3$-dimensional space of functions that are constant
on children of $Q$, and that also have mean zero. Then we have the linear
algebra formula,%
\begin{align}
\bigtriangleup_{Q}f  =   \left\langle f,h_{Q}^{\hor}\right\rangle
h_{Q}^{\hor}+\left\langle f,h_{Q}%
^{\ver}\right\rangle h_{Q}^{\ver%
}+\left\langle f,h_{Q}^{\checker}\right\rangle
h_{Q}^{\checker} =\bigtriangleup_{Q}^{\hor}f+\bigtriangleup
_{Q}^{\ver}f+\bigtriangleup_{Q}%
^{\checker}f, \label{eq:Haar_decomp_pattern}
\end{align}
where $\bigtriangleup_{Q}^{\hor}f$ is the rank one
projection $\left\langle f,h_{Q}^{\hor}\right\rangle
h_{Q}^{\hor}$, etc\dots

Now given two cubes $P$ and $Q$ in $\mathcal{D} (P)$ with $Q\subsetneqq P$,
define%
\[
\left(  Q,P\right]  \equiv\left\{  I\in\mathcal{D} (P) :Q\subsetneqq I\subset
P\right\}
\]
to be the tower of cubes from $Q$ to $P$ that includes $P$ but not $Q$.
Similarly define the towers $\left(  Q,P\right)$, $\left[  Q,P\right]$, $\left[
Q,P\right)  $. Then, for $\left(  Q,P\right]  $, we have the \add{well-known} telescoping
identity,%
\begin{align*}
&  \left(  \mathbb{E}_{Q}f-\mathbb{E}_{P}f\right)  \mathbf{1}_{Q}=\left(
\sum_{I\in\left(  Q,P\right]  }\bigtriangleup_{I}f\right)  \mathbf{1}_{Q}\\
&  =\left(  \sum_{I\in\left(  Q,P\right]  }\left\langle f,h_{I}%
^{\hor}\right\rangle h_{I}^{\hor%
}\right)  \mathbf{1}_{Q}+\left(  \sum_{I\in\left(  Q,P\right]  }\left\langle
f,h_{I}^{\ver}\right\rangle h_{I}^{\ver%
}\right)  \mathbf{1}_{Q} +\left(  \sum_{I\in\left(  Q,P\right]  }\left\langle f,h_{I}%
^{\checker}\right\rangle h_{I}^{\checker%
}\right)  \mathbf{1}_{Q}\\
&  =\left(  \sum_{I\in\left(  Q,P\right]  }\bigtriangleup_{I}%
^{\hor}f\right)  \mathbf{1}_{Q}+\left(  \sum_{I\in\left(
Q,P\right]  }\bigtriangleup_{I}^{\ver}f\right)
\mathbf{1}_{Q}+\left(  \sum_{I\in\left(  Q,P\right]  }\bigtriangleup
_{I}^{\checker}f\right)  \mathbf{1}_{Q}\ .
\end{align*}

\subsubsection{\add{In higher dimension}}
Turning now to dimension $n$, we note that a similar telescoping identity
holds in $\mathbb{R}^{n}$. In particular, given a cube $Q\subset\mathbb{R}%
^{n}$, if we let $\bigtriangleup_{Q}$ denote the Haar projection onto the
space of functions constant on the dyadic children of $Q$ with mean $0$, then
\[
\bigtriangleup_{Q}f=\sum\limits_{j=1}^{d(n)}\langle f,h_{Q}^{j}\rangle
h_{Q}^{j}\equiv\sum\limits_{j=1}^{d(n)}\bigtriangleup_{Q}^{j}f\,,
\]
where $\{h_{Q}^{j}\}_{j=1}^{d(n)}$ is a choice of $L^{2}(Q)$ orthonormal basis
for the range of $\bigtriangleup_{Q}$, and $d(n)=2^{n}-1$ is the dimension of
this space. One of course has an analogue to the telescoping identity above.
In our applications for $n\geq2$, we will be interested in the case that
$h_{Q}^{1}=h_{Q}^{\hor}$, where for $Q=Q_{1}\times
\ldots\times Q_{n}$ we define \add{the horizontal Haar wavelet}
\[
\sqrt{|Q|}h_{Q}^{\hor}(x)\equiv%
\begin{cases}
1 & \text{ if }x\in Q_{-}\\
-1 & \text{ if }x\in Q_{+} \\
0 & \text{ otherwise}
\end{cases}
\,.
\]
We will not care about the choice of $h_{Q}^{2},h_{Q}^{3},\ldots,h_{Q}^{d(n)}$
for each cube $Q$, although we could simply take the orthogonal Haar basis  $\{h_{Q}^{j}\}$ to be the `standard' Haar basis $\left\{  g_{1}%
\otimes...\otimes g_{n}\right\}  $ consisting of all product functions
$g_{1}\left(  x_{1}\right)  \times...\times g_{n}\left(  x_{n}\right)  $ in
which $g_{j}$ is either the Haar function $h_{j}$ on $Q_{j}$, or the
normalized indicator $\frac{1}{\sqrt{\left\vert Q_{j}\right\vert }}%
\mathbf{1}_{Q_{j}}$, and where the constant function on $Q$ is discarded; 
note that 
\begin{equation}\label{eq:Haar_horizontal_defn_dimn}
\frac{1}{\sqrt{\left | Q \right |}} s_{1}^{Q,\hor}=h_{1}\otimes\frac{1}%
{\sqrt{\left\vert Q_{2}\right\vert }}\mathbf{1}_{Q_{2}} \otimes ...\otimes\frac
{1}{\sqrt{\left\vert Q_{n}\right\vert }}\mathbf{1}_{Q_{n}}
\end{equation}
\subsection{\add{Horizontal d}yadic testing}

Given weights $V,U$ on a cube $J$ define
\[
\mathfrak{\gamma}^{\hor}\left(  V,U;J\right)  \equiv\frac
{1}{\left\vert J\right\vert }\sum\limits_{I\in\mathcal{D}(J)}\left\Vert
\bigtriangleup_{I}^{\hor}V\right\Vert _{L^{2} \left ( \mathbb{R}^n \right ) }^{2}%
E_{I}U=\frac{1}{\left\vert J\right\vert }\sum\limits_{I\in\mathcal{D}%
(J)}|\langle V,h_{I}^{\hor}\rangle|^{2}E_{I}U\,.
\]
If $\mathcal{D}$ is the dyadic grid, define the dyadic horizontal testing
constant
\[
\mathfrak{T}^{\hor}\left(  V,U\right)  \equiv\sup
_{J\in\mathcal{D}}\frac{\gamma^{\hor}\left(  V,U;J\right)
}{E_{J}V}.
\]

\begin{remark}
$\mathfrak{T}^{\hor}\left(  V,U\right)  $ is the $L^{2}%
(V)\rightarrow L^{2}(U)$ testing condition\ for the `localized' horizontal
dyadic square function%
\[
S_{J}^{\hor}f\left(  x\right)  \equiv\sqrt{\sum_{\substack{I\in
\mathcal{D}(J):\ \\x\in I}}\frac{\left\Vert \bigtriangleup_{I}%
^{\hor}f\right\Vert ^{2} _{L^2 \left ( \mathbb{R}^n \right ) }}{\left\vert I\right\vert }}=\sqrt
{\sum_{I\in\mathcal{D}(J)}\left\Vert \bigtriangleup_{I}^{\hor%
}f\right\Vert ^{2} _{L^2 \left ( \mathbb{R}^n \right ) }\frac{\mathbf{1}_{I}(x)}{\left\vert I\right\vert }}.
\]
Indeed, we compute%
\begin{align*}
\int_{J}\left\vert S_{J}^{\hor}\left(  \mathbf{1}_{J}V\right)
\left(  x\right)  \right\vert ^{2}U\left(  x\right)  dx  &  =\int_{J}%
\sum_{I\in\mathcal{D}(J)}\left\Vert \bigtriangleup_{I}^{\hor%
}\left(  \mathbf{1}_{J}V\right)  \right\Vert ^{2} _{L^2 \left (\mathbb{R}^n \right ) } U\left(  x\right)
\frac{\mathbf{1}_{I}(x)}{\left\vert I\right\vert }dx
\end{align*}
\[
=\sum_{I\in\mathcal{D}(J)}\left\Vert \bigtriangleup
_{I}^{\hor}\left(  \mathbf{1}_{J}V\right)  \right\Vert
^{2} _{L^2 \left (\mathbb{R}^n \right ) } E_{I}U = \left | J \right | \gamma^{\hor} \left ( V, U , J \right ) \, ,
\]
and so the square of the dyadic testing condition for the localized
horizontal square function is%
\[
\sup_{J\in\mathcal{D}}\frac{\int_{J}\left\vert S_{J}^{\hor%
}\left(  \mathbf{1}_{J}V\right)  \left(  x\right)  \right\vert ^{2}U\left(
x\right)  dx}{\int_{J}V\left(  x\right)  dx}=\sup_{J\in\mathcal{D}}%
\frac{\mathfrak{\gamma}^{\hor}\left(  V,U,J\right)  }{E_{J}V}.
\]
\end{remark}

\subsection{The Bellman construction of the dyadic weights}

\begin{definition}
Given weights $V,U$ on a cube $J$ in $\mathbb{R}^{d}$, we define the dyadic
$A_{2}$ constant relative to $J$ by%
\[
A_{2}^{\operatorname*{dyadic}}(V,U;J)\equiv\sup_{I\in\mathcal{D}(J)}\left(
E_{I}U\right)  \left(  E_{I}V\right)  \,.
\]

\end{definition}

Following the Bellman construction used in \cite{NaVo} gives the following key
result.\footnote{A simpler Bellman proof is provided in \cite{Naz}; one can
also likely obtain the key result by using the disarrangement argument of
\cite{KaTr}.}

\begin{theorem}
\label{Bellman Haar shift} Given a cube $J$ in $\mathbb{R}^{n}$ and arbitrary
constants $\Gamma>0$, $\tau\in(0,1)$, there exist $\tau$-flat weights $V,U$ on
$J$, with $V,U$ constant on all cubes $I\in\mathcal{D}_{m}(J)$ for some $m>0$,
such that
\[
A_{2}^{\operatorname*{dyadic}}(V,U;J)\leq1 \, , \qquad\gamma
^{\hor}\left(  V,U;J\right)  >\Gamma\left(  E_{J}V\right)  \,
.
\]
\add{Furthermore, $U$ and $V$ are in the linear span of the finite set 
\[
 \left \{ \mathbf{1}_J \right \} \bigcup \{h_I ^{\hor} \}_{I \in \mathcal{D}(J), \ell \left (I \right ) \geq 2^{-(m-1)} \ell \left ( J \right )} \, .
\]} In particular when $n=2$, the last conclusion implies
\begin{equation}\label{eq:vanishing_Haar}
\bigtriangleup_{I}^{\ver }U=\bigtriangleup_{I}%
^{\checker}U=0 \, , \quad\bigtriangleup_{I}^{\ver
}V=\bigtriangleup_{I}^{\checker}V=0 \, , \quad I \in
\mathcal{D}(J) \, .
\end{equation}
\end{theorem}

\begin{proof}
The dimension $n=1$ case follows from Nazarov's Bellman argument in
\cite{Naz}. \footnote{See also \cite[Section 3]{NaVo} for a stronger conclusion not used
here, but which requires more difficult Hessian
computations, and also requires an argument to show that their set of
admissible weight pairs $\mathcal{F}_{x}$ is nonempty, the details of which can be found in, e.g., an earlier preprint of this article \add{\cite[Lemma
12]{AlLuSaUr}.}}

\add{For dimension $n\geq 2$, we show matters reduce to the $n=1$ case. We show this for dimension $n=2$, and a similar argument shows the same for dimension $n\geq 3$.} Let $J = J_{1} \times J_{2}$ be a square. So
given parameters $\Gamma$ and $\tau$, suppose our 1-dimensional Theorem gives
us weights $(V_{0},U_{0})$ defined on $J_{1}$. Then define $U$ by
$U(x_{1},x_{2})\equiv\mathbf{1}_{J_{2}}(x_{2})U_{0}(x_{1})$, and similarly for
$V$. Then note that
\[
E_{I}U=E_{I_{1}}U_{0}\,,\quad E_{I}V=E_{I_{1}}V_{0}\,,\quad\text{ for }%
I\in\mathcal{D}(J)\,.
\]
Since $U_{0},V_{0}$ are $\tau$-flat and $A_{2}^{\operatorname{dyadic}}%
(V_{0},U_{0};J_{1})\leq1$, then the above equation shows the same must be true
of $V,U$ on $J$.

Then $2$-dimensional testing is given by
\begin{align*}
\gamma^{\hor}(V,U;J)  &  \approx\sum\limits_{I\in
\mathcal{D}(J)}\frac{|I|}{|J|}(E_{I_{\operatorname{NW}}}%
V+E_{I_{\operatorname{SW}}}V-E_{I_{\operatorname{NE}}}%
V-E_{I_{\operatorname{SE}}}V)^{2}E_{I}U\\
&  = \sum\limits_{k=0}^{\infty}\sum\limits_{I\in\mathcal{D}_{k}(J)}%
2^{-2k}(E_{I_{\operatorname{NW}}}V+E_{I_{\operatorname{SW}}}%
V-E_{I_{\operatorname{NE}}}V-E_{I_{\operatorname{SE}}}V)^{2}E_{I}U\\
&  \approx\sum\limits_{k=0}^{\infty}\sum\limits_{K\in\mathcal{D}_{k}(J_{1}%
)}\sum\limits_{\substack{I\in\mathcal{D}_{k}(J):\\I_{1}=K}}2^{-2k}(E_{K_{-}%
}V_{0}-E_{K_{+}}V_{0})^{2}E_{K}U_{0}\\
&  =\sum\limits_{k=0}^{\infty}\sum\limits_{K\in\mathcal{D}_{k}(J_{1})}%
2^{-k}(E_{K_{-}}V_{0}-E_{K_{+}}V_{0})^{2}E_{K}U_{0}\\
&  =\sum\limits_{k=0}^{\infty}\sum\limits_{K\in\mathcal{D}_{k}(J_{1})}%
\frac{|K|}{|J|}(E_{K_{-}}V_{0}-E_{K_{+}}V_{0})^{2}E_{K}U_{0}\\
&  \approx\gamma^{\hor}(V_{0},U_{0};J_{1})\,,
\end{align*}
which is at least $\Gamma\left(  E_{J_{1}}V_{0}\right)  =\Gamma\left(
E_{J}V\right)  $, which yields the first conclusions after relabeling $\Gamma$.

\add{To see the claim about the span, since $U, V$ are constant on squares in $\mathcal{D}_m (J)$, then $U,V$ are bounded and so are $L^2 (J)$ functions. But the space of $L^2 (J)$ functions which are constant on elements of $\mathcal{D}_m (J)$ has orthonormal basis 
\[
\left \{ \frac{1}{\sqrt{\left | J \right |}} \mathbf{1}_J \right \} \bigcup \left \{ h_I ^{\hor} , h_I ^{\ver}, h_I ^{\checker}\right \}_{I \in \mathcal{D} \left ( J \right ) : \ell (I ) \geq 2^{-(m-1)} \ell (J)} \, .
\] Thus to show the claim about the span, it suffices to show 
\[
\langle h_{I}, U \rangle = \langle h_{I}, V \rangle = 0
\]
for any function $h_I$ that is orthogonal to $h_I ^{\hor}$, has mean $0$, is supported on $I$, and is constant on the dyadic children of $I$. Let $h_I$ be such a function.} Since $h_{I}$ is
piecewise constant on the dyadic children of $I$, we may expand $
\langle U,h_{I}\rangle$ as
\[
\int\limits_{I}U(x)h_{I}(x)dx =E_{I_{\operatorname{NW}}}U\int\limits_{I_{\operatorname{NW}}}%
h_{I}(x)dx+E_{I_{\operatorname{SW}}}U\int\limits_{I_{\operatorname{SW}}}%
h_{I}(x)dx  +E_{I_{\operatorname{NE}}}U\int\limits_{I_{\operatorname{NE}}}%
h_{I}(x)dx+E_{I_{\operatorname{SE}}}U\int\limits_{I_{\operatorname{SE}}}%
h_{I}(x)dx\,.
\]
Substituting averages of $U$ for averages of $U_{0}$, taking $a\equiv
E_{\left(  I_{1}\right)  _{-}}U_{0}$ and $b\equiv E_{\left(  I_{1}\right)
_{+}}U_{0}$ for convenience, we get that this equals
\[
a\int\limits_{I_{-}}h_{I}(x)dx+b\int\limits_{I_{+}}h_{I}(x)dx=\frac{a+b}%
{2}\int\limits_{I}h_{I}(x)dx+\frac{b-a}{2}\left(  \int\limits_{I_{+}}%
h_{I}(x)dx-\int\limits_{I_{-}}h_{I}(x)dx\right)  \,.
\]
Since $h_{I}$ has mean $0$, the first integral on the right
vanishes. Since $\langle h_{I},h_{I}^{\hor}\rangle=0$,
then the last term vanishes too, and thus $\langle U,h_{I}\rangle=0$.
Similarly for $V$.
\end{proof}

We will now adapt the supervisor argument of Nazarov to construct a pair of
doubling weights $\left(  v,u\right)  $, first on a cube in $\mathbb{R}^{n}$
and eventually on the whole space $\mathbb{R}^{n}$, satisfying $A_{2}\left(
v,u\right)  \leq1$ and such that the first Riesz transform $R_{1}$ has
operator norm $\mathfrak{N}_{R_{1}}\left(  v,u\right)  >\Gamma$, while the
other Riesz transforms $R_{j}$, $j\geq2$, have operator norm $\mathfrak{N}%
_{R_{j}}\left(  v,u\right)  \leq1$. Thus the individual Riesz transform
$R_{1}$ is not stable under rotations of doubling weights in the plane. We
will view the supervisor map more simply as a transplantation map, that
readily exploits telescoping properties of projections. \add{To make such conclusions about the norm inequalities, we will compute a testing condition, and if $V$ and $U$ are $\tau$-flat for $\tau$ sufficiently small, then
the classical pivotal condition holds \cite{Gri}, and so we can apply the $T1$
theorem in \cite{SaShUr10} in order to deduce $\mathfrak{N}_{R_{2}}\left(
v,u\right)  \leq1$ from the testing conditions. See Theorem \ref{thm:starting_T1_thm} below for more details.} 

\section{The supervisor and transplantation map}

\label{section:sprvsr_trans}

We again begin our discussion in the plane where matters are more easily
pictured. We will construct our weight pair $\left(  v,u\right)  $ on a square
$Q^{0}\subset\mathbb{R}^{2}$ from the dyadic weight pair $\left(  V,U\right)
$ by adapting the supervisor argument of Nazarov \cite{NaVo} as
follows\footnote{A simpler form of `disarranging' a weight was used in
\cite{Saw1} to provide a counterexample to the conjecture of Muckenhoupt and
Wheeden \cite[page 281]{MuWh} that a one-tailed $A_{p}$ condition was
sufficient for the norm inequality for $M$, but the weights were not
doubling.}. Let $\left\{  k_{t}\right\}  _{t=1}^{\infty}$ be an increasing
sequence of positive integers to be fixed later, and let $\mathcal{D}^{0}$
denote the collection of dyadic subsquares of $Q^{0}$. We denote by
$\mathcal{K}_{t} = \mathcal{K}_{t} (Q^{0})$ the collection of dyadic
subsquares $Q$ of $Q^{0}$ in $\mathcal{D}^{0}$ whose side lengths satisfy
$\ell\left(  Q\right)  =2^{-k_{1}-...-k_{t}}\ell\left(  Q^{0}\right)  $, and
then define
\[
\mathcal{K} = \mathcal{K} ( Q_{0}) = \bigcup\limits_{t \in\mathbb{N}}
\mathcal{K}_{t} (Q_{0}) \,
\]
a subgrid of the dyadic grid $\mathcal{D}^{0}$. Recall we have $\Theta
\equiv\left\{  \operatorname{NW},\operatorname{NE},\operatorname{SW}%
,\operatorname{SE}\right\}  $ the set of four locations of a dyadic square $Q$
within its $\mathcal{D}$-parent $\pi_{\mathcal{D}}Q$.

\subsection{The informal description of the construction}

Here is an informal description of the transplantation argument, that we will
give precisely later on. Given a nonnegative integrable function $U\in
L^{1}\left(  Q^{0}\right)  $, and $t\in\mathbb{N}$, we will define
$u_{t}\left(  x\right)  $ to be a step function on $Q^{0}$ that is constant on
each square in the $t^{th}$ level $\mathcal{K}_{t}$ of $\mathcal{K}$, and
where the constants are among the expected values of $U$ on the squares in the
$t^{th}$ level $\mathcal{D}_{t}^{0}$ of $\mathcal{D}^{0}$, but `scattered'
according to the following plan.

To each square $Q$ in $\mathcal{K}_{t}$, there is associated a unique
descending `$\mathcal{K}$-tower' $\mathbf{T}=\left(  T_{1},...,T_{t}\right)
\in\mathcal{K}^{t}=\mathcal{K}\times...\times\mathcal{K}$ with $T_{t}=Q$,
where the square $T_{\ell}$ is the unique square in $\mathcal{K}_{\ell}$
containing $Q$. To each component square $T_{\ell}$ of $\mathbf{T}$ there is
associated a unique $\theta_{\ell}\in\Theta$, which describes the location of $T_{\ell
}$ within its $\mathcal{D}$-parent $\pi_{\mathcal{D}}T_{\ell}$. We then define $\mathcal{S}\left(  Q\right)  $ to be the square $L$ in $\mathcal{D}_{t} ^0$ which is obtained from $Q^0$ via the following algorithm.
\begin{enumerate}
    \item Set $L=Q^0$.
    \item For $\ell=1, \ldots, t$, replace $L$ by its dyadic child with location $\theta_{\ell}$ within $L$.
    \item Output $L$.
\end{enumerate}   In the terminology of Nazarov \cite{NaVo},
$\mathcal{S}\left(  Q\right)  $ is the \emph{supervisor} of $Q$. We then
`transplant' the expected value $E_{\mathcal{S}\left(  Q\right)  }U$ of $U$ on
the supervisor to the cube $Q$ in $\mathcal{K}_{t}$ that is being supervised.
For example, when $k_{\ell}=1$ for all $\ell$, this construction yields the
identity
\[
u_{t}=\mathbb{E}_{t}U\equiv\sum_{Q\in\mathcal{D}:\ell\left(  Q\right)
=2^{-t}\ell\left(  Q^{0}\right)  }\left(  E_{Q}U\right)  \mathbf{1}_{Q}\ ,
\]
and when the $k_{\ell}^{\prime}$s are bigger than $1$, the values $\frac
{1}{\left\vert Q\right\vert }\int_{Q}U$ are `scattered' throughout $Q^{0}$.
Now we give precise details of the `scattering' construction.

\subsection{The supervisor map}
We define a map
\[
\mathcal{S}: \mathcal{K}_t \to \mathcal{D}_t ^0 \, ,
\]
for every $t \geq 0$. Given a cube $K \in \mathcal{K}_t$, $\mathcal{S} \left ( K \right )$ is called the \emph{supervisor} of $K$.
We define it as follows. Let $K \in \mathcal{K}_t$. If $t=0$, then $K = Q^0$ and so we define $\mathcal{S}(K)$ to be $Q^0$.

If $t \geq 1$, then define $\theta_{\ell} \in \Theta$, $1 \leq \ell \leq t$, to be the unique location for which the  $\mathcal{K}$-parent
\[
 P_{\ell}\equiv \pi_{\mathcal{K}} ^{(t-\ell)} K
\]
satisfies
\[
\left ( \pi_{\mathcal{D}} P_{\ell} \right )_{\theta_{\ell}} = P_{\ell} \, .
\]
Then define 
\[
\mathcal{S} \left (K \right ) \equiv  \left (\ldots \left ( \left ( Q ^0 \right )_{\theta_1} \right )_{\theta_2} \ldots \right )_{\theta_t} \, ,
\]
using the notation introduced at the beginning of Section
\ref{section:dyadic_prelims}.

Note that the supervisor map $\mathcal{S}$ is many-to-one, indeed
$Q \in \mathcal{D}_t ^0$ has $C_{t, k_{1}, \ldots, k_{t}}$
preimages under $\mathcal{S}$. Furthermore we note that $\mathcal{S}
(\pi_{\mathcal{K}} Q) = \pi_{\mathcal{D}} \mathcal{S} (Q)$, i.e., $\pi$ and
$\mathcal{S}$ commute.
\subsection{The formal construction in the plane}

Let $U\in L^{1}\left(  Q^{0}\right)  $ be a nonnegative integrable function,
and let $t\in\mathbb{N}$. We construct $u_t$ by `transplanting' the expected value $E_{\mathcal{S}\left(  Q\right)  }U$
of $U$ on the supervisor $\mathcal{S}\left(  Q\right) \in \mathcal{D}_t ^0 $ to the cube $Q \in \mathcal{K}_{t}$ that is being supervised. Here are the precise formulas
written out using the parent grid $\mathcal{P}$, where for convenience we will
use superscripts to track the level of a square in the grid $\mathcal{D}$:
\begin{align*}
u_{0}\left(  x\right)   &  =\left(  E_{Q^{0}}U\right)  \mathbf{1}_{Q^{0}%
}\left(  x\right)  ,
\end{align*}
and for $t\geq1$,%
\begin{align*}
u_{t}\left(  x\right)   &  = \sum\limits_{Q \in \mathcal{K}_t }  \left(
E_{\mathcal{S} \left ( Q \right )}U\right)  \mathbf{1}_{Q}\left(  x\right)  .
\end{align*}
The weights $u_{t}$ are nonnegative on $Q^{0}$ since
$u_{t}$ is constant on each square $Q$ in $\mathcal{K}_{t}$, and the value of
this constant is the expectation $E_{\mathcal{S}\left(  Q\right)  }U$ of $U$
on the supervisor $\mathcal{S}\left(  Q\right)  $, which is of course nonnegative.  We
also note the following useful fact: $|u_{t}|$ is bounded by a constant
independent of the choice of $\{k_{t}\}_{t\geq0}$, namely $\Vert u_{t}%
\Vert_{L^{\infty}}\leq\Vert U\Vert_{L^{\infty}}$ since the only values $u_{t}$
can take on are precisely the expectations of $U$ over supervising cubes $Q$.

 Recall the Haar projection $\bigtriangleup_{Q}$
associated with $Q$ satisfies
\begin{align}
\label{Haar_split}\bigtriangleup_{Q} f \equiv\left(  \sum_{Q^{\prime}%
\in\mathfrak{C} \left ( Q \right ) }\mathbb{E}_{Q^{\prime} } f \right)
-\mathbb{E}_{Q}f = \left(  \sum_{Q^{\prime}\in\mathfrak{C}\left(  Q\right)  } \left(  E_{Q^{\prime}}f \right)  \mathbf{1}_{Q^{\prime}%
}\right)  - \left(  E_{Q}f \right)  \mathbf{1}_{Q} .
\end{align}

Given cubes $Q, P$, let $\phi_{P \to Q}$ denote the unique translation and
dilation that takes $P$ to $Q$, and define
\[
h_{Q}^{\hor}\left[  P\right]  (x) \equiv h_{Q}%
^{\hor} \left(  \phi_{P \to Q} (x) \right)  \, .
\]
Note that this function does \emph{not} have $L^{2} (P)$ norm equal to $1$. We
can also make the same definition for $h_{Q}^{\ver}\left[
P\right]  , h_{Q}^{\checker}\left[  P\right]  $. Finally,
define
\begin{align*}
\bigtriangleup_{Q}\left[  P\right]  f (x)    \equiv\left(  \bigtriangleup_{Q}
f \right)  (\phi_{P \to Q} (x)) &= \langle f, h_{Q} ^{\hor} \rangle h_{Q}
^{\hor}[P] (x) +\langle f, h_{Q}
^{\ver} \rangle h_{Q} ^{\ver}[P] (x) + \langle f, h_{Q} ^{\checker} \rangle h_{Q}
^{\checker}[P] (x)\,\\
&  \equiv\bigtriangleup_{Q} ^{\hor}\left[  P\right]  f
(x)+ \bigtriangleup_{Q} ^{\ver}\left[  P\right]  f (x)+
\bigtriangleup_{Q} ^{\checker}\left[  P\right]  f (x).
\end{align*}
Then using (\ref{Haar_split}) for $t \geq1$, the first order differences of
the weights $u_{t}$ are given by%
\[
u_{t+1}\left(  x\right)  -u_{t}\left(  x\right)  =\sum_{Q \in \mathcal{K}_t }\left\{  \left(  \sum_{P \in \mathfrak{C}^{(k_{t+1}-1)} \left ( Q \right )} \sum\limits_{Q' \in \mathfrak{C} \left ( P \right )}\left(  E_{\mathcal{S} \left (Q '  \right ) }U\right)  \mathbf{1}_{Q'}\left(  x\right)  \right)  -\left(  E_{\mathcal{S}\left ( Q \right )} U\right)  \mathbf{1}_{Q}\left(  x\right)  \right\} 
\]
\[= \sum_{Q \in \mathcal{K}_t }\left\{    \sum_{P \in \mathfrak{C}^{(k_{t+1}-1)} \left ( Q \right )} \sum\limits_{Q' \in \mathfrak{C} \left ( P \right )}\left(  E_{\mathcal{S} \left (Q '  \right ) }U      -  E_{\mathcal{S}\left ( Q \right )} U\right)  \mathbf{1}_{Q'}\left(  x\right)  \right\}  = \sum_{Q \in \mathcal{K}_t }\left\{    \sum_{P \in \mathfrak{C}^{(k_{t+1}-1)} \left ( Q \right )} \bigtriangleup_{\mathcal{S}\left ( Q\right )} [P] U \left(  x\right)  \right\} \, .
\]

Let $\mathcal{B}$ denote a set indexing our choice of Haar basis: since we are
working in dimension $2$, we take
\[
\mathcal{B} \equiv\{ \hor, \ver,
\checker \} \, .
\]
For a square $Q$ and an integer $M\in\mathbb{N}$, we define three alternating
functions, one for each $\operatorname{pattern} \in\mathcal{B}$:%
\begin{align}\label{eq:pattern_s_defn}
s_{M}^{Q,\operatorname{pattern}}\left(  x\right)   &  =\sum_{Q^{\prime}%
\in\mathfrak{C}^{\left(  M-1\right)  }\left(  Q\right)  }\sqrt{\left\vert
Q^{\prime}\right\vert }h_{Q^{\prime}}^{\operatorname{pattern}} \, ,
\qquad\operatorname{pattern} \in\mathcal{B} \, .
\end{align}
Note that each of these three alternating functions is a constant $\pm1$ on
grandchildren $P^{\prime}\in\mathfrak{C}^{\left(  M\right)  }\left(  Q\right)
$ of depth $M$, and when restricted to a grandchild $Q^{\prime}\in
\mathfrak{C}^{\left(  M-1\right)  }\left(  Q\right)  $, each alternating
function $s_{k}^{Q,\hor}$, $s_{k}^{Q,\ver}$ and $s_{k}^{Q,\checker}$ has the arrangement of $\pm1$ given respectively by $\left[
\begin{array}
[c]{cc}%
+ & -\\
+ & -
\end{array}
\right]  ,\left[
\begin{array}
[c]{cc}%
- & -\\
+ & +
\end{array}
\right]  ,\left[
\begin{array}
[c]{cc}%
+ & -\\
- & +
\end{array}
\right] $. For instance, $s_{k}^{Q,\hor}$ is the
function on $Q$ consisting of $\pm1$ arranged in the following fashion:%
\[
s_{k}^{Q,\hor}\sim\text{the }2^{k}\times2^{k}\text{
matrix }\left[
\begin{array}
[c]{ccccccc}%
+ & - & + & - & \cdots & + & -\\
+ & - & + & - & \cdots & + & -\\
+ & - & + & - & \cdots & + & -\\
+ & - & + & - & \cdots & + & -\\
\vdots & \vdots & \vdots & \vdots &  & \vdots & \vdots\\
+ & - & + & - & \cdots & + & -\\
+ & - & + & - & \cdots & + & -
\end{array}
\right]  \, ,
\]
and similarly
\[
s_{k}^{Q,\ver}\sim\left[
\begin{array}
[c]{ccccccc}%
- & - & - & - & \cdots & - & -\\
+ & + & + & + & \cdots & + & +\\
- & - & - & - & \cdots & - & -\\
+ & + & + & + & \cdots & + & +\\
\vdots & \vdots & \vdots & \vdots &  & \vdots & \vdots\\
- & - & - & - & \cdots & - & -\\
+ & + & + & + & \cdots & + & +
\end{array}
\right]  \text{ and }s_{k}^{Q,\checker}\sim\left[
\begin{array}
[c]{ccccccc}%
+ & - & + & - & \cdots & + & -\\
- & + & - & + & \cdots & - & +\\
+ & - & + & - & \cdots & + & -\\
- & + & - & + & \cdots & - & +\\
\vdots & \vdots & \vdots & \vdots &  & \vdots & \vdots\\
+ & - & + & - & \cdots & + & -\\
- & + & - & + & \cdots & - & +
\end{array}
\right]  .
\]

\begin{remark}
Notice that the matrix for $s_{k}^{Q,\hor}$ is given by
transplanting $2^{2k-2}$ copies of the $2\times2$ matrix $\left[
\begin{array}
[c]{cc}%
+ & -\\
+ & -
\end{array}
\right]  $, which corresponds to the tensor product of a 1-dimensional Haar
function with matrix $\left[
\begin{array}
[c]{cc}%
+ & -
\end{array}
\right]  $, and an indicator function with matrix $\left[
\begin{array}
[c]{c}%
+\\
+
\end{array}
\right]  $. 
\end{remark}

We now write the projections $\bigtriangleup_{Q}U$ as a sum of the horizontal,
vertical and checkerboard components as in (\ref{eq:Haar_decomp_pattern})
to obtain for $t \geq1$,%
\[
u_{t+1} \left ( x \right ) - u_t \left (x \right ) = \sum\limits_{\operatorname{pattern} \in \mathcal{B}} \sum_{Q \in \mathcal{K}_t}\left\{    \sum_{P \in \mathfrak{C}^{(k_{t+1}-1)} \left ( Q \right )} \bigtriangleup_{\mathcal{S}\left ( Q\right )} ^{\operatorname{pattern}} [P] U \left(  x\right)  \right\} 
\]
\[
= \sum\limits_{\operatorname{pattern} \in \mathcal{B}} \sum_{Q \in \mathcal{K}_t }\left\{    \sum_{P \in \mathfrak{C}^{(k_{t+1}-1)} \left ( Q \right )} \left \langle U, h_{\mathcal{S} \left ( Q \right )} ^{\operatorname{pattern}}  \right \rangle h_{\mathcal{S} \left ( Q \right )} ^{\operatorname{pattern}}   \left [ P \right ] \left (x \right )\right\} \,
\]
\begin{equation} \label{eq:differnce_ut}
= \sum\limits_{\operatorname{pattern} \in \mathcal{B}} \sum_{Q \in \mathcal{K}_t} \left \langle U, h_{\mathcal{S} \left ( Q \right )} ^{\operatorname{pattern}}  \right \rangle \frac{1}{\sqrt{\left | \mathcal{S} \left ( Q \right ) \right |}} s_{k_{t+1}} ^{Q, \operatorname{pattern}} 
\end{equation}

\subsection{The construction in dimension $n=1$}
In dimension $n=1$, we can do the above transplantation construction, with the simple substitution 
\[
\mathcal{B} \equiv \left \{\hor  \right \} \, .
\]
Then the transplantation construction reduces to the
`supervisor and alternating function' construction in Nazarov and Volberg
\cite{NaVo}. Since there is only one choice of pattern $\mathcal{B}$ in 1 dimension, or alternatively only one choice of Haar wavelet basis in 1 dimension, $\{\pm
h^{Q,\hor}\}$, we will use the simplified notation
\begin{equation}\label{eq:defn_sk_1d}
s_{k}^{Q}\equiv s_{k}^{Q,\hor}  
\end{equation}
in dimension $1$, where $s_{k}^{Q,\hor}$ is defined as in  (\ref{eq:pattern_s_defn}).
\subsection{The construction in higher dimensions} Turning now to general dimension $n$, we may define
\[
s_{k}^{Q,\hor}(x)=s_{k}^{Q_{1}}(x_{1})\mathbf{1}%
_{Q_{2}\times\ldots\times Q_{n}}(x_{2},\ldots,x_{n})\,,
\]
where $s_k ^Q$ is the $1$-dimensional alternating function as in (\ref{eq:defn_sk_1d}). Again, the horizontal direction indicates the direction of sign change. All of the
calculations above extend to dimension $n$ using $s_{1}%
^{Q,\hor}$ as part of an otherwise arbitrarily chosen
basis of Haar functions for the cube $Q=Q_{1}\times...\times Q_{n}$. Again, we could consider the `standard' Haar basis $\left\{  g_{1}%
\otimes...\otimes g_{n}\right\}  $ consisting of all product functions
$g_{1}\left(  x_{1}\right)  \times...\times g_{n}\left(  x_{n}\right)  $ in
which $g_{j}$ is either the Haar function $h_{j}$ on $Q_{j}$, or the
normalized indicator $\frac{1}{\sqrt{\left\vert Q_{j}\right\vert }}%
\mathbf{1}_{Q_{j}}$, and where the constant function on $Q$ is discarded; we recall the definition of the horizontal Haar wavelets \eqref{eq:Haar_horizontal_defn_dimn}.

\section{Weak Convergence Properties of the Riesz
transforms\label{section:action_of_Riesz}}

We let $H$ denote the Hilbert transform on $\mathbb{R}$, i.e.,
\[
Hf(x)\equiv\frac{1}{\pi}\operatorname{p.v.}\int_{\mathbb{R}}\frac{f(x-y)}{y}dy,
\]
and we let $R_{j}$ denote the $j$th individual Riesz transform on
$\mathbb{R}^{n}$, i.e.,
\begin{equation}
R_{j}f(x)\equiv c_{n}\operatorname{p.v.}\int_{\mathbb{R}^{n}}\frac{y_{j}%
}{\left\vert y\right\vert ^{n+1}}f(x-y)dy,\text{ \ \ \ \ where }c_{n}%
=\frac{\Gamma\left(  \frac{n+1}{2}\right)  }{\pi^{\frac{n+1}{2}}}. \label{c_n}%
\end{equation}
Note that with these choices of constants, the symbols of the operators $H$
and $R_{j}$ are $-i\operatorname{sgn}\xi$ and $-i\frac{\xi_{j}}{\left\vert
\xi\right\vert }$ respectively. In what follows, all singular integrals are
understood to be taken in the sense of principal values, even when we do not
explicitly write $\operatorname{p.v.}$ in front of the integral. If we apply the
Riesz transform $R_{j}$ in the plane to the difference $u_{t+1}-u_{t}$ in (\ref{eq:differnce_ut}) we obtain%
\begin{align*}
R_{j}\left(  u_{t+1}-u_{t}\right)   &  = \sum\limits_{\operatorname{pattern}
\in\mathcal{B}} \sum_{Q\in\mathcal{K}_{t}}\left\langle U,h_{\mathcal{S}\left(
Q\right)  }^{\operatorname{pattern}}\right\rangle \frac{1}{\sqrt{\left\vert
\mathcal{S}\left(  Q\right)  \right\vert }}R_{j}s_{k_{t+1}}%
^{Q,\operatorname{pattern}}%
\end{align*}
and in particular, if $\bigtriangleup_{P}^{\ver}U$,
$\bigtriangleup_{P}^{\ver}V$, $\bigtriangleup
_{P}^{\checker}U$ and $\bigtriangleup_{P}%
^{\checker}V$ vanish for all $P$, then we have both%
\begin{align}
R_{j}\left(  u_{t+1}-u_{t}\right)   &  =\sum_{Q\in\mathcal{K}_{t}}\left\langle
U,h_{\mathcal{S}\left(  Q\right)  }^{\hor}\right\rangle
\frac{1}{\sqrt{\left\vert \mathcal{S}\left(  Q\right)  \right\vert }}%
R_{j}s_{k_{t+1}}^{Q,\hor}, \label{eq:Riesz_diff_u}\\
R_{j}\left(  v_{t+1}-v_{t}\right)   &  =\sum_{Q\in\mathcal{K}_{t}}\left\langle
V,h_{\mathcal{S}\left(  Q\right)  }^{\hor}\right\rangle
\frac{1}{\sqrt{\left\vert \mathcal{S}\left(  Q\right)  \right\vert }}%
R_{j}s_{k_{t+1}}^{Q,\hor} \label{eq:Riesz_diff_v}.
\end{align}
In Section \ref{section:full_proof}, we will wish to establish three key testing estimates: for an arbitrarily large
$\Gamma$,

\begin{enumerate}
\item $\sup_{Q}\frac{1}{\left\vert Q\right\vert _{v}}\int_{Q}\left\vert
R_{1}\mathbf{1}_{Q}v\right\vert ^{2}u\geq\Gamma$,

\item $\sup_{Q}\frac{1}{\left\vert Q\right\vert _{v}}\int_{Q}\left\vert
R_{2}\mathbf{1}_{Q}v\right\vert ^{2}u\leq1$,

\item $\sup_{Q}\frac{1}{\left\vert Q\right\vert _{u}}\int_{Q}\left\vert
R_{2}\mathbf{1}_{Q}u\right\vert ^{2}v\leq1$.
\end{enumerate}

As in \cite{NaVo}, the first estimate is accomplished by inductively choosing the rapidly
increasing sequence $\left\{  k_{t}\right\}  _{t=1}^{m}$ of positive integers
so that at each stage of the construction labelled by $t$, the discrepancy
\[
\int\left\vert R_{1}\left(  v_{t+1}\right)  \right\vert ^{2}u_{t+1}%
-\int\left\vert R_{1}\left(  v_{t}\right)  \right\vert ^{2}u_{t}
\]
looks like
$\sum_{\ell\left(  I\right)  =2^{t}}\left\Vert \bigtriangleup_{I}%
^{\hor}V\right\Vert ^{2} E_{I}U$, whose sum over $t$
exceeds $\Gamma$. As suggested by (\ref{eq:Riesz_diff_u}) and (\ref{eq:Riesz_diff_v}), it turns out one must then understand the convergence properties of 
\[
R_j s_{k_{t+1}} ^{Q, \hor}
\]
for $j=1,2$, which we do in this section. For $j=2$, we show this converges to $0$ strongly from an application of the alternating series test to exploit the cancellation in $s_k ^{Q, \hor}$. But for $j=1$, the issue is more subtle. In dimension $1$, Nazarov proved $H s_k ^{I} \to 0$ weakly, and other subtle weak convergence properties using holomorphic function theory on the unit disk.
To extend these considerations to higher dimensions, we have not managed to escape the need for holomorphic function theory, so instead we reduce the study of $ R_1 s_{k_{t+1}} ^{Q, \hor}$ to $H s_{k_{t+1}} ^{Q}$ using the alternating series test to exploit cancellation in the $s_k ^{Q, \hor}$ functions, from which point we can then use Nazarov's techniques. But a
considerable amount of preparation is needed to prove these convergence properties: we begin with a
discussion of the notion of weak convergence, which we use in connection with
the alternating functions introduced in Section \ref{section:sprvsr_trans}.

\add{Given $1<p <\infty$, recall $f_{i}\rightarrow0$ weakly in $L^{p}\left(
\mathbb{R}^{n}\right)  $ if for all functions $b\in L^{p^{\prime}}\left(  \mathbb{R}^{n}\right)  $, we have
\begin{equation}
\lim_{i\rightarrow\infty}\int_{\mathbb{R}^{n}}f_{i}\left(  x\right)  b\left(
x\right)  dx=0 \, . \label{weak limit}
\end{equation}
Bounded operators on $L^p (\mathbb{R}^n)$ send weakly convergent sequences to weakly convergent sequences. If $\left\{  f_{i}\right\}  $ is uniformly
bounded in $L^{p}\left(  \mathbb{R}^{n}\right)  $ and $X$ is a dense subset of  $L^{p^{\prime}}\left(
\mathbb{R}^{n}\right)  $, then $f_i \to 0$ weakly if and only if  (\ref{weak limit}) holds
for all $b \in X$. We will apply this last result when $X$ equals $L^{\infty}\left(  \mathbb{R}^{n}\right)  \cap
L^{p^{\prime}}\left(  \mathbb{R}^{n}\right)  $, or when $X$ is the space of compactly supported functions on $\mathbb{R}^{n}$ which are constant on dyadic cubes of fixed size.}

We now turn to some lemmas in dimension $n=1$ that we will use for
establishing the three key testing estimates listed above.

\subsection{Weak convergence properties of the Hilbert transform}

In Nazarov's supervisor argument in \cite{NaVo}, the weak limits appearing in
Lemma \ref{Nazarov_weak_convergence} below, for the alternating functions
$s_{k}^{I}$, were proved using holomorphic function theory. \add{While the results of this subsection already appear in \cite{Naz, NaVo}, to keep this paper self-contained, we provide the proofs here along with details omitted in previous articles.} 

If $f \in L^{p}(\mathbb{R)}$, then for every $z\in\mathbb{R}_{+}^{2}$
\ define \add{the} Poisson extension $\mathbf{P} f(z)$ of $f$ by
\begin{align}
&  \mathbb{P}f_{k}\left(  z\right)  =\mathbb{P}f_{k}\left(  x+iy\right)
\equiv\int_{\mathbb{R}}f_{k}\left(  t\right)  P_{x+iy}\left(  t\right)  dt, \label{eq:defn_Poisson_extension}\\
&  \text{where }P_{x+iy}\left(  t\right)  \equiv\frac{y}{\left(  x-t\right)
^{2}+y^{2}}\text{ is the Poisson kernel.} \nonumber
\end{align}
A key observation in \cite{Naz, NaVo} was the following lemma.

\begin{lemma}
\label{Nazarov_H2}Given $p\in(1,\infty)$, let
$\left\{  f_{k}\right\}  _{k}$ in $L^{p}(\mathbb{R)}$ be a bounded sequence. Then $f_{k}%
\rightarrow0$ weakly in $L^{p}\left(  \mathbb{R}\right)  $ if and only if
$\lim_{k\rightarrow\infty}\mathbb{P}f_{k}(z)=0$ for all $z\in\mathbb{R}%
_{+}^{2}$.
\end{lemma}

\begin{proof}
If $f_{k}\rightarrow0$ weakly in $L^{p}\left(  \mathbb{R}\right)  $, then it is immediate that $\lim_{k\rightarrow\infty}\mathbb{P}%
f_{k}(z)=0$ for all $z\in\mathbb{R}_{+}^{2}$.

If on the other hand $\lim_{k\rightarrow\infty}\mathbb{P}f_{k}(z)=0$ for all $z\in\mathbb{R}%
_{+}^{2}$, then because finite linear combinations of Poisson kernels are
dense\footnote{\emph{Hint}: Consider the unit circle $\mathbb{T} = [0,2 \pi)$.
Let $f\in C\left(  \mathbb{T} \right)  $ and $\varepsilon>0$. For $r<1$
sufficiently close to $1$, and for $n$ sufficiently large depending on $r$, we
have%
\[
\left\vert P_{r}\ast f\left(  x\right)  -\sum_{k=0}^{n-1}\left(  \int_{2
\pi\frac{k}{n}}^{2 \pi\frac{k+1}{n}}f\right)  P_{r}\left(  x-\frac{2\pi k}%
{n}\right)  \right\vert \leq\varepsilon.
\]
} in the dual space $L^{p^{\prime}}\left(  \mathbb{R}\right)  $, and the norms
$\left\Vert f_{k}\right\Vert _{L^{p}\left(  \mathbb{R}\right)  }$ are
uniformly bounded in $L^{p}\left(  \mathbb{R}\right)  $, we get $f_{k}%
\rightarrow0$ weakly in $L^{p}\left(  \mathbb{R}\right)  $.
\end{proof}

\add{In what follows, given $1 \leq p <\infty$, let $H^p (\mathbb{C}^+)$ denotes the functions $f$ on $\mathbb{R}$ which are the nontangential boundary values of an analytic function  on the upper-half plane
\[
\mathbb{C}^+ := \{ (x,y) \in \mathbb{R}^2 ~:~ y > 0\} \, ,
\]
which we call $f$, 
such that 
\[
\sup\limits_{y > 0 } \left (\int\limits_{-\infty} ^{\infty} |f(x+i y) |^p dx \right )^{\frac 1 p} < \infty \, . 
\]
If $1<p<\infty$, and if $f \in L^p (\mathbb{R})$ is real-valued, then
\[
f + i H f \in H^p (\mathbb{C} ^+) \, .
\]
}

\begin{lemma}
[{\cite[Section 4]{NaVo}}]\label{Nazarov_weak_convergence} Suppose
$p\in(1,\infty)$. With $s_{k}^{I}$ as above, we have
\[
s_{k}^{I}\rightarrow0\text{\ },\ Hs_{k}^{I}\rightarrow0\text{\ },\ s_{k}%
^{I}Hs_{k}^{I}\rightarrow0\,,\ s_{k}^{I}\left(  Hs_{k}^{I}\right)
^{2}\rightarrow0\,,\ \left(  Hs_{k}^{I}\right)  ^{2}\rightarrow\mathbf{1}%
_{I}\ ,
\]
weakly in $L^{p}\left(  \mathbb{R}\right)  $ as $k\rightarrow\infty$. More
generally for nonnegative $a,b$ not both zero, there exist positive constants
$c_{a,b}$, with $c_{0,2}=1$, such that%
\[
\left(  s_{k}^{I}\right)  ^{a}\left(  Hs_{k}^{I}\right)  ^{b}\rightarrow
\left\{
\begin{array}
[c]{ccc}%
0 & if & a\text{ or }b\text{ is odd}\\
c_{a,b}\mathbf{1}_{I} & if & a\text{ and }b\text{ are even}%
\end{array}
\right.  \text{ weakly in }L^{p}\left(  \mathbb{R}\right)  \text{ as
}k\rightarrow\infty.
\]

\end{lemma}

\begin{proof}
Since $\lim_{k\rightarrow\infty}\int_{\mathbb{R}}s_{k}^{I}\left(  t\right)
g\left(  t\right)  dt=0$ for all dyadic step functions $g$ on $\mathbb{R}$, and
since finite linear combinations of dyadic step functions are dense in $L^{p}\left(
\mathbb{R}\right)  $, we conclude that $s_{k}^{I}\rightarrow0$ weakly in
$L^{p}\left(  \mathbb{R}\right)  $. Since $H$ is bounded on $L^{p}\left(
\mathbb{R}\right)  $, we also have $Hs_{k}^{I}\rightarrow0$ weakly in
$L^{p}\left(  \mathbb{R}\right)  $. Let $f_{k}^{I}\equiv s_{k}^{I}+iHs_{k}%
^{I}\in H^{p}\left(  \mathbb{C}^+\right)  $. By an application of Lemma
\ref{Nazarov_H2} using $f_{k}^{I}\rightarrow0$ weakly in $L^{p}\left(
\mathbb{R}\right)  $, followed by the fact that $\left(  \mathbb{P}f_{k}%
^{I}\right)  ^{2}$ is holomorphic and must be the Poisson extension of $\left ( f_k ^I \right )^2$ since they share the same boundary values, and then finally writing $\left(  f_{k}%
^{I}\right)  ^{2}$ in terms of its real and imaginary parts, we get
\[
0=\left[  \lim_{k\rightarrow\infty}\mathbb{P}f_{k}^{I}\left(  z\right)
\right]  ^{2}=\lim_{k\rightarrow\infty}\mathbb{P}\left[  \left(  f_{k}%
^{I}\right)  ^{2}\right]  \left(  z\right)  =\lim_{k\rightarrow\infty
}\mathbb{P}\left[  \left(  s_{k}^{I}\right)  ^{2}-\left(  Hs_{k}^{I}\right)
^{2}+i2s_{k}^{I}Hs_{k}^{I}\right]  \left(  z\right)
\]
for all $z\in\mathbb{C}^{+}$. By Lemma \ref{Nazarov_H2} again,%
\begin{align*}
s_{k}^{I}Hs_{k}^{I}  &  \rightarrow0\text{ weakly in }L^{p}\left(
\mathbb{R}\right)  ,\\
\mathbf{1}_{I}-\left(  Hs_{k}^{I}\right)  ^{2}  &  =\left(  s_{k}^{I}\right)
^{2}-\left(  Hs_{k}^{I}\right)  ^{2}\rightarrow0\text{ weakly in }L^{p}\left(
\mathbb{R}\right)  ,
\end{align*}
since $\left(  s_{k}^{I}\right)  ^{2}=\mathbf{1}_{I}$. Similarly, we see that
the real part of $\left(  f_{k}^{I}\right)  ^{3}$ $\rightarrow0$ weakly in
$L^{p}\left(  \mathbb{R}\right)  $, i.e.%
\[
\left(  s_{k}^{I}\right)  ^{3}-3\left(  s_{k}^{I}\right)  \left(  Hs_{k}%
^{I}\right)  ^{2}\rightarrow0\text{ weakly in }L^{p}\left(  \mathbb{R}\right)
,
\]
which gives $s_{k}^{I}\left(  Hs_{k}^{I}\right)  ^{2}\rightarrow0$ weakly in
$L^{p}\left(  \mathbb{R}\right)  $ since $\left(  s_{k}^{I}\right)
^{2}=\mathbf{1}_{I}$ and $s_{k}^{I}\rightarrow0$ weakly in $L^{p}\left(
\mathbb{R}\right)  $.

The more general statement involving powers $a$ and $b$ follows similar
arguments.
\end{proof}

To carry out Nazarov's supervisor argument in \cite{NaVo}, one also needs to
understand the weak convergence of mixed terms $s_{k}^{I} \left(  Hs_{k}^{J}
\right)  \left(  Hs_{k}^{K} \right)  $, where $I,J,K$ are dyadic intervals of
same side length. We will often make use of the trivial observation that if
$I_{1},I_{2},...,I_{N}$ are pairwise disjoint sets, and functions
$a_{k}^{I_{j}}$ are supported on $I_{j}$, then $\sum_{j=1}^{N}a_{k}^{I_{j}%
}\rightarrow0$ weakly in $L^{p}\left(  \mathbb{R}\right)$ as $k \to \infty$ if and only if
$a_{k}^{I_{j}}\rightarrow0$ weakly in $L^{p}\left(  \mathbb{R}\right)  $ for
each $j=1,2,...,N$.

\begin{lemma}
\label{lem:Nazarov_convergence_mixed} Suppose $p\in(1,\infty)$. Let $I,J,K$ be
dyadic intervals all of equal sidelength. Then
\begin{align}\label{eq:1_multi_weak}
&  s_{k}^{I}\left(  Hs_{k}^{J}\right)  \rightarrow0\,\text{weakly in }%
L^{p}\left(  \mathbb{R}\right)  \text{ as }k\rightarrow\infty,\\
&  \left(  Hs_{k}^{I}\right)  \left(  Hs_{k}^{J}\right)  \rightarrow
0\,\text{weakly in }L^{p}\left(  \mathbb{R}\right)  \text{ as }k\rightarrow
\infty\text{ if }I\neq J, \label{eq:2_multi_weak}\\
&  s_{k}^{I}\left(  Hs_{k}^{J}\right)  \left(  Hs_{k}^{K}\right)
\rightarrow0\,\text{weakly in }L^{p}\left(  \mathbb{R}\right)  \text{ as
}k\rightarrow\infty \label{eq:3_multi_weak} \, .
\end{align}
\end{lemma}

\begin{proof}
Let us first show (\ref{eq:1_multi_weak}). If $I=J$, this follows by Lemma
\ref{Nazarov_weak_convergence}, so assume $I\neq J$. Write $f_{k}^{I}\equiv
s_{k}^{I}+iHs_{k}^{I}$, and similarly for $J$. Since $f_{k}^{I}f_{k}^{J}\in
H^{p}\left(  \mathbb{C}^+\right)  $ (because $H$ is bounded on, e.g.,  $L^{2p}\left(
\mathbb{R}\right)  $), the method of proof of Lemma
\ref{Nazarov_weak_convergence} combined with Lemma \ref{Nazarov_H2} implies
that the real and imaginary parts of $f_{k}^{I}f_{k}^{J}$\ go to $0$ weakly in
$L^{p}\left(  \mathbb{R}\right)  $. In particular since $s_{k}^{I}s_{k}^{J}=0$
because of their disjoint support, we get
\begin{align}
-\left(  Hs_{k}^{I}\right)  \left(  Hs_{k}^{J}\right)  \rightarrow
0\,\text{weakly in }L^{p}\left(  \mathbb{R}\right)  \, , \qquad s_{k}%
^{I}Hs_{k}^{J}+s_{k}^{J}Hs_{k}^{I} \rightarrow0\,\text{weakly in }L^{p}\left(
\mathbb{R}\right)  \label{eq:real_imaginary_multi_weak_init}.
\end{align}
Then (\ref{eq:1_multi_weak}) follows from the right identity in (\ref{eq:real_imaginary_multi_weak_init}): 
since $I,J$ are disjoint, it follows that $s_{k}^{I}Hs_{k}^{J}\rightarrow0$
weakly in $L^{p}\left(  \mathbb{R}\right)  $ and $s_{k}^{J}Hs_{k}%
^{I}\rightarrow0$ weakly in $L^{p}\left(  \mathbb{R}\right)  $. 

As for (\ref{eq:2_multi_weak}), it follows immediately from the identity on the left of (\ref{eq:real_imaginary_multi_weak_init}).

Now let us show (\ref{eq:3_multi_weak}).
Define $f_{k}^{I},f_{k}^{J},f_{k}^{K}$ as above. We will expand $f_{k}%
^{I}f_{k}^{J}f_{k}^{K}$ into its real and imaginary parts, which by Lemma
\ref{Nazarov_H2} go to $0$ weakly in $L^{p}\left(  \mathbb{R}\right)  $. We
will consider various cases involving the dyadic intervals $I,J,K$.

\textbf{Case 1: }$I=J=K$\textbf{.} Then $s_{k}^{I}\left(  Hs_{k}^{J}\right)
\left(  Hs_{k}^{K}\right)  =s_{k}^{I}\left(  Hs_{k}^{I}\right)  ^{2}%
\rightarrow0$ weakly in $L^{p}\left(  \mathbb{R}\right)  $ by Lemma
\ref{Nazarov_weak_convergence}.

\textbf{Case 2: }$I\neq J=K$. Then using that $|s_{k}^{I}|^{2}=\mathbf{1}_{I}%
$, and similarly for $J$, we compute the real part
\[
\operatorname{Re}\left(  f_{k}^{I}f_{k}^{J}f_{k}^{K}\right)
=\operatorname{Re}\left(  f_{k}^{I}\left(  f_{k}^{J}\right)  ^{2}\right)
=\operatorname{Re}\left(  \left(  s_{k}^{I}+iHs_{k}^{I}\right)  \left(
s_{k}^{J}+iHs_{k}^{J}\right)  ^{2}\right)  =-2\left(  Hs_{k}^{I}\right)
s_{k}^{J}\left(  Hs_{k}^{J}\right)  -s_{k}^{I}\left(  Hs_{k}^{J}\right)
^{2}\,.
\]
Since the real part is the sum of two functions with disjoint support, by
Lemma \ref{Nazarov_H2}, $s_{k}^{I}\left(  Hs_{k}^{J}\right)  ^{2}\rightarrow0$
weakly in $L^{p}\left(  \mathbb{R}\right)  $.

\textbf{Case 3: }$I=J\neq K$ or $I=K\neq J$. Assume without loss of generality
that $I=J\neq K$. Using that $s_{k}^{J}s_{k}^{K}=0$ because they have disjoint
supports, we get $f_{k}^{I}f_{k}^{J}f_{k}^{K}$ has real part
\[
-2s_{k}^{J}\left(  Hs_{k}^{J}\right)  \left(  Hs_{k}^{K}\right)  -s_{k}%
^{K}\left(  Hs_{k}^{J}\right)  ^{2}\rightarrow0\,\text{weakly in }L^{p}\left(
\mathbb{R}\right)
\]
by Lemma \ref{Nazarov_H2}. But the two terms have disjoint support $J$ and
$K$, so each goes to $0\,$weakly in $L^{p}\left(  \mathbb{R}\right)  $.

\textbf{Case 4: }$I,J,K$ are pairwise disjoint. We compute the real part of
$f_{k}^{I}f_{k}^{J}f_{k}^{K}$ equals
\[
-s_{k}^{I}\left(  Hs_{k}^{J}\right)  \left(  Hs_{k}^{K}\right)  -\left(
Hs_{k}^{I}\right)  s_{k}^{J}\left(  Hs_{k}^{K}\right)  -s_{k}^{K}\left(
Hs_{k}^{I}\right)  \left(  Hs_{k}^{J}\right)  \rightarrow0\,\text{weakly in
}L^{p}\left(  \mathbb{R}\right)  \,,
\]
by Lemma \ref{Nazarov_H2}. Since the three terms have pairwise disjoint
support, then each individual term goes to $0\,$weakly in $L^{p}\left(
\mathbb{R}\right)  $.
\end{proof}

\subsection{From Hilbert to Riesz}

In analogy with $\left(  Hs_{k}^{I}\right)  ^{2}\rightarrow\mathbf{1}_{I}$
weakly in $L^{2} \left ( \mathbb{R}^n \right )$, we want to show that $\left(  R_{1}s_{k}%
^{P,\hor}\right)  ^{2}\rightarrow c\mathbf{1}_{P}$ weakly
in $L^{2} \left ( \mathbb{R}^n \right )$ for some positive constant $c$, and also that $R_{2}s_{k}%
^{P,\hor}\longrightarrow0$ strongly in $L^{2}$, even
$L^{p}$, as $k\rightarrow\infty$. Using real variable techniques, we will
calculate matters in such a way that our claim for $R_{1}$ reduces to that of
the Hilbert transform $H$, where the holomorphic methods used by Nazarov are
available, while the claim for $R_{2}$ does not need reduction to $H$.

The following notation will also be useful.

\begin{notation}
Given a sequence $\left\{  f_{k}\right\}  _{k=1}^{\infty}$ of functions in
$L^{2}\left(  \mathbb{R}^{n}\right)  $, we write
\[
f_{k}=o_{k\rightarrow\infty}^{\operatorname*{weakly}}\left(  1\right)
\]
if
\[
\lim_{k\rightarrow\infty}\int_{\mathbb{R}^{n}}f_{k}\left(  t\right)  g\left(
t\right)  dt=0\text{ for all }g\in L^{2}\left(  \mathbb{R}^{n}\right)  ,
\]
and we write $f_{k}=o_{k\rightarrow\infty}^{\operatorname*{strongly}}\left(
1\right)  $ if
\[
\lim_{k\rightarrow\infty}\int_{\mathbb{R}^{n}}\left\vert f_{k}\left(
t\right)  \right\vert ^{p}dt=0\text{ for all }p\in(1,\infty).
\]

\end{notation}

We first need an elementary consequence of the alternating series test.

\begin{lemma}
\label{Alt_series_test}If $b$ is a bounded function on $[0,1]$ and there
exists a partition 
\[
\left\{  z_{0}\equiv0<z_{1}<...<z_{N-1}<z_{N}%
\equiv1\right\}  
\]
such that $b$ is monotone, and of one sign, on each
subinterval $(z_{j},z_{j+1})$, then%
\[
\left\vert \int b(x)s_{k}^{[0,1]}(x)dx\right\vert \leq CN2^{-k}\left\Vert
b\right\Vert _{\infty}.
\]

\end{lemma}

\begin{proof}
If $b$ is monotone on $[0,1]$, and if say $b(0)>b(1) \geq0$, then
\begin{align}
\label{alt_series_estimate}\left\vert \int b(x)s_{k}^{[0,1]}(x)dx\right\vert
=\left\vert \sum_{j=1}^{2^{k}}\left(  -1\right)  ^{j}\int_{\frac{j-1}{2^{k}}%
}^{\frac{j}{2^{k}}}b(x)dx\right\vert \leq\int_{0}^{\frac{1}{2^{k}}}\left\vert
b(x)\right\vert dx\leq2^{-k}\left\Vert b\right\Vert _{\infty},
\end{align}
by the alternating series test. More generally, we can apply this argument to
the subinterval $\left[  z_{m-1},z_{m}\right]  $ if the endpoints lie in
$\left\{  j2^{-k}\right\}  _{j=0}^{2^{k}}$, the points of change in sign of
$s_{k}^{[0,1]}$. In the general case, note that if we denote by $\frac
{j_{m-1}}{2^{k}}$ (or $\frac{j_{m}}{2^{k}}$) the leftmost (or rightmost) point
of the form $\frac{j}{2^{k}}$ in $[z_{m-1}, z_{m}]$, then the integrals in
each one of the intervals $[z_{m-1}, \frac{j_{m-1}}{2^{k}}]$ , $[\frac
{j_{m-1}}{2^{k}}, \frac{j_{m}}{2^{k}}]$, and $[\frac{j_{m}}{2^{k}}, z_{m}]$
all satisfy the same bound as (\ref{alt_series_estimate}).
\end{proof}

We will use Lemma \ref{Alt_series_test} to prove the following results, which
encompass the technical details for the estimates in this section. In particular, 
Lemmas \ref{Alt_Series_strong_convergence} and \ref{Alt_Series_strong_convergence_old} below, while technical, will allow for cleaner proofs of the main results Lemma \ref{rep} and Lemma \ref{reduction} of this section.  The reader should keep in mind Lemmas \ref{Alt_Series_strong_convergence} and \ref{Alt_Series_strong_convergence_old} essentially follow from an application of the alternating series test Lemma \ref{Alt_series_test}. 
We first
need to establish some notation.

\begin{definition}
A function $g$ on $\left[  a,b\right]  $ is $M$\emph{-piecewise
monotone} if there is a partition 
\[
\left\{  a=t_{1}<t_{2}<...<t_{M}=b\right\}
 \]
 such that $g$ is monotone and of one sign on each subinterval $\left(
t_{k},t_{k+1}\right)  $, $1\leq k<M$.
\end{definition}

\begin{notation}
For $x\in\mathbb{R}^{n}$ and $P=P_{1}\times...\times P_{n}$ a cube in
$\mathbb{R}^{n}$, we write%
\begin{align}
x  &  =(x_{1},...,x_{n})=(x_{1},x^{\prime})=(x_{1},x_{2},x^{\prime\prime
})=(\widehat{x},x_{n})=(x_{1},\widetilde{x},x_{n}), \label{notation_variables}%
\\
P  &  =P_{1}\times...\times P_{n}=P_{1}\times P^{\prime}=P_{1}\times
P_{2}\times P^{\prime\prime}=\widehat{P}\times P_{n}=P_{1}\times\widetilde
{P}\times P_{n}.\nonumber
\end{align}

\end{notation}

\begin{definition}
The common definition of the $\delta$-halo of a cube $P$ is given by%
\[
H_{\delta}^{P}\equiv\left\{  x\in\mathbb{R}^{n}:\operatorname*{dist}%
(x_{,}\partial P)<\delta\ell\left(  P\right)  \right\}  .
\]
Given a cube $Q\supset P$ we define the $Q$-extended halo of $P$ by%
\[
H_{\delta}^{P;Q}\equiv\left\{  x\in Q:\operatorname*{dist}(x_{j,}\partial
P_{j})<\delta\ell\left(  P\right)  \text{ for some }1\leq j\leq n\right\}  .
\]

\end{definition}

We also write $s_{k}$ in place of $s_{k}^{\left[  -1,1\right]  }$.

\begin{lemma}
\label{Alt_Series_strong_convergence}Let $p\in(1,\infty)$ and $M\geq1$. Let
$P=P_{1}\times P^{\prime}$ be a subcube of a cube $Q=Q_{1}\times Q^{\prime
}\subset\mathbb{R\times R}^{n-1}$. Furthermore suppose that \
\[
F:Q\times P_{1}\times\widetilde{P}\mathbb{\rightarrow R}%
\]
satisfies the following three properties:
\begin{align}
& \label{3 assump}\\
&  \left(  \mathbf{i}\right)  ~~y_{1}\rightarrow F\left(  x,y_{1}%
,\widetilde{y}\right)  \text{ is }M\text{-piecewise monotone}\ \text{for each
}\left(  x,\widetilde{y}\right)  \in\left(  Q\setminus H_{\delta}%
^{P;Q}\right)  \times\widetilde{P}\text{ }\text{ for all }0<\delta<\frac{1}%
{2},\nonumber\\
&  \left(  \mathbf{ii}\right)  ~~\sup_{\left(  x,y_{1},\widetilde{y}\right)
\in\left(  Q\setminus H_{\delta}^{P;Q}\right)  \times P_{1}\times\widetilde
{P}}\left\vert F\left(  x,y_{1},\widetilde{y}\right)  \right\vert \leq
C_{\delta}<\infty\text{ for all }0<\delta<\frac{1}{2},\nonumber\\
&  \left(  \mathbf{iii}\right)  ~~\mathbf{1}_{H_{\delta}^{P;Q}}\left(
x\right)  \int_{\widetilde{P}}\int_{P_{1}}\left\vert F\left(  x,y_{1}%
,\widetilde{y}\right)  \right\vert dy_{1}d\widetilde{y}\rightarrow0\text{
strongly in }L^{p}\left(  Q\right)  \text{ as }\delta\rightarrow0.\nonumber
\end{align}
Then
\[
\int_{\widetilde{P}}\int_{P_{1}}F\left(  x,y_{1},\widetilde{y}\right)
s_{k}\left(  y_{1}\right)  dy_{1}d\widetilde{y}\rightarrow0\text{ strongly in
}L^{p}\left(  Q\right)  \text{ as }k\rightarrow\infty.
\]

\end{lemma}

\begin{proof}
Write%
\[
\int_{\widetilde{P}}\int_{P_{1}}F\left(  x,y_{1},\widetilde{y}\right)
s_{k}\left(  y_{1}\right)  dy_{1}d\widetilde{y}=\left\{  \mathbf{1}%
_{H_{\delta}^{P;Q}}\left(  x\right)  +\mathbf{1}_{Q\setminus H_{\delta}^{P;Q}%
}\left(  x\right)  \right\}  \int_{\widetilde{P}}\int_{P_{1}}F\left(
x,y_{1},\widetilde{y}\right)  s_{k}\left(  y_{1}\right)  dy_{1}d\widetilde
{y}.
\]
For the first term use%
\[
\mathbf{1}_{H_{\delta}^{P;Q}}\left(  x\right)  \left\vert \int_{\widetilde{P}%
}\int_{P_{1}}F\left(  x,y_{1},\widetilde{y}\right)  s_{k}\left(  y_{1}\right)
dy_{1}d\widetilde{y}\right\vert \leq\mathbf{1}_{H_{\delta}^{P;Q}}\left(
x\right)  \int_{\widetilde{P}}\int_{P_{1}}\left\vert F\left(  x,y_{1}%
,\widetilde{y}\right)  \right\vert dy_{1}d\widetilde{y}%
\]
and the third assumption in (\ref{3 assump}).

For the second term we will use the alternating series test Lemma
\ref{Alt_series_test} adapted to the interval $P_{1}$ on the integral
$\int_{P_{1}}F\left(  x,y_{1},\widetilde{y}\right)  s_{k}\left(  y_{1}\right)
dy_{1}$ together with the first and second assumptions in (\ref{3 assump}).
Indeed, by the first assumption and Lemma \ref{Alt_series_test}\ we have that
for $\left(  x,\widetilde{y}\right)  \in\left(  Q\setminus H_{\delta}%
^{P;Q}\right)  \times\widetilde{P}$ , there exists a partition $\left\{
t_{0},t_{1},...,t_{M}\right\}  $ of $P_{1}$ \emph{depending on} $\left(
x,\widetilde{y}\right)  $, but with $M$ \emph{independent} of $\left(
x,\widetilde{y}\right)  $, such that
\begin{align*}
&  \left\vert \int_{P_{1}}F\left(  x,y_{1},\widetilde{y}\right)  s_{k}\left(
y_{1}\right)  dy_{1}\right\vert \leq\sum_{j=0}^{M-1}\left\vert \int_{t_{j}%
}^{t_{j+1}}F\left(  x,y_{1},\widetilde{y}\right)  s_{k}\left(  y_{1}\right)
dy_{1}\right\vert \leq CMC_{\delta}2^{-k},
\end{align*}
where the final inequality follows from the second assumption. Thus away from
the halo we have uniform convergence to zero, and altogether we obtain the
desired conclusion.
\end{proof}

We will also need a version of the previous lemma in which some of the $y$
variables have been integrated out.

\begin{lemma}
\label{Alt_Series_strong_convergence_old}Let $p\in(1,\infty)$ and $M\geq1$.
	Let $P \equiv [-1,1]^n$, which we will sometimes write as $P_{1}\times P^{\prime} $, and assume $P$ is subcube of a cube $Q=Q_{1}\times
Q^{\prime}\subset\mathbb{R\times R}^{n-1}$. Furthermore suppose that \
\[
F:Q\times P_{1}\mathbb{\rightarrow R}%
\]
can be written as
\[
F\left(  x,y_{1}\right)  =\int_{\left[  -1,1\right]  ^{n-2}}F_{y^{\prime
\prime}}\left(  x,y_{1}\right)  dy^{\prime\prime},
\]
where for each fixed $x$, the function $y_{1}\rightarrow F_{y^{\prime\prime}%
}\left(  x,y_{1}\right)  $ doesn't change sign, and where the following three
properties hold:
\begin{align}
& \label{3 assump'}\\
&  \left(  \mathbf{i}\right)  ~~y_{1}\rightarrow F_{y^{\prime\prime}}\left(
x,y_{1}\right)  \text{ is }M\text{-piecewise monotone}\ \text{for each }x\in
Q,y^{\prime\prime}\in\left[  -1,1\right]  ^{n-2}\nonumber\\
&  \left(  \mathbf{ii}\right)  ~~\sup_{\left(  x,y_{1}\right)  \in\left(
Q\setminus H_{\delta}^{P;Q}\right)  \times P_{1}}\left\vert F\left(
x,y_{1}\right)  \right\vert \leq C_{\delta}<\infty\text{ for all }%
0<\delta<\frac{1}{2},\nonumber\\
&  \left(  \mathbf{iii}\right)  ~~\mathbf{1}_{H_{\delta}^{P;Q}}\left(
x\right)  \int_{P_{1}}\left\vert F\left(  x,y_{1}\right)  \right\vert
dy_{1}\rightarrow0\text{ strongly in }L^{p}\left(  Q\right)  \text{ as }%
\delta\rightarrow0.\nonumber
\end{align}
Then
\[
\int_{P_{1}}F\left(  x,y_{1}\right)  s_{k}\left(  y_{1}\right)  dy_{1}%
\rightarrow0\text{ strongly in }L^{p}\left(  Q\right)  \text{ as }%
k\rightarrow\infty.
\]

\end{lemma}

\begin{proof}
This short proof is virtually identical to that of the previous lemma, but we
include it for convenience. Write%
\[
\int_{P_{1}}F\left(  x,y_{1}\right)  s_{k}\left(  y_{1}\right)  dy_{1}%
=\left\{  \mathbf{1}_{H_{\delta}^{P;Q}}\left(  x\right)  +\mathbf{1}%
_{Q\setminus H_{\delta}^{P;Q}}\left(  x\right)  \right\}  \int_{P_{1}}F\left(
x,y_{1}\right)  s_{k}\left(  y_{1}\right)  dy_{1}.
\]
For the first term use%
\[
\mathbf{1}_{H_{\delta}^{P;Q}}\left(  x\right)  \left\vert \int_{P_{1}}F\left(
x,y_{1}\right)  s_{k}\left(  y_{1}\right)  dy_{1}\right\vert \leq
\mathbf{1}_{H_{\delta}^{P;Q}}\left(  x\right)  \int_{P_{1}}\left\vert F\left(
x,y_{1}\right)  \right\vert dy_{1}%
\]
and the third assumption in (\ref{3 assump'}).

For the second term we will use the alternating series test on the integral
$\int_{P_{1}}F_{y^{\prime\prime}}\left(  x,y_{1}\right)  s_{k}\left(
y_{1}\right)  dy_{1}$ together with the first and second assumptions in
(\ref{3 assump'}). Indeed, by the first assumption there exists a partition
$\left\{  t_{0},t_{1},...,t_{M}\right\}  $ of $P_{1}$ \emph{depending on} $x$
and $y^{\prime\prime}$, but with $M$ \emph{independent} of $x$ and
$y^{\prime\prime}$, and then from Lemma \ref{Alt_series_test}\ we have for
$x\in Q\setminus H_{\delta}^{P;Q}$ that%
\begin{align*}
&  \left\vert \int_{P_{1}}F\left(  x,y_{1}\right)  s_{k}\left(  y_{1}\right)
dy_{1}\right\vert =\left\vert \int_{\left[  -1,1\right]  ^{n-2}}\left\{
\int_{P_{1}}F_{y^{\prime\prime}}\left(  x,y_{1}\right)  s_{k}\left(
y_{1}\right)  dy_{1}\right\}  dy^{\prime\prime}\right\vert \\
&  \leq\int_{\left[  -1,1\right]  ^{n-2}}\left\vert \int_{P_{1}}%
F_{y^{\prime\prime}}\left(  x,y_{1}\right)  s_{k}\left(  y_{1}\right)
dy_{1}\right\vert dy^{\prime\prime}\leq\int_{\left[  -1,1\right]  ^{n-2}}%
\sum_{j=1}^{M-1}\left\vert \int_{t_{j}}^{t_{j+1}}F_{y^{\prime\prime}}\left(
x,y_{1}\right)  s_{k}\left(  y_{1}\right)  dy_{1}\right\vert dy^{\prime\prime
}\\
&  \leq\int_{\left[  -1,1\right]  ^{n-2}}\sum_{j=1}^{M-1}\left\{  \int_{t_{j}%
}^{t_{j}+2^{1-k}}+\int_{t_{j+1}-2^{1-k}}^{t_{j+1}}\right\}  \left\vert
F_{y^{\prime\prime}}\left(  x,y_{1}\right)  \right\vert dy_{1}dy^{\prime
\prime}\\
&  \leq\sum_{j=1}^{M-1}\left\{  \int_{t_{j}}^{t_{j}+2^{1-k}}+\int
_{t_{j+1}-2^{1-k}}^{t_{j+1}}\right\}  \left\vert F\left(  x,y_{1}\right)
\right\vert dy_{1}\leq CMC_{\delta}2^{-k},
\end{align*}
where the penultimate inequality follow since $F_{y^{\prime\prime}}$ doesn't
change sign, and the final inequality follows from the second assumption.
\end{proof}

Here is our main reduction of the action of Riesz transforms on $s_{k}%
^{P,\hor}\left(  x\right)  $ to that of the Hilbert
transform $H$ on $s_{k}^{P_{1}}\left(  x_{1}\right)  $.

\begin{lemma}
\label{rep}Given $n\geq1$, a cube $P \subset\mathbb{R}^{n}$ and $p\in
(1,\infty)$, we have for $x=\left(  x_{1},x^{\prime}\right)  \in\mathbb{R}%
^{1}\times\mathbb{R}^{n-1}$,%
\[
R_{1}s_{k}^{P,\hor}(x)=\bn{n}Hs_{k}^{P_{1}}(x_{1}%
)\mathbf{1}_{P^{\prime}}(x^{\prime})+\err_{k}^{P}(x)\,,
\]
where
\[
\bn{n}=c_{n}\an{n}\an{n-1}...\an{1},\ \ \ c_{n}\text{ is as in (\ref{c_n}%
)},\ \ \ \an{n}\equiv\int_{\mathbb{R}}\frac{1}{\left(  1+z^{2}\right)
^{\frac{n+1}{2}}}dz>0,
\]
and the error $\err_{k}^{P}$ tends to $0$ strongly in $L^{p}(\mathbb{R}^{n})$, i.e.
\[
\lim_{k\rightarrow\infty}\left\Vert \err_{k}^{P}\right\Vert _{L^{p}%
(\mathbb{R}^{n})}=0.
\]

\end{lemma}

\begin{proof}
We prove the lemma by induction on the dimension $n\geq1$. Since $\bn{1}=1$,
the case $n=1$ is a tautology (with the understanding that $R_{1} = H$ on
$\mathbb{R}$, note that the constants in front of the integrals match) and so
we now suppose that $n\geq2$, and assume the conclusion of the lemma holds
with $n-1$ in place of $n$.

Let $\varepsilon>0$. For every $M>1$,we have
\[
R_{1}s_{k}^{P,\hor}(x)=\mathbf{1}_{MP}(x)R_{1}%
s_{k}^{P,\hor}(x)+\mathbf{1}_{\mathbb{R}^{n}\setminus
MP}(x)R_{1}s_{k}^{P,\hor}(x).
\]
We note that the second term $\mathbf{1}_{\mathbb{R}^{n}\setminus MP}%
(x)R_{1}s_{k}^{P,\hor}(x)$ goes to $0$ strongly in
$L^{p}(\mathbb{R}^{n})$ as $M\rightarrow\infty$, since%
\[
\int_{\mathbb{R}^{n}\setminus MP}\left\vert R_{1}s_{k}%
^{P,\hor}(x)\right\vert ^{p}dx\leq C\int_{\mathbb{R}%
^{n}\setminus MP}\left(  \int_{P}\frac{1}{\left\vert x-y\right\vert ^{n}%
}dy\right)  ^{p}dx\leq C\int_{\mathbb{R}^{n}\setminus MP}\left(
\frac{\left\vert P\right\vert }{\left\vert \operatorname*{dist}\left(
x,P\right)  \right\vert ^{n}}\right)  ^{p}dx,
\]
which goes to $0$ as $M\rightarrow\infty$ uniformly in $k$; in particular
choose $M$ such that $\int_{\mathbb{R}^{n}\setminus MP}\left\vert R_{1}%
s_{k}^{P,\hor}(x)\right\vert ^{p}dx<\frac{\varepsilon}%
{2}$ for all $k \geq0$. With $Q=MP$, it will suffice to show that
$\lim_{k\rightarrow\infty}\left\Vert \err_{k}^{P}\right\Vert _{L^{p}(Q)}%
<\frac{\varepsilon}{2}$ for $k$ sufficiently large, where $\err_{k}^{P}$ is
implicitly defined as in the statement of the lemma.

Without loss of generality we suppose that $P=\left[  -1,1\right]  ^{n}$.
Recalling that $\widehat{x}=\left(  x_{1},...,x_{n-1}\right)  $, $\widehat
{y}=\left(  y_{1},...,y_{n-1}\right)  $, we write%
\begin{align*}
R_{1}s_{k}^{P,\hor}\left(  x\right)   &  =c_{n}\int
_{-1}^{1}\int_{\left[  -1,1\right]  ^{n-1}}\frac{\left(  x_{1}-y_{1}\right)
s_{k}^{\left[  -1,1\right]  }\left(  y_{1}\right)  }{\left[  \left(
x_{1}-y_{1}\right)  ^{2}+\left\vert x^{\prime}-y^{\prime}\right\vert
^{2}\right]  ^{\frac{n+1}{2}}}dy_{1}...dy_{n-1}dy_{n}\\
&  \equiv\int_{\left[  -1,1\right]  ^{n-1}}\Psi\left(  \widehat{x}%
,x_{n},\widehat{y}\right)  s_{k}^{\left[  -1,1\right]  }\left(  y_{1}\right)
d\widehat{y},
\end{align*}
where, by the change of variables $z = \frac{x_{n} - y_{n}}{|\widehat{x} -
\widehat{y}|}$, we have
\begin{align*}
\Psi\left(  \widehat{x},x_{n},\widehat{y}\right)   &  =c_{n}\int_{-1}^{1}%
\frac{x_{1}-y_{1}}{\left[  \left\vert \widehat{x}-\widehat{y}\right\vert
^{2}+\left\vert x_{n}-y_{n}\right\vert ^{2}\right]  ^{\frac{n+1}{2}}}%
dy_{n}=\frac{c_{n}}{c_{n-1}}K_{1}^{[n-1]}\left(  \widehat{x}-\widehat
{y}\right)  \int_{\frac{x_{n}-1}{\left\vert \widehat{x}-\widehat{y}\right\vert
}}^{\frac{x_{n}+1}{\left\vert \widehat{x}-\widehat{y}\right\vert }}\frac
{1}{\left(  1+z^{2}\right)  ^{\frac{n+1}{2}}}dz \, ,
\end{align*}
and where $K_{1} ^{[m]}$ is the kernel of the first individual Riesz transform
$R_{1}^{[m]}$ in $m$ dimensions. Note
\[
\Phi^{n-1}\left(  \widehat{x},x_{n},\widehat{y}\right)  \equiv\int
_{\frac{x_{n}-1}{\left\vert \widehat{x}-\widehat{y}\right\vert }}^{\frac
{x_{n}+1}{\left\vert \widehat{x}-\widehat{y}\right\vert }}\frac{1}{\left(
1+z^{2}\right)  ^{\frac{n+1}{2}}}dz
\]
is a bounded function of $\left(  \widehat{x},x_{n},\widehat{y}\right)  $
with
\[
\left\Vert \Phi^{n-1}\right\Vert _{\infty}\leq\int_{\mathbb{R}}\frac
{1}{\left(  1+z^{2}\right)  ^{\frac{n+1}{2}}}dz= \an{n}>0.
\]

With $l_{n}(x,\widehat{y})\equiv\frac{x_{n}-1}{\left\vert \widehat{x}%
-\widehat{y}\right\vert }$ and $u_{n}(x,\widehat{y})\equiv\frac{x_{n}%
+1}{\left\vert \widehat{x}-\widehat{y}\right\vert }$ we may further decompose
$\Phi^{n-1}\left(  \widehat{x},x_{n},\widehat{y}\right)  $ as
\begin{align*}
&  \left\{  \int_{l_{n}}^{0}+\int_{0}^{u_{n}}\right\}  \frac{1}{\left(
1+z^{2}\right)  ^{\frac{n+1}{2}}}dz= \left\{  -\operatorname{sgn}(x_{n}%
-1)\int_{0}^{|l_{n}|}+\operatorname{sgn}(x_{n}+1)\int_{0}^{|u_{n}|}\right\}
\frac{1}{\left(  1+z^{2}\right)  ^{\frac{n+1}{2}}}dz\\
=  &  \an{n} \mathbf{1}_{P_{n}}(x_{n})+\left\{  -\operatorname{sgn}%
(x_{n}-1)\left(  \int_{0}^{|l_{n}|}\frac{1}{\left(  1+z^{2}\right)
^{\frac{n+1}{2}}}dz-\frac{\an{n}}{2}\right)  +\operatorname{sgn}(x_{n}%
+1)\left(  \int_{0}^{|u_{n}|}\frac{1}{\left(  1+z^{2}\right)  ^{\frac{n+1}{2}%
}}dz-\frac{\an{n}}{2}\right)  \right\} \\
\equiv &  \an{n}\mathbf{1}_{P_{n}}(x_{n}) - \operatorname{sgn}(x_{n}%
-1)L_{n}^{1}(x,\widehat{y})+ \operatorname{sgn}(x_{n}+1)L_{n}^{2}%
(x,\widehat{y})\,.
\end{align*}

Relating the above computations to $R_{1}^{\left[  n-1\right]  }$ and
$R_{1}^{\left[  n\right]  }$, we obtain
\begin{align*}
&  R_{1}^{\left[  n\right]  }s_{k}^{P,\hor}(x)
=\frac{c_{n}}{c_{n-1}}\an{n}R_{1}^{\left[  n-1\right]  }\left(  s_{k}%
^{[-1,1]}\otimes\mathbf{1}_{P_{2}\times...\times P_{n-1}}\right)  (\widehat
{x})\mathbf{1}_{P_{n}}(x_{n})\\
&  -\frac{c_{n}}{c_{n-1}} \operatorname{sgn}(x_{n}-1)\int_{\left[
-1,1\right]  ^{n-1}}\frac{x_{1}-y_{1}}{\left\vert \widehat{x}-\widehat
{y}\right\vert ^{n}}L_{n}^{1}(x,\widehat{y})s_{k}^{[-1,1]}(y_{1})d\widehat
{y}+\frac{c_{n}}{c_{n-1}} \operatorname{sgn}(x_{n}+1)\int_{\left[
-1,1\right]  ^{n-1}}\frac{x_{1}-y_{1}}{\left\vert \widehat{x}-\widehat
{y}\right\vert ^{n}}L_{n}^{2}(x,\widehat{y})s_{k}^{[-1,1]}(y_{1})d\widehat
{y}\\
&  \equiv\frac{c_{n}}{c_{n-1}}\an{n}R_{1}^{\left[  n-1\right]  }\left(
s_{k}^{[-1,1]}\otimes\mathbf{1}_{\widetilde{P}}\right)  (\widehat
{x})\mathbf{1}_{P_{n}}(x_{n})+\err_{k}^{1}(x)+\err_{k}^{2}(x)\,.
\end{align*}

We now apply our induction hypothesis to the term $R_{1}^{\left[  n-1\right]
}\left(  s_{k}^{[-1,1]}\otimes\mathbf{1}_{\widetilde{P}}\right)  (\widehat
{x})$ to obtain%
\begin{align*}
\frac{c_{n}}{c_{n-1}}\an{n}R_{1}^{\left[  n-1\right]  }\left(  s_{k}%
^{[-1,1]}\otimes\mathbf{1}_{\widetilde{P}}\right)  (\widehat{x})\mathbf{1}%
_{P_{n}}(x_{n}) = \bn{n} Hs_{k}^{P_{1}}(x_{1})\mathbf{1}_{\widetilde{P}}%
(x_{2},...,x_{n})+\frac{c_{n}}{c_{n-1}}\an{n}\err_{k}^{\widehat{P}}(\widehat
{x})\mathbf{1}_{P_{n}}(x_{n})
\end{align*}
where $\err_{k}^{\widehat{P}}(\widehat{x})\mathbf{1}_{P_{n}}(x_{n})$ tends to $0$
strongly in $L^{p}\left(  Q\right)  $ by the induction hypothesis.

So it remains only to show that both $\err_{k}^{1}(x)$ and $\err_{k}^{2}(x)$ go to
$0$ strongly in $L^{p}(Q)$, and by symmetry it suffices to consider just
$\err_{k}^{2}(x)$. We have%
\[
L_{n}^{2}(x,\widehat{y})=\int_{0}^{|u_{n}|}\frac{1}{\left(  1+z^{2}\right)
^{\frac{n+1}{2}}}dz-\frac{\an{n}}{2}=-\int_{|u_{n}|}^{\infty}\frac{1}{\left(
1+z^{2}\right)  ^{\frac{n+1}{2}}}dz\,,
\]
where we recall that $|u_{n}|=\frac{\left\vert x_{n}+1\right\vert }{\left\vert
\widehat{x}-\widehat{y}\right\vert }$.

We now see that it suffices to verify (\textbf{i}), (\textbf{ii}) and
(\textbf{iii}) of Lemma \ref{Alt_Series_strong_convergence} for the cube $Q$
and the function%
\begin{align}
&  F(x,y_{1},\widetilde{y})=\frac{x_{1}-y_{1}}{\left\vert \widehat{x}%
-\widehat{y}\right\vert ^{n}}\int_{\frac{\left\vert x_{n}+1\right\vert
}{\left\vert \widehat{x}-\widehat{y}\right\vert }}^{\infty}\frac{1}{\left(
1+z^{2}\right)  ^{\frac{n+1}{2}}}dz . \label{F}%
\end{align}

We first turn to verifying property (\textbf{i}), and since we only require
upper bounds at this point, we will not keep track of absolute constants. The case $n=2$
turns out to be rather special and easily handled so we dispose of that case
first. We have when $n=2$ that
\[
F(x,y_{1})=\frac{1}{x_{1}-y_{1}}\int
\limits_{|u_{2}(x,y_{1})|}^{\infty}\frac{1}{\left(  1+z^{2}\right)  ^{\frac
{3}{2}}}dz\,.
\]
where $|u_{2}(x,y_{1})|=\frac{\left\vert x_{2}+1\right\vert }{\left\vert
x_{1}-y_{1}\right\vert }$. For any fixed $x$, $|u_{2}(x,y_{1})|$ is monotone
as a function of $|x_{1}-y_{1}|$. We now claim that the function $F(x,y_{1})$
is $M$-piecewise monotone for $M=7$ as a function of $y_{1}$. Since $F(x,
y_{1})$ only changes sign once, to see this it suffices to show that with
$s=|u_{2}(x,y_{1})|$ the function%
\[
H_{\beta}\left(  s\right)  \equiv s\int_{s}^{\infty}\left(  1+t^{2}\right)
^{-\beta}dt,\ \ \ \ \ \text{for }s\in\left(  -\infty,\infty\right)  \, ,
\qquad\beta> \frac{1}{2} \, ,
\]
has $3$ changes in monotonicity on $\left(  -\infty,\infty\right)  $. We
compute
\begin{align*}
H_{\beta}^{\prime\prime}\left(  s\right)  =2\left\{  \left(  \beta-1\right)
s^{2}-1\right\}  \left(  1+s^{2}\right)  ^{-\beta-1}%
\end{align*}
has at most $2$ zeroes in $\left(  -\infty,\infty\right)  $, hence $H_{\beta
}^{\prime}\left(  s\right)  $ has at most $3$ zeroes, which proves our claim.

Now we turn to the more complicated case $n\geq3$. Let $t=x-y$. Then we may
write
\begin{align*}
&  F(x,y_{1},\widetilde{y})=\frac{t_{1}}{\left(  t_{1}^{2}+\left\vert
\widetilde{t}\right\vert ^{2}\right)  ^{\frac{n}{2}}}V_{n}\left(
\frac{\left\vert x_{n}+1\right\vert }{\left(  t_{1}^{2}+\left\vert
\widetilde{t}\right\vert ^{2}\right)  ^{\frac{1}{2}}}\right)  \,,
\end{align*}
where $V_{n}\left(  w\right)  \equiv\int_{w}^{\infty}\frac{1}{\left(
1+z^{2}\right)  ^{\frac{n+1}{2}}}dz$. Note that the antiderivative%
\begin{align}
&  \int\frac{1}{\left(  1+z^{2}\right)  ^{\frac{n+1}{2}}}dz=\int\frac
{1}{\left(  1+\tan^{2}\theta\right)  ^{\frac{n+1}{2}}}d\tan\theta
\label{Theta}\\
&  =\int\frac{\sec^{2}\theta}{\left(  \sec^{2}\theta\right)  ^{\frac{n+1}{2}}%
}d\theta=\int\cos^{n-1}\theta\ d\theta=C_{n}\theta+R_{n}\left(  z,\sqrt
{1+z^{2}}\right)  ,\ \ \ \ \ z=\tan\theta,\nonumber
\end{align}
where $R_{n}$ is a rational function of $z = \tan\theta$ and $\sqrt{1+z^{2}} =
\sec\theta$, and $C_{n}=0$ when $n$ is even. Indeed, one can use the well
known recursion%
\begin{align*}
\int\cos^{m}\theta\ d\theta &  =\frac{1}{m}\cos^{m-1}\theta\ \sin\theta
+\frac{m-1}{m}\int\cos^{m-2}\theta\ d\theta=\frac{1}{m}\frac{1}{\sec^{m}%
\theta}\ \tan\theta+\frac{m-1}{m}\int\cos^{m-2}\theta\ d\theta.
\end{align*}

Then setting
\[
z=\tan\theta=\frac{\left\vert x_{n}+1\right\vert }{\left(  t_{1}%
^{2}+\left\vert \widetilde{t}\right\vert ^{2}\right)  ^{\frac{1}{2}}} \, ,
\qquad E_{0}\equiv\left(  \frac{x_{n}+1}{\left\vert \widetilde{t}\right\vert
}\right)  ^{2}\,,\qquad E_{1}\equiv\frac{\left\vert \widetilde{t}\right\vert
}{|x_{n}+1|^{n}}\, ,
\]
and using (\ref{Theta}), we may write (\ref{F}) as
\begin{align*}
F\left(  x,y_{1},\widetilde{y}\right)   &  =\frac{t_{1}}{\left(  t_{1}%
^{2}+\left\vert \widetilde{t}\right\vert ^{2}\right)  ^{\frac{n}{2}}}\left\{
R_{n}\left(  z,\sqrt{1+z^{2}}\right)  +C_{n}\theta+C\right\} \\
&  =E_{1}\tan^{n-1}\theta\sqrt{E_{0}-\tan^{2}\theta}\left\{  R_{n}\left(
z,\sqrt{1+z^{2}}\right)  +C_{n}\theta+C\right\}  \equiv D_{x,\widetilde{t}%
}\left(  \theta\right)  \, .
\end{align*}

At this point, we employ the convention that $R_{n},T_{n},U_{n}$ are rational
functions which may change line to line, or instance to instance, but their
degree will be bounded a constant depending only on the dimension $n$, where
the degree is the sum of degrees of the numerator and denominator. Similarly,
we will take $M$ to be an integer which may change line to line or instance to
instance, but will only depend on the dimension $n$. We also recall the fact
that the function $R_{n}\left(  z,\sqrt{1+z^{2}}\right)  $ can equal $0$ or
$\infty$ at most $M$ times: indeed, $R_{n}$ is a rational function of $z$ and
$\sqrt{1+z^{2}}$, which is in turn a nontrivial rational function of
$\sin\theta$ and $\cos\theta$, with degree depending only on $n$. Thus the
number of zeros or poles it possesses is at most a constant depending only the
degree, i.e. a constant which only depends $n$.

Now fix $x$ and $\widetilde{y}$, or equivalently $x$ and $\widetilde{t}$, and
let us only consider when $t_{1}=x_{1}-y_{1}>0$, as the case $t_{1}<0$ will be
similar. Then since $t_{1}\mapsto\theta(t_{1})$ is a decreasing injective map
from $\mathbb{R}_{+}\rightarrow(0,\frac{\pi}{2})$, then $y_{1}\mapsto
F(x,y_{1},\widetilde{y})$ is $M$-piecewise monotone on $\{y_{1} \in\mathbb{R}
~:~y_{1}<x_{1}\}$ if $\theta\mapsto D_{x,\widetilde{t}}\left(  \theta\right)
$ is $M$-piecewise monotone on $(0,\frac{\pi}{2})$. Since $t_{1}>0$, then $F$
is positive and so is $D_{x,\widetilde{t}}$ when $\theta> 0$, since both
functions possess the same sign. Since $u\mapsto u^{2}$ is increasing for
$u>0$, then $D_{x,\widetilde{t}}\left(  \theta\right)  $ is $M$-piecewise
monotone if and only if $D_{x,\widetilde{t}}\left(  \theta\right)  ^{2}$ is
$M$-piecewise monotone, which we will now show below.

In the reasoning that follows, we assume all rational functions we consider
below are non-constant; in the case one of them is constant or even
identically $0$, the proof of $M$-piecewise monotonicity is even simpler than
the proof below, the details of which we leave to the reader. We have
\begin{align*}
D_{x,\widetilde{t}}\left(  \theta\right)  ^{2}  &  =E_{1}^{2}\left[
E_{0}-\tan^{2}\theta\right]  \left[  R_{n}\left(  z,\sqrt{1+z^{2}}\right)
\tan^{n-1}\theta+\left(  C_{n}\theta+C\right)  \tan^{n-1}\theta\right]  ^{2}\\
&  =R_{n}\left(  z,\sqrt{1+z^{2}}\right)  \theta^{2}+T_{n}\left(
z,\sqrt{1+z^{2}}\right)  \theta+U_{n}\left(  z,\sqrt{1+z^{2}}\right)  .
\end{align*}
To check $D_{x,\widetilde{t}}\left(  \theta\right)  ^{2}$ is $M$ monotone, it
suffices to show $D_{x,\widetilde{t}}\left(  \theta\right)  ^{2}$ has at most
$M$ critical points. For this we compute
\begin{align*}
\frac{d}{d\theta}D_{x,\widetilde{t}}\left(  \theta\right)  ^{2}  &
=R_{n}\left(  z,\sqrt{1+z^{2}}\right)  \theta^{2}+T_{n}\left(  z,\sqrt
{1+z^{2}}\right)  \theta+U_{n}\left(  z,\sqrt{1+z^{2}}\right) \\
&  =R_{n}\left(  z,\sqrt{1+z^{2}}\right)  \left\{  \theta^{2}+T_{n}\left(
z,\sqrt{1+z^{2}}\right)  \theta+U_{n}\left(  z,\sqrt{1+z^{2}}\right)
\right\}
\end{align*}
which equals $0$ or $\infty$ if
\[
R_{n}\left(  z,\sqrt{1+z^{2}}\right)  =0\text{ or }\infty\,,\text{ or }%
\theta^{2}+\theta R_{n}\left(  z,\sqrt{1+z^{2}}\right)  +T_{n}\left(
z,\sqrt{1+z^{2}}\right)  =0\text{ or }\infty\,.
\]
The first equality can clearly only hold for at most $M$ values of $\theta$.
To show the function
\[
\theta^{2}+\theta R_{n}\left(  z,\sqrt{1+z^{2}}\right)  +T_{n}\left(
z,\sqrt{1+z^{2}}\right)
\]
can equal $0$ or $\infty$ at most $M$ times, it suffices to show that this
function also has at most $M$ critical points.

Its derivative is of the form
\[
R_{n}\left(  z,\sqrt{1+z^{2}}\right)  \left(  \theta+T_{n}\left(
z,\sqrt{1+z^{2}}\right)  \right)  \,,
\]
which we claim equals $0$ or $\infty$ at most $M$ times. Indeed, $R_{n}$
equals $0$ or $\infty$ at most $M$ times, and the function
\[
\theta+T_{n}\left(  z,\sqrt{1+z^{2}}\right)
\]
equals $0$ or $\infty$ at most $M$ times because its derivative is given by
\[
1+T_{n}\left(  z,\sqrt{1+z^{2}}\right)  ,\,
\]
which in turn equals $0$ or $\infty$ at most $M$ times.

Thus $y_{1} \mapsto F(x,y_{1},\widetilde{y})$ is $M$-piecewise monotone for
some $M$ depending only on $n$, and not on the additional parameters $x$ and
$y_{2},...,y_{n}$. This completes the verification of property (i) in Lemma
\ref{Alt_Series_strong_convergence}.

\textbf{(ii)} For any $x\in Q$ we have from (\ref{F}) and $\left\vert
u_{n}\right\vert =\frac{\left\vert 1+x_{n}\right\vert }{\left\vert \widehat
{x}-\widehat{y}\right\vert }$ that\
\[
\left\vert F\left(  x,y_{1},\widetilde{y}\right)  \right\vert \leq
\frac{\left\vert x_{1}-y_{1}\right\vert }{\left\vert \widehat{x}-\widehat
{y}\right\vert ^{n}}\int_{|u_{n}|}^{\infty}\frac{1}{\left(  1+z^{2}\right)
^{\frac{n+1}{2}}}dz=\frac{\left\vert x_{1}-y_{1}\right\vert }{\left\vert
1+x_{n}\right\vert ^{n}}\left\vert u_{n}\right\vert ^{n}\int_{|u_{n}|}%
^{\infty}\frac{1}{\left(  1+z^{2}\right)  ^{\frac{n+1}{2}}}dz
\]

We claim that $\left\vert u_{n}\right\vert ^{n}\int_{|u_{n}|}^{\infty}\frac
{1}{\left(  1+z^{2}\right)  ^{\frac{n+1}{2}}}dz\leq C_{n}$. Indeed, when
$\left\vert u_{n}\right\vert \leq1$, this follows from integrability of the
integrand, and when $\left\vert u_{n}\right\vert \geq1$ this follows from a
direct computation using the fact that $\left(  1+z^{2}\right)  ^{\frac
{n+1}{2}}\approx z^{n+1}$. Thus $\left\vert F\left(  x,y_{1},\widetilde
{y}\right)  \right\vert \leq C_{n}\frac{\left\vert x_{1}-y_{1}\right\vert
}{\left\vert 1+x_{n}\right\vert ^{n}}\leq C_{n,Q,\delta}$ when $y\in P$, $x\in
Q\setminus H_{\delta}^{P;Q}$.

\textbf{(iii)} To show $\mathbf{1}_{H_{\delta}^{P;Q}}(x)\int_{-1}^{1}%
\int_{[-1,1]^{n-2}}\left\vert F\left(  x,y_{1},\widetilde{y}\right)
\right\vert d\widetilde{y}dy_{1}\rightarrow0$ strongly in $L^{p}%
(\mathbb{R}^{n})$ as $\delta\rightarrow0$, we split%
\begin{align*}
\mathbf{1}_{H_{\delta}^{P;Q}}(x)\int_{-1}^{1}\int_{[-1,1]^{n-2}}\left\vert
F\left(  x,y_{1},\widetilde{y}\right)  \right\vert d\widetilde{y}dy_{1}  &
\leq\mathbf{1}_{H_{\delta}^{P;Q}}(x)\int_{\left\{  \widehat{y}\in
\lbrack-1,1]^{n-1}:\left\vert \widehat{x}-\widehat{y}\right\vert >\left\vert
1+x_{n}\right\vert \right\}  }\left\vert F\left(  x,y_{1},\widetilde
{y}\right)  \right\vert d\widehat{y}\\
&  +\mathbf{1}_{H_{\delta}^{P;Q}}(x)\int_{\left\{  \widehat{y}\in
\lbrack-1,1]^{n-1}:\left\vert \widehat{x}-\widehat{y}\right\vert
\leq\left\vert 1+x_{n}\right\vert \right\}  }\left\vert F\left(
x,y_{1},\widetilde{y}\right)  \right\vert d\widehat{y}.
\end{align*}

To bound the first term, we use the estimate%
\begin{align*}
\left\vert F\left(  x,y_{1},\widetilde{y}\right)  \right\vert \leq
\frac{\left\vert x_{1}-y_{1}\right\vert }{\left\vert \widehat{x}-\widehat
{y}\right\vert ^{n}}\int_{|u_{n}|}^{\infty}\frac{1}{\left(  1+z^{2}\right)
^{\frac{n+1}{2}}}dz\leq\frac{1}{\left\vert \widehat{x}-\widehat{y}\right\vert
^{n-1}}\int_{|u_{n}|}^{\infty}\frac{1}{\left(  1+z^{2}\right)  ^{\frac{n+1}%
{2}}}dz\leq C_{n}\frac{1}{\left\vert \widehat{x}-\widehat{y}\right\vert
^{n-1}},
\end{align*}
and polar coordinates to get%
\begin{align*}
&  \int_{\left\{  \widehat{y}\in\lbrack-1,1]^{n-1}:\left\vert \widehat
{x}-\widehat{y}\right\vert >\left\vert 1+x_{n}\right\vert \right\}
}\left\vert F\left(  x,y_{1},\widetilde{y}\right)  \right\vert d\widehat
{y}\leq C_{n}\int_{\left\{  \widehat{y}\in\lbrack-1,1]^{n-1}:\left\vert
\widehat{x}-\widehat{y}\right\vert >\left\vert 1+x_{n}\right\vert \right\}
}\frac{1}{\left\vert \widehat{x}-\widehat{y}\right\vert ^{n-1}}d\widehat{y}\\
&  \leq C_{n}\int_{\mathbb{S}^{n-2}}\int_{\left\vert 1+x_{n}\right\vert
}^{c_{Q}}\frac{1}{r}drd\theta\leq C_{n}\ln\frac{1}{\operatorname*{dist}%
(x_{n},\partial P_{n})},
\end{align*}
where we have used the fact that $\left\vert \widehat{x}-\widehat
{y}\right\vert \leq c_{Q}$. Thus
\[
\mathbf{1}_{H_{\delta}^{P;Q}}(x)\int_{\left\{  \widehat{y}\in\lbrack
-1,1]^{n-1}:\left\vert \widehat{x}-\widehat{y}\right\vert >\left\vert
1+x_{n}\right\vert \right\}  }\left\vert F\left(  x,y_{1},\widetilde
{y}\right)  \right\vert d\widehat{y}\rightarrow0
\]
strongly in $L^{p}(Q)$ as $\delta\rightarrow0$.

As for the second term, for $\left\vert u_{n}\right\vert \geq1$ we estimate%
\begin{align*}
&  \left\vert F\left(  x,y_{1},\widetilde{y}\right)  \right\vert \leq
\frac{\left\vert x_{1}-y_{1}\right\vert }{\left\vert \widehat{x}-\widehat
{y}\right\vert ^{n}}\int_{|u_{n}|}^{\infty}\frac{1}{\left(  1+z^{2}\right)
^{\frac{n+1}{2}}}dz \leq\frac{1}{\left\vert \widehat{x}-\widehat{y}\right\vert
^{n-1}}\int_{|u_{n}|}^{\infty}\frac{1}{\left(  1+z^{2}\right)  ^{\frac{n+1}%
{2}}}dz\\
&  =\frac{\left\vert u_{n}\right\vert ^{n-1}}{\left\vert 1+x_{n}\right\vert
^{n-1}}\int_{|u_{n}|}^{\infty}\frac{1}{\left(  1+z^{2}\right)  ^{\frac{n+1}%
{2}}}dz\leq\frac{C_{n}}{\left\vert 1+x_{n}\right\vert ^{n-1}},
\end{align*}
and so%
\[
\int_{\left\{  \widehat{y}\in\lbrack-1,1]^{n-1}:\left\vert \widehat
{x}-\widehat{y}\right\vert \leq\left\vert 1+x_{n}\right\vert \right\}
}\left\vert F\left(  x,y_{1},\widetilde{y}\right)  \right\vert d\widehat
{y}\leq C_{n}\int_{\left\{  \widehat{y}\in\lbrack-1,1]^{n-1}:\left\vert
\widehat{x}-\widehat{y}\right\vert \leq\left\vert 1+x_{n}\right\vert \right\}
}\frac{1}{\left\vert 1+x_{n}\right\vert ^{n-1}}d\widehat{y}\leq C_{n}.
\]

Thus
\[
\mathbf{1}_{H_{\delta}^{P;Q}}(x)\int_{\left\{  \widehat{y}\in\lbrack
-1,1]^{n-1}:\left\vert \widehat{x}-\widehat{y}\right\vert \leq\left\vert
1+x_{n}\right\vert \right\}  }\left\vert F\left(  x,y_{1},\widetilde
{y}\right)  \right\vert d\widehat{y}\rightarrow0
\]
strongly in $L^{p}(Q)$ as $\delta\rightarrow0$.
\end{proof}

The next lemma is an extension of the one-dimensional lemma of Nazarov in
\cite{NaVo}.

\begin{lemma}
\label{reduction}Suppose $p\in(1,\infty)$. Let $a$ and $b$ be nonnegative
integers, not both zero. Given a cube $P=P_{1}\times P_{2}\times...\times
P_{n}\subset\mathbb{R}^{n}$, we have

\begin{enumerate}
\item $\lim_{k\rightarrow\infty}\int_{\mathbb{R}^{n}}\left(  s_{k}%
^{P,\hor}\left(  x\right)  \right)  ^{a}\left(
R_{1}s_{k}^{P,\hor}\left(  x\right)  \right)
^{b}f\left(  x\right)  dx=0$ for all functions $f\in L^{p}(\mathbb{R}^{n})$
when $a$ or $b$ is odd.

\item $\lim_{k\rightarrow\infty}\int_{\mathbb{R}^{n}}\left(  s_{k}%
^{P,\hor}\left(  x\right)  \right)  ^{a}\left(
R_{1}s_{k}^{P,\hor}\left(  x\right)  \right)
^{b}f\left(  x\right)  dx=C_{a,b}\int_{P}f\left(  x\right)  dx$ for all
functions $f\in L^{p}(\mathbb{R}^{n})$ when both $a$ and $b$ are even, and
$C_{a,b}>0$ and $C_{0,2}=\bn{n}^{2}$.

\item $R_{j}s_{k}^{P,\hor}\left(  x\right)  $ tends to
$0$ strongly in $L^{p}\left(  \mathbb{R}^{n}\right)  $ as $k\rightarrow\infty$
for all $p\in(1,\infty)$, for all $2 \leq j \leq n$.
\end{enumerate}
\end{lemma}

\begin{remark}
A careful reading of the proofs of Lemma
\ref{Alt_Series_strong_convergence_old} and part (3) above show that for all
$k\geq1$ and $M>1$, we have the pointwise inequality%
\[
\left\vert R_{2}s_{k}^{P,\hor}\left(  x\right)
\right\vert \leq C\ln\frac{1}{\operatorname*{dist}(x_{2},\partial P_{2}%
)}\mathbf{1}_{\left\{  \operatorname*{dist}(x_{2},\partial P_{2}%
)<\delta\right\}  }\left(  x\right)  +C_{\delta}2^{-k}\mathbf{1}_{\left\{
\operatorname*{dist}(x_{2},\partial P_{2})\geq\delta\right\}  }\left(
x\right)  ,\ \ \ \ x\in MP.
\]

\end{remark}

\begin{proof}
\textbf{(1) and (2): }By Lemma \ref{rep} , we may write
\[
\left(  s_{k}^{P,\hor}\left(  x\right)  \right)
^{a}\left(  R_{1}s_{k}^{P,\hor}\left(  x\right)  \right)
^{b}f\left(  x\right)  =\bn{n}^{b}\left(  s_{k}^{P_{1}}\left(  x_{1} \right)
\right)  ^{a}\left(  Hs_{k}^{P_{1}}\left(  x_{1}\right)  \right)  ^{b}f\left(
x\right)  \mathbf{1}_{P^{\prime}}\left(  x^{\prime}\right)  +\err_{k}%
^{P,f,a,b}(x),
\]
where $\err_{k}^{P,f,a,b}(x)$ goes to $0$ strongly in $L^{1}(Q)$, and $P^{\prime
}=P_{2}\times...\times P_{n}$ and $x=\left(  x_{2},...x_{n}\right)  $. Thus
integrating over $\mathbb{R}^{n}$ and using Lemma
\ref{Nazarov_weak_convergence} yields the conclusions sought.

\textbf{(3)}: By permuting variables, we can assume without loss of generality that $j=2$. Let $\varepsilon>0$. Arguing as in the proof of Lemma \ref{rep},
for every $M>1$,we have
\[
R_{2}s_{k}^{P,\hor}(x)=R_{2}s_{k}%
^{P,\hor}(x)\mathbf{1}_{MP}(x)+R_{2}s_{k}%
^{P,\hor}(x)\mathbf{1}_{\mathbb{R}^{n}\setminus MP}(x).
\]

We note that the second term $R_{2}s_{k}^{P,\hor%
}(x)\mathbf{1}_{\mathbb{R}^{n}\setminus MP}(x)$ goes to $0$ strongly in
$L^{p}(\mathbb{R}^{n})$ as $M\rightarrow\infty$, since%
\[
\int_{\mathbb{R}^{n}\setminus MP}\left\vert R_{2}s_{k}%
^{P,\hor}(x)\right\vert ^{p}dx\leq C\int_{\mathbb{R}%
^{n}\setminus MP}\left(  \int_{P}\frac{1}{\left\vert x-y\right\vert ^{n}%
}dy\right)  ^{p}dx\leq C\int_{\mathbb{R}^{n}\setminus MP}\left(
\frac{\left\vert P\right\vert }{\left\vert \operatorname*{dist}\left(
x,P\right)  \right\vert ^{n}}\right)  ^{p}dx,
\]
which goes to $0$ as $M\rightarrow\infty$. So choose $M$ such that
$\int_{\mathbb{R}^{n}\setminus MP}\left\vert R_{2}s_{k}%
^{P,\hor}(x)\right\vert ^{p}dx<\frac{\varepsilon}{2}$.
Thus with $Q=MP$, it remains to show that $\left\Vert R_{2}s_{k}%
^{P,\hor}\right\Vert _{L^{p}(Q)}<\frac{\varepsilon}{2}$
for $k$ sufficiently large, which we show below.

Again we may assume that $P=\left[  -1,1\right]  ^{n}$. We have%
\begin{align*}
R_{2}s_{k}^{P,\hor}\left(  x\right)   &  =\int_{-1}%
^{1}\int_{-1}^{1}...\int_{-1}^{1}c_{n}\frac{\left(  x_{2}-y_{2}\right)
s_{k}^{\left[  -1,1\right]  }\left(  y_{1}\right)  }{\left[  \left(
x_{1}-y_{1}\right)  ^{2}+\left\vert x^{\prime}-y^{\prime}\right\vert
^{2}\right]  ^{\frac{n+1}{2}}}dy_{1}dy^{\prime}\\
&  \equiv \int_{-1}^{1}F(x,y_{1})s_{k}^{\left[  -1,1\right]  }\left(  y_{1}\right)
dy_{1}.
\end{align*}
For each fixed $y^{\prime\prime}\in\left[  -1,1\right]  ^{n-2}$ define the
function%
\begin{align}
F_{y^{\prime\prime}}(x,y_{1})  &  \equiv-c_{n}\int_{x_{2}+1}^{x_{2}-1}\frac
{t}{\left[  \left(  x_{1}-y_{1}\right)  ^{2}+t^{2}+\left\vert x^{\prime\prime
}-y^{\prime\prime}\right\vert ^{2}\right]  ^{\frac{n+1}{2}}}dt \nonumber \\
&  =-c_{n}\int_{|x_{2}+1|}^{|x_{2}-1|}\frac{t}{\left[  \left(  x_{1}%
-y_{1}\right)  ^{2}+t^{2}+\left\vert x^{\prime\prime}-y^{\prime\prime
}\right\vert ^{2}\right]  ^{\frac{n+1}{2}}}dt, \label{eq:F_odd_kernel}
\end{align}
where the second line follows from oddness of the kernel, and thus using the
substitution $t=x_{2}-y_{2}$ we have%
\[
F(x,y_{1})=\int_{[-1,1]^{n-2}}F_{y^{\prime\prime}}(x,y_{1})dy^{\prime\prime} \, .
\]
Thus to show $\left\Vert R_{2}s_{k}%
^{P,\hor}\right\Vert _{L^{p}(Q)} \to 0$ as $k \to \infty$, it suffices to show that $F_{y^{\prime\prime}}(x,y_{1})$
satisfies conditions \textbf{(i)}-\textbf{(iii)} of Lemma
\ref{Alt_Series_strong_convergence_old}, noting that for each fixed $x$, this function of $y_1$ doesn't change sign.

\textbf{(i) }Now we note that
\begin{align*}
F_{y^{\prime\prime}}\left(  x,y_{1}\right)   &  =-c_n\int_{\left\vert
x_{2}+1\right\vert }^{\left\vert x_{2}-1\right\vert }\frac{t}{\left[  \left(
x_{1}-y_{1}\right)  ^{2}+t^{2}+\left\vert x^{\prime\prime}-y^{\prime\prime
}\right\vert ^{2}\right]  ^{\frac{n+1}{2}}}dt \, ,
\end{align*}
and so differentiating in $y_1$ yields
\begin{align*}
	\frac{\partial }{\partial y_1} F_{y^{\prime\prime}}\left(  x,y_{1}\right)   &  =c_n ' (x_1 - y_1) \int_{\left\vert
x_{2}+1\right\vert }^{\left\vert x_{2}-1\right\vert }\frac{t}{\left[  \left(
x_{1}-y_{1}\right)  ^{2}+t^{2}+\left\vert x^{\prime\prime}-y^{\prime\prime
	}\right\vert ^{2}\right]  ^{\frac{n+1}{2}+1}} dt \, .
\end{align*}
The integral above is of one sign, and so $\frac{\partial }{\partial y_1} F_{y^{\prime\prime}}\left(  x,y_{1}\right)$ only changes sign at $y_1 = x_1$. Thus $F_{y^{\prime\prime}}\left(  x,y_{1}\right)$ has at most $1$ critical point in $y_{1}$, and so is $2$-monotone.

\textbf{(ii) }By (\ref{eq:F_odd_kernel}), we have
\begin{align*}
\left\vert F\left(  x,y_{1}\right)  \right\vert  &  \leq c_{n} \int
_{[-1,1]^{n-2}}\left\{  \int_{\min\{ |x_{2} + 1|, |x_{2} - 1| \}}^{\max\{
|x_{2} + 1|, |x_{2} - 1| \}}\frac{t}{\left[  \left(  x_{1}-y_{1}\right)
^{2}+t^{2}+ \left\vert x^{\prime\prime}-y^{\prime\prime}\right\vert
^{2}\right]  ^{\frac{n+1}{2}}}dt \right\}  dy^{\prime\prime}\\
&  \leq c_{n} \int_{[-1,1]^{n-2}}\left\{  \int_{\min\{ |x_{2} + 1|, |x_{2} -
1| \}}^{\max\{ |x_{2} + 1|, |x_{2} - 1| \}}\frac{t}{\delta^{n+1}}dt \right\}
dy^{\prime\prime} \, ,
\end{align*}
since if $x\in Q\setminus H_{\delta}^{P;Q}$ then $t >\delta$ by separation.
Thus $\left\vert \mathbf{1}_{Q\setminus H_{\delta}^{P;Q}}(x)F\left(
x,y_{1}\right)  \right\vert \leq C\frac{1}{\delta^{n+1}}$.

\textbf{(iii)} Let
\begin{align*}
A_{x}  &  \equiv\left\{  (y_{1},y^{\prime\prime})\in\lbrack-1,1]^{n-1}%
:\left\vert (x_{1}-y_{1},x^{\prime\prime}-y^{\prime\prime})\right\vert
>\left\vert 1-x_{2}\right\vert \right\}  \text{,}\\
\text{ }B_{x}  &  \equiv\left\{  (y_{1},y^{\prime\prime})\in\lbrack
-1,1]^{n-1}:\left\vert (x_{1}-y_{1},x^{\prime\prime}-y^{\prime\prime
})\right\vert <\left\vert 1-x_{2}\right\vert \right\}  \,,
\end{align*}
and assume without loss of generality that $|x_{2}-1|\leq|x_{2}+1|$. For $x\in
H_{\delta}^{P;Q}$, we have%
\begin{align*}
&  \int_{-1}^{1}\left\vert F(x,y_{1})\right\vert dy_{1}\lesssim\int_{-1}%
^{1}\int_{[-1,1]^{n-2}}\left\{  \int_{\min\{|x_{2}+1|,|x_{2}-1|\}}^{\infty
}\frac{t}{\left[  \left(  x_{1}-y_{1}\right)  ^{2}+t^{2}+\left\vert
x^{\prime\prime}-y^{\prime\prime}\right\vert ^{2}\right]  ^{\frac{n+1}{2}}%
}dt\right\}  dy^{\prime\prime}dy_{1}\\
&  =\frac{1}{n-1}\int_{-1}^{1}\int_{[-1,1]^{n-2}}\frac{1}{\left[  \left(
x_{1}-y_{1}\right)  ^{2}+\left(  1-x_{2}\right)  ^{2}+\left\vert
x^{\prime\prime}-y^{\prime\prime}\right\vert ^{2}\right]  ^{\frac{n-1}{2}}%
}dy^{\prime\prime}dy_{1}\,\\
&  \leq\left\{  \int_{A_{x}}+\int_{B_{x}}\right\}  \frac{1}{\left[  \left(
x_{1}-y_{1}\right)  ^{2}+\left(  1-x_{2}\right)  ^{2}+\left\vert
x^{\prime\prime}-y^{\prime\prime}\right\vert ^{2}\right]  ^{\frac{n-1}{2}}%
}d(y_{1},y^{\prime\prime})\\
&  \leq\int_{A_{x}}\frac{1}{\left[  \left(  x_{1}-y_{1}\right)  ^{2}%
+\left\vert x^{\prime\prime}-y^{\prime\prime}\right\vert ^{2}\right]
^{\frac{n-1}{2}}}d(y_{1},y^{\prime\prime})+\int_{B_{x}}\frac{1}{\left\vert
1-x_{2}\right\vert ^{n-1}}d(y_{1},y^{\prime\prime}).
\end{align*}

By a crude estimate the second integral is bounded by
\[
\int_{B_{x}}\frac{1}{\left\vert 1-x_{2}\right\vert ^{n-1}}d(y_{1}%
,y^{\prime\prime})\leq C_{n}|B_{x}|\frac{1}{\left\vert 1-x_{2}\right\vert
^{n-1}}\leq C_{n}.
\]
As for the first integral, integration using polar coordinates yields the
upper bound%
\[
c\int_{\left\vert 1-x_{2}\right\vert }^{c_{n}}\frac{r^{n-2}}{r^{n-1}}%
dr=c\ln\frac{c_{n}}{\left\vert 1-x_{2}\right\vert }\in L^{p}(Q).
\]
Similar estimates hold when $\left\vert x_{2}+1\right\vert <\left\vert
x_{2}-1\right\vert $ and $x\in H_{\delta}^{P;Q}$. Thus $\mathbf{1}_{H_{\delta
}^{P;Q}}(x)\int\limits_{-1}^{1}\left\vert F(x,y_{1})\right\vert dy_{1}$ goes
to $0$ strongly in $L^{p}(Q)$ as $\delta\rightarrow0$.
\end{proof}

\begin{theorem}
\label{weak van Riesz} The conclusions of Lemma
\ref{lem:Nazarov_convergence_mixed}, namely (\ref{eq:1_multi_weak}), (\ref{eq:2_multi_weak}) and (\ref{eq:3_multi_weak}), hold if one replaces $H$ by $R_{1}$ and
$s_{k}^{I_{1}}$ by $s_{k}^{I,\hor}$, and similarly for
$J,K$.
\end{theorem}

\begin{proof}
One argues as previously in the proofs of Lemma \ref{reduction} parts
\textbf{(1)} and \textbf{(2)}, in particular using Lemmas \ref{rep} and Lemma
\ref{lem:Nazarov_convergence_mixed}.
\end{proof}

\section{Boundedness properties of the Riesz transforms}

\label{section:full_proof}

We now are equipped with the convergence results we need to complete the proof
of the main theorem by following the supervisor argument of Nazarov in
\cite{NaVo}. We begin with a short formal argument, then
we adapt Nazarov's supervisor argument for the Hilbert
transform to the transplantation of Riesz transforms, and then complete the proof by extending our weights to all of $\mathbb{R}^n$ and showing the Riesz transform $R_1$ has large norm for this weight pair, while $R_2$ has small norm.

\subsection{A brief overview of the argument}

We now take $Q^{0} \equiv \left[  0,1\right]  ^{n}$ to be the unit cube in $\mathbb{R}^n$, and let $V,U$ be as in Theorem \ref{Bellman Haar shift}. We
apply the transplantation argument of Section \ref{section:sprvsr_trans} to
$V,U$ to obtain weights $v_{t},u_{t}$ for all $1\leq t\leq m$, with $u\equiv u_{m}$,
$v \equiv v_{m}$. We will compute the $R_1$-testing conditions for $(v,u)$ \ by first
estimating them for the pair $\left(  v_{t+1}-v_{t},u_{t}\right)  $. Since both $V,U$ have Haar support on finitely many horizontal Haar wavelets in $Q^0$, then by the estimates of Section \ref{section:action_of_Riesz},
we obtain that in the limit only the diagonal terms in $\left[  R_{1}\left(
v_{t+1}-v_{t}\right)  \right]  ^{2}$ survive the integration with $u_{t}$. Indeed,
recall that%
\[
R_{1}\left(  v_{t+1}-v_{t}\right)  =\sum_{Q\in\mathcal{K}_{t}}\left\langle
V,h_{\mathcal{S}\left(  Q\right)  }^{\hor}\right\rangle
\frac{1}{\sqrt{\left\vert S\left(  Q\right)  \right\vert }}R_{1}s_{k_{t+1}%
}^{Q,\hor},
\]
and the vanishing weak convergence results of Section
\ref{section:action_of_Riesz} yield for \add{$k_{t+1} \geq C(k_1, \ldots, k_t)$} and $Q, Q^{\prime}$
dyadic subcubes of $[0,1]^{n}$
\[
\int R_{1}s_{k_{t+1}}^{Q,\hor}R_{1}s_{k_{t+1}}%
^{Q^{\prime},\hor}u_{t}\rightarrow%
\begin{cases}
0 & \text{ if }Q\neq Q^{\prime}\\
\left(  \bn{n}\right)  ^{2}\int\limits_{Q}u_{t} & \text{ if }Q=Q^{\prime}%
\end{cases}
\text{ on }\left[  0,1\right]  ^{n},
\]
where $\bn{n}$ is the constant appearing in Lemma \ref{rep}, and so using once
again the vanishing weak convergence results of Section
\ref{section:action_of_Riesz} we get for \add{$k_{t+1} \geq C(k_{1}, \ldots, k_t)$}
\begin{align*}
&  \int\left[  R_{1}\left(  v_{t+1}-v_{t}\right)  \right]  ^{2}u_{t}%
=\int\left[  \sum_{Q\in\mathcal{K}_{t}}\left\langle V,h_{\mathcal{S}\left(
Q\right)  }^{\hor}\right\rangle \frac{1}{\sqrt{\left\vert
S\left(  Q\right)  \right\vert }}R_{1}s_{k_{t+1}}^{Q,\hor%
}\right]  ^{2}u_{t}\\
&  =\sum_{Q\in\mathcal{K}_{t}}\int\left\langle V,h_{\mathcal{S}\left(
Q\right)  }^{\hor}\right\rangle ^{2}\left[
R_{1}s_{k_{t+1}}^{Q,\hor}\right]  ^{2}\frac{1}{\left\vert
\mathcal{S}\left(  Q\right)  \right\vert }u_{t}+\operatorname*{offdiagonal}%
\rightarrow\left(  \bn{n}\right)  ^{2}\sum_{Q\in\mathcal{K}_{t}}\left\langle
V,h_{\mathcal{S}\left(  Q\right)  }^{\hor}\right\rangle
^{2}\frac{1}{\left\vert \mathcal{S}\left(  Q\right)  \right\vert }\int
_{Q}u_{t},
\end{align*}
and if we now \add{sum} in $t$, pigeonhole cubes $Q$ based on their
supervisor $S$, use the fact that $E_{Q}u_{t}=E_{S}U$, and finally
$\sum_{\substack{Q\in\mathcal{K}_{t}\\\mathcal{S}\left(  Q\right)  =S}%
}\frac{\left\vert Q\right\vert }{\left\vert S\right\vert }=1$, we obtain%
\begin{align*}
&  \int\left[  R_{1}\sum_{t=1}^{m-1}\left(  v_{t+1}-v_{t}\right)  \right]
^{2}u_{t}\approx\sum_{t=1}^{m-1}\int\left[  R_{1}\left(  v_{t+1}-v_{t}\right)
\right]  ^{2}u_{t}\\
&  \approx\left(  \bn{n}\right)  ^{2}\sum_{t=1}^{m-1}\sum_{Q\in\mathcal{K}_{t}%
}\left\langle V,h_{\mathcal{S}\left(  Q\right)  }^{\hor%
}\right\rangle ^{2}\frac{1}{\left\vert \mathcal{S}\left(  Q\right)
\right\vert }\int_{Q}u_{t}=\left(  \bn{n}\right)  ^{2}\sum_{t=1}^{m-1}\sum
_{S\in\mathcal{D}_{t}}\sum_{\substack{Q\in\mathcal{K}_{t}\\\mathcal{S}\left(
Q\right)  =S}}\left\langle V,h_{S}^{\hor}\right\rangle
^{2}E_{Q}u_{t}\frac{\left\vert Q\right\vert }{\left\vert S\right\vert }\\
&  =\left(  \bn{n}\right)  ^{2}\sum_{t=1}^{m-1}\sum_{S\in\mathcal{D}_{t}}%
\sum_{\substack{Q\in\mathcal{K}_{t}\\\mathcal{S}\left(  Q\right)
=S}}\left\langle V,h_{S}^{\hor}\right\rangle ^{2}%
E_{S}U\frac{\left\vert Q\right\vert }{\left\vert S\right\vert }  =\left(  \bn{n}\right)  ^{2}\sum_{t=1}^{m-1}\sum_{S\in\mathcal{D}_{t}%
}\left\langle V,h_{S}^{\hor}\right\rangle ^{2}%
E_{S}U>\left(  \bn{n}\right)  ^{2} \Gamma \left(  E_{[0,1]^{n}}V\right)
,
\end{align*}
which shows that testing for $R_{1}$ blows up, and hence two-weight norm for $R_1$ blows up as well. On the other hand, \add{we will see that dyadic testing for $R_2$ is controlled by the dyadic $A_2$ condition, namely
\[
\sup\limits_{Q \in \mathcal{D} ([0,1]^n)} \frac{1}{|Q|_v} \int\limits_{Q} \left | R_2 \mathbf{1}_Q v \right |^2 u + \sup\limits_{Q \in \mathcal{D} ([0,1]^n)} \frac{1}{|Q|_u} \int\limits_{Q} \left | R_2 \mathbf{1}_Q u \right |^2 v \lesssim A_2 ^{\operatorname{dyadic}} (u,v; [0,1]^n) \, . 
\]}
for $k_{1},k_{2},...,k_{m}$ all chosen large enough in an
inductive fashion. To make this formal argument precise in the next subsection, we follow the
scheme in \cite{NaVo} for $R_1$, while the scheme for $R_2$ is our own.

This gives us weights $(v,u)$ in the unit cube $[0,1]^n$. We then extend these weights periodically to the plane (with an additional small decay term), so that they continue to fails the norm inequality for $R_1$ (since the testing condition is large), while the \emph{dyadic} testing condition for $R_2$ holds. However, our weights will be doubling with doubling constant close to Lebesgue measure. So we will be able to leverage the $T1$ theorem of \cite{SaShUr10} and doubling to show that dyadic testing for $R_2$ implies that the norm inequality holds for $R_2$. Thus we will have constructed a weight pair for which $R_2$ is norm bounded, but $R_1$ is not, i.e., this weight pair will be rotationally unstable. 

\subsection{The Nazarov argument for Riesz transforms}

We now continue to carry out our adaptation of Nazarov's supervisor argument
to the higher dimensional setting of the supervisor and transplantation map.
Equipped with the supervisor and transplantation map, and the weak convergence
results above, this remaining argument follows the proof in \cite{NaVo} for $R_1$, but we include additional details that were omitted in \cite{NaVo} which will clarify the presentation here. The argument for $R_2$ is new, however. 

Recall that $\left\{  k_{t}\right\}  _{t=0}^{\infty}$ is a strictly increasing
sequence of nonnegative integers $k_{t}\in\mathbb{Z}_{+}$ with $k_{0}=0$, and
whose members will be chosen sufficiently large in the arguments below. We
define $\mathcal{K}\equiv\overset{\cdot}{%
%TCIMACRO{\dbigcup }%
%BeginExpansion
{\displaystyle\bigcup}
%EndExpansion
}_{t=0}^{\infty}\mathcal{K}_{t}$ where $\mathcal{K}_{0}=\left\{
Q^{0}\right\}  =\left\{  [0,1]^{n}\right\}  $ and
\[
\mathcal{K}_{t}\equiv\left\{  Q\in\mathcal{D} (Q^{0}) :\ell\left(  Q\right)
=2^{-k_{1}-k_{2}-...-k_{t}}\right\}  ,\ \ \ \ \ t\geq1.
\]

\begin{proposition}
[Nazarov \cite{NaVo} in the case of the Hilbert transform]%
\label{Riesz Nazarov}For every $\Gamma>1$ and $0<\tau<1$ there exist positive
weights $u,v$ on the unit cube $Q^{0}\equiv\left[  0,1\right]  ^{n}$
satisfying
\begin{align}
&  \int_{\left[  0,1\right]  ^{n}}\left\vert R_{1}v\left(  x\right)
\right\vert ^{2}u\left(  x\right)  dx\geq\Gamma\int_{\left[  0,1\right]  ^{n}%
}v\left(  x\right)  dx, \label{eq:norm_blow_up}\\
&  \int_{I}\left\vert R_{2}\mathbf{1}_{I}v\left(  x\right)  \right\vert
^{2}u\left(  x\right)  dx\leq\int_{I}v\left(  x\right)  dx,\ \ \ \ \ \text{for
all dyadic cubes }I\in\mathcal{D}^{0}, \label{eq:forwards_dyadic_testing_weights}\\
&  \int_{I}\left\vert R_{2}\mathbf{1}_{I}u\left(  x\right)  \right\vert
^{2}v\left(  x\right)  dx\leq\int_{I}u\left(  x\right)  dx,\ \ \ \ \ \text{for
all dyadic cubes }I\in\mathcal{D}^{0} \label{eq:backwards_dyadic_testing_weights},\\
&  \left(  \frac{1}{\left\vert I\right\vert }\int_{I}u\left(  x\right)
dx\right)  \left(  \frac{1}{\left\vert I\right\vert }\int_{I}v\left(
x\right)  dx\right)  \leq1,\ \ \ \ \ \text{for all dyadic cubes }I\in\mathcal{D}%
^{0}\text{ }, \label{eq:dyadic_A2_weights}\\
&  1-\tau<\frac{E_{J}u}{E_{K}u},\frac{E_{J}v}{E_{K}v}<1+\tau
,\ \ \ \ \ \text{for adjacent dyadic cubes }J,K\in\mathcal{D}^{0} ,  \label{eq:t-flat_weights}
\end{align}
where $J$ and $K$ in \eqref{eq:t-flat_weights} need not be dyadic siblings, only adjacent.
\end{proposition}

\begin{proof}
Let $V,U$ be as arising from Theorem \ref{Bellman Haar shift} with
$\frac{\gamma(V,U,Q^{0})}{E_{Q^{0}}V}>$ $\Gamma^{\prime}$ sufficiently large. 
We apply the transplantation argument of Section \ref{section:sprvsr_trans} to
$V,U$ to obtain \add{nonnegative} weights $v_{t},u_{t}$\ with $1\leq t\leq m$, and set 
\[
u\equiv u_{m} \, , \quad v\equiv v_{m} \, ,
\]
where $m$ is as in Theorem \ref{Bellman Haar shift}. 
It will be convenient to denote the
differences
\begin{align*}
\eta_{t+1}  &  \equiv u_{t+1}-u_{t} \, , \quad \delta_{t+1}    \equiv v_{t+1}-v_{t}
\end{align*}
respectively. Note that by (\ref{eq:vanishing_Haar}) and (\ref{eq:differnce_ut}), $\eta_{t}$ and $\delta_{t}$ are of the form
\[
\sum\limits_{Q\in\mathcal{K}_{t}}c_{Q}\frac{1}{\sqrt{\left\vert \mathcal{S}%
\left(  Q\right)  \right\vert }}s_{k_{t+1}}^{Q,\hor%
}=o_{k_{t+1}\rightarrow\infty}^{\operatorname*{weakly}}\left(  1\right)  \,,
\]
where the sum is $o_{k_{t+1}\rightarrow\infty}^{\operatorname*{weakly}}\left(
1\right)  $ because the constants $c_{Q}$ depend only on the levels $1$
through $t$ of the construction and the number of terms in the sum only
depends on $k_{1},\ldots,k_{t}$.
We may then write
\begin{align*}
u  &  \equiv \left ( E_{Q^{0}}U \right )\mathbf{1}_{Q^{0}}+\sum_{t=0}^{m-1}\sum_{Q\in
\mathcal{K}_{t}}\left\langle U,h_{\mathcal{S}\left(  Q\right)  }%
^{\hor}\right\rangle \frac{1}{\sqrt{\left\vert
\mathcal{S}\left(  Q\right)  \right\vert }}s_{k_{t+1}}%
^{Q,\hor},\\
v  &  \equiv \left ( E_{Q^{0}}V \right )\mathbf{1}_{Q^{0}}+\sum_{t=0}^{m-1}\sum_{Q\in
\mathcal{K}_{t}}\left\langle V,h_{\mathcal{S}\left(  Q\right)  }%
^{\hor}\right\rangle \frac{1}{\sqrt{\left\vert
\mathcal{S}\left(  Q\right)  \right\vert }}s_{k_{t+1}}%
^{Q,\hor} \add{.}
\end{align*}

We will now focus on the `testing' constants 
\[
\frac{1}{\left\vert \left[
0,1\right]  ^{n}\right\vert _{v}}\int_{\left[  0,1\right]  ^{n}}\left\vert
R_{1}v\left(  x\right)  \right\vert ^{2}u\left(  x\right)  dx
\]
and
\[
\sup_{Q\in\mathcal{D}(Q^{0})}\frac{1}{\left\vert Q\right\vert _{v}}\int
_{Q}\left\vert R_{2}\mathbf{1}_{Q}v\right\vert ^{2}u \, , \quad \sup_{Q\in
\mathcal{D}(Q^{0})}\frac{1}{\left\vert Q\right\vert _{u}}\int_{Q}\left\vert
R_{2}\mathbf{1}_{Q}u\right\vert ^{2}v \, ,
\] and show that the first is large, and
second and third are small, provided we take the integers $k_{t}$ sufficiently
large in an inductive fashion. \add{To tackle the first testing constant,} define the discrepancy for
$R_{1}$ on $Q^0 = [0,1]^n$ by%
\begin{align*}
\mathrm{Disc} \left(  t\right)   &  \equiv\int
_{Q^0}\left(  R_{1}\mathbf{1}_{Q}v_{t+1}\left(  x\right)  \right)  ^{2}%
u_{t+1}\left(  x\right)  dx-\int_{Q^0}\left(  R_{1}\mathbf{1}_{Q}v_{t}\left(
x\right)  \right)  ^{2}u_{t}\left(  x\right)  dx \, .
\end{align*}
We begin with the decomposition
\begin{align}
& \mathrm{Disc} \left(  t\right)     =\int_{Q^0}\left(
R_{1}\mathbf{1}_{Q^0}\delta_{t+1}+R_{1}\mathbf{1}_{Q^0}v_{t}\right)  ^{2}%
u_{t+1}-\int_{Q^0}\left(  R_{1}\mathbf{1}_{Q^0}v_{t}\right)  ^{2}u_{t}%
\nonumber\\
&  =\int_{Q^0}\left(  R_{1}\mathbf{1}_{Q^0}\delta_{t+1}\right)  ^{2}u_{t+1}%
+\int_{Q^0}\left\{  2\left(  R_{1}\mathbf{1}_{Q^0}\delta_{t+1}\right)  \left(
R_{1}\mathbf{1}_{Q^0}v_{t}\right)  \right\}  \left(  u_{t}+\eta_{t+1}\right)
+\int_{Q^0}\left(  R_{1}\mathbf{1}_{Q^0}v_{t}\right)  ^{2}\left(  u_{t+1}%
-u_{t}\right) \nonumber\\
&  =\left\langle \left(  R_{1}\mathbf{1}_{Q^0}\delta_{t+1}\right)  ^{2}%
,u_{t+1}\right\rangle _{L^{2}(Q^0)}+2\left\langle \left(  R_{1}\mathbf{1}%
_{Q^0}\delta_{t+1}\right)  \left(  R_{1}\mathbf{1}_{Q^0}v_{t}\right)
,u_{t}\right\rangle _{L^{2}(Q^0)} \nonumber\\
&  +2\left\langle \left(  R_{1}\mathbf{1}_{Q^0}\delta_{t+1}\right)  \left(
R_{1}\mathbf{1}_{Q^0}v_{t}\right)  ,\eta_{t+1}\right\rangle _{L^{2}%
(Q^0)}+\left\langle \left(  R_{1}\mathbf{1}_{Q^0}v_{t}\right)  ^{2},\eta
_{t+1}\right\rangle _{L^{2}(Q^0)}\nonumber\\
&  \equiv A+B+C+D.\nonumber
\end{align}

We first claim that 
\begin{equation}
\mathrm{Disc}\left(  t\right) =\left(  \bn{n}\right)  ^{2}\sum_{I\in\mathcal{D}:\ \ell\left(  I\right)
=2^{-t}}\left(  \bigtriangleup_{I}^{\hor}V\right)
^{2}\left(  E_{I}U\right)  +\sum_{r=0}^{t}o_{k_{r+1}\rightarrow\infty}\left(
1\right)  . \label{disc 1}%
\end{equation}
We will see in a moment that $A$ is the main term. Using that $v_{t}%
,u_{t}$ and $\delta_{t+1},\eta_{t+1}$ are supported in $\left[  0,1\right]
^{n}$,
\[
B=2\left\langle \left(  R_{1}v_{t}\right)  u_{t},R_{1}\delta
_{t+1}\right\rangle _{L^{2}\left(  \left[  0,1\right]  ^{n}\right)
}=-2\left\langle R_{1}\left[  \left(  R_{1}v_{t}\right)  u_{t}\right]
,\delta_{t+1}\right\rangle _{L^{2}\left(  \left[  0,1\right]  ^{n}\right)
}=o_{k_{t+1}\rightarrow\infty}\left(  1\right)
\]
since the function $R_{1}\left[  \left(  R_{1}v_{t}\right)  u_{t}\right]  \in
L^{p}(\mathbb{R}^{2})$ for all $p\in\left(  1,\infty\right)  $, and in
particular belongs to $L^{2}(\mathbb{R}^{2})$, and is independent of $k_{t+1}%
$, and finally since $\delta_{t+1}=o_{k_{t+1}\rightarrow\infty}%
^{\operatorname*{weakly}}\left(  1\right)  $. Similarly, since $R_{1}v_{t}\in
L^{4}(\mathbb{R}^{2})$, we have
\[
D=\left\langle \left(  R_{1}v_{t}\right)  ^{2},\eta_{t+1}\right\rangle
_{L^{2}\left(  \left[  0,1\right]  ^{n}\right)  }=o_{k_{t+1}\rightarrow\infty
}\left(  1\right)  .
\]

For term $C$ we have%
\begin{align*}
C  &  =2\left\langle \left(  R_{1}\delta_{t+1}\right)  \left(  R_{1}%
v_{t}\right)  ,\eta_{t+1}\right\rangle _{L^{2}\left(  \left[  0,1\right]
^{2}\right)  }\\
&  =2\int_{\left[  0,1\right]  ^{n}}\left(  \sum_{Q\in\mathcal{K}_{t}%
}\left\langle V,h_{\mathcal{S}\left(  Q\right)  }^{\hor%
}\right\rangle R_{1}\frac{1}{\sqrt{\left\vert \mathcal{S}\left(  Q\right)
\right\vert }}s_{k_{t+1}}^{Q,\hor}\right)  \left(
R_{1}v_{t}\right)  \left(  \sum_{Q^{\prime}\in\mathcal{K}_{t}}\left\langle
U,h_{\mathcal{S}\left(  Q^{\prime}\right)  }^{\hor%
}\right\rangle \frac{1}{\sqrt{\left\vert \mathcal{S}\left(  Q^{\prime}\right)
\right\vert }}s_{k_{t+1}}^{Q^{\prime},\hor}\right) \\
&  =2\sum_{Q,Q^{\prime}\in\mathcal{K}_{t}}\left\langle V,h_{\mathcal{S}\left(
Q\right)  }^{\hor}\right\rangle \left\langle
U,h_{\mathcal{S}\left(  Q^{\prime}\right)  }^{\hor%
}\right\rangle \int_{\left[  0,1\right]  ^{n}}R_{1}\frac{1}{\sqrt{\left\vert
\mathcal{S}\left(  Q\right)  \right\vert }}s_{k_{t+1}}%
^{Q,\hor}\left(  R_{1}v_{t}\right)  \frac{1}%
{\sqrt{\left\vert \mathcal{S}\left(  Q ^{\prime}\right)  \right\vert }%
}s_{k_{t+1}}^{Q^{\prime},\hor}=o_{k_{t+1}\rightarrow
\infty}\left(  1\right)  ,
\end{align*}
by Theorem \ref{weak van Riesz} since $\mathcal{K}_{t}$ and $R_{1}v_{t}$ are
both independent of $k_{t+1}$, while $\left(  R_{1}s_{k_{t+1}}%
^{Q,\hor}\right)  s_{k_{t+1}}^{Q^{\prime}%
,\hor}\rightarrow0$ weakly in $L^{2}\left(
\mathbb{R}^{n}\right)  .$

Finally, for term $A$ we have%
\[
A=\left\langle \left(  R_{1}\delta_{t+1}\right)  ^{2},u_{t+1}\right\rangle
_{L^{2}\left(  \left[  0,1\right]  ^{n}\right)  }=\left\langle \left(
\sum_{Q\in\mathcal{K}_{t}}\left\langle V,h_{\mathcal{S}\left(  Q\right)
}^{\hor}\right\rangle R_{1}\frac{1}{\sqrt{\left\vert
\mathcal{S}\left(  Q\right)  \right\vert }}s_{k_{t+1}}%
^{Q,\hor}\right)  ^{2},u_{t+1}\right\rangle
_{L^{2}\left(  \left[  0,1\right]  ^{n}\right)  }.
\]
We first note that if the sum is taken outside the square, so that we consider
only the `diagonal' terms, we have%
\begin{align*}
&  \left\langle \sum_{Q\in\mathcal{K}_{t}}\left(  \left\langle
V,h_{\mathcal{S}\left(  Q\right)  }^{\hor}\right\rangle
R_{1}\frac{1}{\sqrt{\left\vert \mathcal{S}\left(  Q\right)  \right\vert }%
}s_{k_{t+1}}^{Q,\hor}\right)  ^{2},u_{t+1}\right\rangle
\\
&  =\sum_{Q\in\mathcal{K}_{t}}\frac{1}{\left\vert \mathcal{S}\left(  Q\right)
\right\vert }\left\langle V,h_{\mathcal{S}\left(  Q\right)  }%
^{\hor}\right\rangle ^{2}\left\{  \left\langle \left(
R_{1}s_{k_{t+1}}^{Q,\hor}\right)  ^{2},u_{t}\right\rangle
+\left\langle \left(  R_{1}s_{k_{t+1}}^{Q,\hor}\right)
^{2},\eta_{t+1}\right\rangle \right\} \\
&  =\left(  \bn{n}\right)  ^{2}\left\{  \sum_{Q\in\mathcal{K}_{t}}\left\langle
V,h_{\mathcal{S}\left(  Q\right)  }^{\hor}\right\rangle
^{2}\frac{\left\vert Q\right\vert }{\left\vert \mathcal{S}\left(  Q\right)
\right\vert }E_{\mathcal{S}\left(  Q\right)  }U\right\}  +\left\{  \sum
_{Q\in\mathcal{K}_{t}}\frac{1}{\left\vert \mathcal{S}(Q) \right\vert
}\left\langle V,h_{\mathcal{S}\left(  Q\right)  }^{\hor%
}\right\rangle ^{2}\left\langle \left(  R_{1}s_{k_{t+1}}%
^{Q,\hor}\right)  ^{2},\eta_{t+1}\right\rangle \right\}
+o_{k_{t+1}\rightarrow\infty}\left(  1\right) \\
&  \equiv F+G+o_{k_{t+1}\rightarrow\infty}\left(  1\right)
\end{align*}
by Lemma \ref{reduction} part (2) for $k_{t+1}$ sufficiently large, and since
$\frac{1}{\left\vert Q\right\vert }\int_{Q}u_{t}=E_{\mathcal{S}\left(
Q\right)  }U$.

To compute $F$, we pigeonhole the cubes $Q\in\mathcal{K}_{t}$ according to
their supervisors $S=\mathcal{S}\left(  Q\right)  $,
\[
\frac{F}{\bn{n} ^{2}}=\sum_{S\in\mathcal{D}_{t}}\sum_{\substack{Q\in
\mathcal{K}_{t}\\\mathcal{S}\left(  Q\right)  =S}}\left\langle
V,h_{\mathcal{S}\left(  Q\right)  }^{\hor}\right\rangle
^{2}\frac{\left\vert Q\right\vert }{\left\vert \mathcal{S}\left(  Q\right)
\right\vert }E_{\mathcal{S}\left(  Q\right)  }U=\sum_{S\in\mathcal{D}_{t}%
}\left\langle V,h_{S}^{\hor}\right\rangle ^{2}E_{S}%
U\sum_{\substack{Q\in\mathcal{K}_{t}\\\mathcal{S}\left(  Q\right)  =S}%
}\frac{\left\vert Q\right\vert }{\left\vert \mathcal{S}\left(  Q\right)
\right\vert }=\sum_{S\in\mathcal{D}_{t}}\left\langle V,h_{S}%
^{\hor}\right\rangle ^{2}E_{S}U.
\]

However to compute $G$, using the definition $\eta_{t+1}=\sum_{Q\in
\mathcal{K}_{t}}\left\langle U,h_{\mathcal{S}\left(  Q\right)  }%
^{\hor}\right\rangle \frac{1}{\sqrt{\left\vert
\mathcal{S}\left(  Q\right)  \right\vert }}s_{k_{t+1}}%
^{Q,\hor}$, we have%
\begin{align*}
G =\sum_{Q,Q^{\prime}\in\mathcal{K}_{t}}\frac{1}{\left\vert \mathcal{S}\left(
Q\right)  \right\vert }\left\langle V,h_{\mathcal{S}\left(  Q\right)
}^{\hor}\right\rangle ^{2}\left\langle U,h_{\mathcal{S}%
\left(  Q^{\prime}\right)  }^{\hor}\right\rangle
\left\langle \left(  R_{1}s_{k_{t+1}}^{Q,\hor}\right)
^{2},s_{k_{t+1}}^{Q^{\prime},\hor}\right\rangle
=o_{k_{t+1}\rightarrow\infty}\left(  1\right)
\end{align*}
by Theorem \ref{weak van Riesz}, and thus we conclude that the sum of the
diagonal terms equals%
\[
\bn{n} ^{2} \sum_{S\in\mathcal{D}_{t}}\left\langle V,h_{S}%
^{\hor}\right\rangle ^{2}E_{S}U+\sum_{r=0}^{t}%
o_{k_{r+1}\rightarrow\infty}\left(  1\right)  .
\]

Turning now to the sum of the off diagonal terms,%
\[
\sum_{Q\neq Q^{\prime}\in\mathcal{K}_{t}}\frac{1}{\sqrt{\left\vert
\mathcal{S}\left(  Q\right)  \right\vert }}\frac{1}{\sqrt{\left\vert
\mathcal{S}\left(  Q^{\prime}\right)  \right\vert }}\left\langle R_{1}\left[
\left\langle V,h_{\mathcal{S}\left(  Q\right)  }^{\hor%
}\right\rangle s_{k_{t+1}}^{Q,\hor}\right]  R_{1}\left[
\left\langle V,h_{\mathcal{S}\left(  Q^{\prime}\right)  }%
^{\hor}\right\rangle s_{k_{t+1}}^{Q^{\prime
},\hor}\right]  ,u_{t+1}\right\rangle ,
\]
we see that they all tend to $0$ weakly as $k_{t+1}\rightarrow\infty$ by
Theorem \ref{weak van Riesz}. Indeed, we write $u_{t+1}=u_{t}+\eta_{t+1}$, and
split $\eta_{t+1}$ into a linear combination of functions $s_{k_{t+1}%
}^{L,\hor}$, noting that the resulting number of terms in
the above display is independent of $k_{t+1}$ and that each such term tends to
$0$ as $k_{t+1}\rightarrow\infty$ by Theorem \ref{weak van Riesz}. Thus we can
choose the components of the sequence $\left\{  k_{t}\right\}  _{t=1}^{m}$
sufficiently large that%
\[
\int_{\left[  0,1\right]  ^{n}}\left\vert R_{1}v\left(  x\right)  \right\vert
^{2}u\left(  x\right)  dx\geq(\Gamma^{\prime}-CA_{2}^{\operatorname{dyadic}%
}(V,U,\left[  0,1\right]  ^{n}))\int_{\left[  0,1\right]  ^{n}}v\left(
x\right)  dx,
\]

since we also have%
\begin{align*}
&  \int_{\left[  0,1\right]  ^{n}}\left\vert R_{1}v_{0}\left(  x\right)
\right\vert ^{2}u_{0}\left(  x\right)  dx =\int_{\left[  0,1\right]  ^{n}%
}\left\vert R_{1}\mathbf{1}_{\left[  0,1\right]  ^{n}}E_{\left[  0,1\right]
^{n}}V\right\vert ^{2}\mathbf{1}_{\left[  0,1\right]  ^{n}}E_{\left[
0,1\right]  ^{n}}Udx\\
&  =\left(  E_{\left[  0,1\right]  ^{n}}V\right)  ^{2}\left(  E_{\left[
0,1\right]  ^{n}}U\right)  \int_{\left[  0,1\right]  ^{n}}\left\vert
R_{1}\mathbf{1}_{\left[  0,1\right]  ^{n}}\right\vert ^{2}dx =C\left(
E_{\left[  0,1\right]  ^{n}}V\right)  ^{2}\left(  E_{\left[  0,1\right]  ^{n}%
}U\right)  \leq CA_{2} ^{\operatorname{dyadic}} (V,U;\left[  0,1\right]
)E_{\left[  0,1\right]  ^{n}}V.
\end{align*}

Our next task is to show that the two testing conditions for $R_2$ are finite. They are symmetric, so it suffices to show the bound only for the
testing condition with $u$ outside the operator. We will argue so using part (3) of Lemma \ref{reduction} and
Theorem \ref{weak van Riesz}. Let $Q \in \mathcal{D}^0$, and for convenience let $k_0 \equiv 0$. We first consider the case that there exists $t=t\left(  Q\right)  $ such that $2^{-k_0-k_{1}%
-k_{2}-...-k_{t}}\leq\ell\left(  Q\right)  <2^{-k_0-k_{1}-k_{2}-...-k_{t-1}}$. We
will deal later with the remaining cubes $Q$ for which such a $t$ does not exist. Note that at each stage $t$, there are
only finitely many cubes $Q\in\mathcal{D}^{0}$ such that $\ell\left(
Q\right)  \geq2^{-k_0-k_{1}-k_{2}-...-k_{t}}$, and hence will only have to consider finitely many error terms which are $o_{k_{t+1}\rightarrow\infty}\left(  1\right)  $.
Writing $u=u_{t}+\sum_{s=t+1}^{m}\eta_{s}$ and
$v=v_{t}+\sum_{s=t+1}^{m}\delta_{s}$, we then compute
\[
\int_{Q}\left\vert R_{2}%
\mathbf{1}_{Q}v\left(  x\right)  \right\vert ^{2}u\left(  x\right)
dx\lesssim \int_{Q}\left\vert R_{2}%
\mathbf{1}_{Q}\left(  v_{t}\right)  \left(  x\right)  \right\vert
^{2}u\left(  x\right)  dx + \int_{Q}\left\vert R_{2}%
\mathbf{1}_{Q}\left(  \sum\limits_{s=t+1}^{m} \delta_s \right)  \left(  x\right)  \right\vert
^{2}u\left(  x\right)  dx
\]
\[
= \int_{Q}\left\vert R_{2}%
\mathbf{1}_{Q}\left(  v_{t}\right)  \left(  x\right)  \right\vert
^{2}u_t \left(  x\right)  dx + \int_{Q}\left\vert R_{2}%
\mathbf{1}_{Q}\left(  v_{t}\right)  \left(  x\right)  \right\vert
^{2}\left ( \sum\limits_{s=t+1}^m \eta_s \left(  x\right) \right ) dx + \int_{Q}\left\vert R_{2}%
\mathbf{1}_{Q}\left(  \sum\limits_{s=t+1}^{m} \delta_s \right)  \left(  x\right)  \right\vert
^{2}u\left(  x\right)  dx 
\]
\[
\equiv \left\vert Q\right\vert _{v} \left (  \operatorname{main} + \err_1 + \err_2 \right ) \, .
\]
We first claim $\err_2$ can be made arbitrarily small, so long as $k_{t+1}, k_{t+2}, \ldots, k_{m}$ are all  chosen sufficiently large.   Indeed, we use $u\left(  x\right)  \leq\left\Vert U\right\Vert _{\infty}$
independent of the choice of $k_{1},...,k_{m}$, which gives, using part (3) of
Lemma \ref{reduction},
\begin{align*}
\err_2 = \frac{1}{\left\vert Q\right\vert _{v}}\int_{Q}\left\vert R_{2}\mathbf{1}%
_{Q}\left(  \sum_{s=t+1}^{m}\delta_{s}\right)  \left(  x\right)  \right\vert
^{2}u\left(  x\right)  dx  &  \leq\left\Vert U\right\Vert _{\infty}\frac
{1}{\left\vert Q\right\vert _{v}}\int_{Q}\left\vert R_{2}\mathbf{1}_{Q}\left(
\sum_{s=t+1}^{m}\delta_{s}\right)  \left(  x\right)  \right\vert ^{2}dx\\
&  \rightarrow0\text{ as }k_{t+j}\rightarrow\infty,\text{ }j=1,2,...,m-t,
\end{align*}
where we recall that $t=t\left(  Q\right)  $

As for $\err_1$, it too can be made arbitrarily small by choosing $k_{t+1}$
sufficiently large, and using the strong convergence of $\eta_{t+j} \rightarrow0$ in $L^{p}(\mathbb{R}^{n})$ for all $j \geq 1$ by Lemma \ref{reduction} (3), as $R_{2}\mathbf{1}_{Q}%
v_{t}$ only depends on $k_0, k_1, \ldots, k_t$ and is hence independent of $k_{t+j}$ for $j\geq 1$.

So we are left with estimating $\operatorname{main}$. Note now that
$E_{Q}v_{t}=E_{Q}v=E_{\mathcal{S}(Q^{\ast})}V$, where
$Q^{\ast}$ is the unique cube in $\mathcal{K}_{t}$ containing $Q$. Note as
well that $v_{t}$ is constant on each $I\in\mathcal{K}_{t+1}%
$, and satisfies the pointwise estimate 
\[
\mathbf{1}_{Q} (x) v_{t}(x)\leq\left(
E_{\mathcal{S}(Q^{\ast})}V\right)  (1+\tau)
\]
since $v_{t}$
inherits dyadic $\tau$-flatness from $V$; similarly for $u_t$. Then applying the
pointwise estimate to $u_{t}$, followed by the estimate $\Vert R_{2}
\mathbf{1}_{Q} v_{t} \Vert_{L^{2} \left ( \mathbb{R}^n \right )} \leq\Vert\mathbf{1}_{Q}
v_{t} \Vert_{L^{2} \left ( \mathbb{R}^n \right )}$ by boundedness of $R_{2}$, and then the
pointwise estimate applied to $v_{t}$, we get
\begin{align*}
\int\limits_{Q}(R_{2}\mathbf{1}_{Q}v_{t})^{2}%
u_{t}dx \leq(1+\tau)\left(  E_{\mathcal{S}(Q^{\ast}%
)}U\right)  \int\limits_{Q}(R_{2}\mathbf{1}_{Q}v_{t})^{2}dx
&  \leq(1+\tau)\left(  E_{\mathcal{S}(Q^{\ast})}U\right)  \int\limits_{Q}%
(v_{t})^{2}dx\\
&  \leq(1+\tau)^{3}\left(  E_{\mathcal{S}(Q^{\ast})}U\right)  (E_{\mathcal{S}%
(Q^{\ast})}V)^{2}|Q|\,.
\end{align*}
Since $A_{2}^{\operatorname*{dyadic}}(V,U;Q^{0})\leq1$, the above is
controlled by
\[
(1+\tau)^{3}(E_{\mathcal{S} (Q^{\ast})}V)|Q|=(1+\tau)^{3}(E_{Q}v)|Q|=(1+\tau)^{3}\int
_{Q}v\,.
\]

Finally we consider cubes $Q$ for which $t(Q)$ does not exist, i.e., $\ell\left(  Q\right)  <2^{-k_0-k_{1}-k_{2}-...-k_{m}}$. Then
$v,u$ are constant on $Q$ with $E_{Q}v=E_{\mathcal{S}(Q^{\ast})}V$,
$E_{Q}u=E_{\mathcal{S}(Q^{\ast})}U$, where $Q^*$ is the unique cube in $\mathcal{K}_m$ which contains $Q$. Thus
\[
\int\limits_{Q}(R_{2}\mathbf{1}_{Q}v)^{2}u=\left(  E_{\mathcal{S}(Q^{\ast}%
)}V\right)  ^{2}\left(  E_{\mathcal{S}(Q^{\ast})}U\right)  \int\limits_{Q}%
(R_{2}\mathbf{1}_{Q})^{2}\leq\left(  E_{\mathcal{S}(Q^{\ast})}V\right)
\left\vert Q\right\vert =\int_{Q}v,
\]
where in the inequality we use that $\left(  E_{\mathcal{S}(Q^{\ast}%
)}V\right)  \left(  E_{\mathcal{S}(Q^{\ast})}U\right)  \leq1$ and $\Vert
R_{2}\Vert_{L^{2} \left ( \mathbb{R}^n \right )\rightarrow L^{2} \left ( \mathbb{R}^n \right )}=1$.

Since $\tau\in(0,1)$, we obtain that the dual testing constant for $R_{2}$ on
dyadic cubes is bounded; similarly for the testing constant on dyadic cubes.

Finally, to remove the restriction that $J$ and $K$ must be dyadic siblings from \eqref{eq:t-flat_weights}, one can modify the transplantation argument following \cite{NaVo},
as described in Appendix \ref{subsection trans}. However, complete proofs were
not provided in \cite{NaVo} and we invite the reader to consult Appendix
\ref{subsection trans} for missing details, namely Lemma \ref{lem:adjacent doubling}; see also \cite{Naz} and \cite{KaTr}. We also explicitly point out that this modification of transplantation will not affect any of the limiting arguments above involving taking $k_{t}$ sufficiently large for each $t$, and by Remark \ref{rmk:A2_transition}, the dyadic $A_2$ condition will be unaffected.

Finally, by multiplying $v,u$ by an appropriate (small) positive constant, we obtain
the statements in the theorem with the required constants.
\end{proof}

\begin{remark}
The weights $u\left(  x\right)  ,v\left(  x\right)  $ in $\left[  0,1\right]
^{n}$ constructed in the proof of Proposition \ref{Riesz Nazarov}
depend only on the first variable $x_{1}$ of $x$.
\end{remark}

\begin{remark}
A careful reading of the proof shows that our weights $v,u$ satisfy the
$L^{p}$-testing and dual $L^{p}$-testing conditions for the operator $R_{2}$
when $p\in(1,\infty)$. Thus if there was a $T1$ theorem for $L^{p}$ with
doubling weights, our results regarding $R_{2}$ would extend to $L^{p}$. See \cite{SaWi} for a vector-valued $T1$ theorem, where the norm inequality holds if vector-valued analogues of the testing and $A_p$ conditions hold.
\end{remark}

In order to complete the proof of Theorem \ref{stab}, we need to extend our
doubling conclusions to classical doubling, and remove the restriction to
dyadic cubes in our testing conditions for the weight pair $\left(
v,u\right)  $ in Proposition \ref{Riesz Nazarov}.

\subsection{Classical doubling, $A_{2}$ and dyadic testing in $\mathbb{R}^{n}%
$}

By Proposition \ref{Riesz Nazarov}, we have constructed a pair of
weights $\left(  v,u\right)  $ on $Q^{0}=\left[  0,1\right]  ^{n}$, which we
relabel here as $\left(  \sigma,\omega\right)  $, that satisfy the flatness condition \eqref{eq:t-flat_weights} on $Q^0$, the
$A_{2}^{\operatorname{dyadic}}(\sigma,\omega;\left[  0,1\right]  ^{n})$ condition as well as the
\emph{dyadic} testing conditions,%
\begin{align*}
\int_{Q^{0}}\left\vert R_{1}\left(  \mathbf{1}_{Q^{0}}\omega\right)
\right\vert ^{2}d\sigma &  >\Gamma\left\vert Q^{0}\right\vert _{\omega}
\end{align*}
and for all $Q \in \mathcal{D}^0$,
\begin{align*}
\int_{Q}\left\vert R_{2}\left(  \mathbf{1}_{Q}\sigma\right)  \right\vert
^{2}d\omega   \leq\left\vert Q\right\vert _{\sigma
} \, , \quad \int_{Q}\left\vert R_{2}\left(  \mathbf{1}_{Q}\omega\right)  \right\vert
^{2}d\sigma   \leq\left\vert Q\right\vert _{\omega
} \, .
\end{align*}

We now extend these measures to the entire space by reflecting in each
coordinate separately to obtain an extension to $\left[  0,2\right]  ^{n}$,
and then by adding translates $\left[  0,2\right]  ^{n}+2\left(  \alpha
_{1},\alpha_{2},...,\alpha_{n}\right)  $, $\alpha\in\mathbb{Z}^{n}$, so as to
be periodic of period two on the entire space $\mathbb{R}^{n}$.
After this reflection process, note that adjacent cubes from neighbouring dyadic cubes of side
length $1$ also satisfy the adjacent doubling condition, and with constant $1$
since they have equal measures by the reflection extension process, and so for
\emph{any} adjacent dyadic cubes $I_{1}$ and $I_{2}$, we have $\frac{E_{I_{1}}\sigma
}{E_{I_{2}} \sigma} \in(1-\tau, 1+\tau)$, and similarly for $\omega$. In particular, one can show this implies that $\sigma$ and $\omega$ are both $o_{\tau \to 0} (1)$ flat, and hence doubling \cite[Lemma 4.2]{NaVo}. We also note that after this reflection process the pair $(\sigma, \omega)$
satisfies the dyadic $A_{2}$ condition
\begin{equation}\label{eq:dyadic_A2_full_space}
\left ( \frac{1}{\left | Q \right|} \int\limits_{Q} \sigma \left (x \right )\right ) \left ( \frac{1}{\left | Q \right|} \int\limits_{Q} \omega \left (x \right )\right ) \leq 1 \, , \quad \text{ for all dyadic cubes } Q \in \mathcal{D} \, .
\end{equation}
Because $\sigma$ and $\omega$ are doubling, then from \eqref{eq:dyadic_A2_full_space} we obtain that $A_2 (\sigma, \omega) \lesssim 1$. By multiplying $\sigma$ and $\omega$ by an appropriate constant, we may assume without loss of generality that
\[
A_2 (\sigma, \omega) \leq 1 \, .
\]
Furthermore, after this reflection process, the pair $\left ( \sigma, \omega \right )$ also satisfy the \emph{dyadic} testing
conditions for all $\mathcal{D}$-dyadic cubes of side length at most $1$.
 We now
set%
\[
Q_{\alpha}\equiv\left[  0,1\right]  ^{n}+\left(  \alpha_{1},\alpha
_{2},...,\alpha_{n}\right)  ,\ \ \ \ \ \text{for all }\alpha\in\mathbb{Z}%
^{n}.
\]

Let $\tau\in(0,1)$ be as in Proposition \ref{Riesz Nazarov}\ and multiply each
of these measures by the factor
\begin{align*}
\varphi_{\tau}\left(  x\right)   &  \equiv\sum_{\alpha}a_{\alpha}%
\mathbf{1}_{Q_{\alpha}}\left(  x\right)  ,\\
\text{where }a_{\alpha}  &  \equiv\frac{1}{\left\vert Q_{\alpha}\right\vert
}\int_{Q_{\alpha}}d\mu_{\tau}\text{ and }d\mu_{\tau}\left(  x\right)
\equiv\frac{dx}{\left(  1+\left\vert x\right\vert \right)  ^{\tau}},
\end{align*}
and consider the measure pairs $\left(  \sigma_{\tau},\omega_{\tau}\right)  $
with $\sigma_{\tau}\equiv\varphi_{\tau}\left(  x\right)  d\sigma\left(
x\right)  $ and $\omega_{\tau}\equiv\varphi_{\tau}\left(  x\right)
d\omega\left(  x\right)  $. We set $A\equiv\left\vert \left[  0,1\right]
^{n}\right\vert _{\sigma}$ and $B=\left\vert \left[  0,1\right]
^{n}\right\vert _{\omega}$. Note that $A=\left\vert Q_{\alpha}\right\vert
_{\sigma}$ and $B=\left\vert Q_{\alpha}\right\vert _{\omega}$ for all
$\alpha\in\mathbb{Z}^{n}$, and $AB\leq A_{2}(\sigma,\omega) \leq 1$.

\begin{lemma}
The measures $ \sigma_{\tau},\omega_{\tau}$ are both $o_{\tau \to 0} (1)$-flat,
i.e., the adjacent doubling constant of each measure tends to $1$ as $\tau\searrow0$.
\end{lemma}

\begin{proof}
Indeed, if $Q_{\alpha}$ and $Q_{\alpha^{\prime}}$ are two adjacent cubes of
the form $Q_{\alpha}\equiv\left[  0,1\right]  ^{n}+\left(  \alpha_{1}%
,\alpha_{2},...,\alpha_{n}\right)  $, then%
\[
\frac{\int_{Q_{\alpha}}\sigma_{\tau}}{\int_{Q_{\alpha^{\prime}}}\sigma_{\tau}%
}=\frac{a_{\alpha}\int_{Q_{\alpha}}\sigma}{a_{\alpha^{\prime}}\int
_{Q_{\alpha^{\prime}}}\sigma}=\frac{a_{\alpha}A}{a_{\alpha^{\prime}}A}%
=\frac{\int_{Q_{\alpha}}d\mu_{\tau}}{\int_{Q_{\alpha^{\prime}}}d\mu_{\tau}},
\]
tends to $1$ as $\tau\searrow0$ independent of the pair $\left(  Q_{\alpha
},Q_{\alpha^{\prime}}\right)  $ since $\mu_{\tau}$ is a doubling weight on
$\mathbb{R}^{n}$ with adjacent doubling constant roughly $1+O_{\tau
\rightarrow0}(\tau)$. If instead we consider adjacent cubes $P$ and
$P^{\prime}$ that are each a union of cubes $Q_{\alpha}$, then
\[
\frac{\int_{P}\sigma_{\tau}}{\int_{P^{\prime}}\sigma_{\tau}}=\frac
{\sum_{\alpha:Q_{\alpha}\subset P}a_{\alpha}\left\vert Q_{\alpha}\right\vert
_{\sigma}}{\sum_{\alpha^{\prime}:Q_{\alpha^{\prime}}\subset P^{\prime}%
}a_{\alpha^{\prime}}\left\vert Q_{\alpha^{\prime}}\right\vert _{\sigma}}%
=\frac{\sum_{\alpha:Q_{\alpha}\subset P}\int_{Q_{\alpha}}d\mu_{\tau}}%
{\sum_{\alpha^{\prime}:Q_{\alpha^{\prime}}\subset P^{\prime}}\int
_{Q_{\alpha^{\prime}}}d\mu_{\tau}}=\frac{\int_{P}d\mu_{\tau}}{\int_{P^{\prime
}}d\mu_{\tau}},
\]
which again tends to $1$ as $\tau\searrow0$ independent of the pair $\left(
P,P^{\prime}\right)  $. Thus for any adjacent dyadic cubes $I_{1}$ and $I_{2}%
$, we have $\frac{E_{I_{1}}\sigma_{\tau}}{E_{I_{2}} \sigma_{\tau}} \in(1-\tau,
1+\tau)$. A standard argument shows that $\sigma_{\tau}$ has adjacent doubling
constant equal to $1+ o(1)$ as $\tau\searrow0$, and similarly for
$\omega_{\tau}$.
\end{proof}

Next we turn to the final task of establishing the testing conditions for
$R_{2}$ on the doubling measure pair $\left(  \sigma_{\tau},\omega_{\tau
}\right)  $ uniformly for any $\tau\in(0,1)$, which then leads to boundedness
of $R_{2}$ via the main result in \cite[Theorem 2]{SaShUr10} \add{for $\tau >0$ sufficiently small, since if a pair doubling measures with doubling constant sufficiently close to Lebesgue satisfies the $A_2$ condition, then they will satisfy the Energy condition \cite[Section 1.7]{SaShUr10}.}  Of
course, testing fails for $R_{1}$. To state this formally, we will need the definition of a
weighted norm inequality as used in \add{\cite{SaShUr7, SaShUr10}}.

We follow the approach in \cite[see page 314]{SaShUr9}. So we suppose that
$K^{\alpha}$ is a standard smooth $\alpha$-fractional Calder\'{o}n-Zygmund
kernel, and $\sigma,\omega$ are locally finite positive Borel measures on
$\mathbb{R}^{n}$, and we introduce a family $\left\{  \eta_{\delta,R}^{\alpha
}\right\}  _{0<\delta<R<\infty}$ of nonnegative functions on $\left[
0,\infty\right)  $ so that the truncated kernels $K_{\delta,R}^{\alpha}\left(
x,y\right)  =\eta_{\delta,R}^{\alpha}\left(  \left\vert x-y\right\vert
\right)  K^{\alpha}\left(  x,y\right)  $ are bounded with compact support for
fixed $x$ or $y$, and uniformly satisfy the smooth Calder\'{o}n-Zygmund kernel
estimates (\ref{size and smoothness}). Then the truncated operators
\[
T_{\sigma,\delta,R}^{\alpha}f\left(  x\right)  \equiv\int_{\mathbb{R}^{n}%
}K_{\delta,R}^{\alpha}\left(  x,y\right)  f\left(  y\right)  d\sigma\left(
y\right)  ,\ \ \ \ \ x\in\mathbb{R}^{n},
\]
are pointwise well-defined when $f$ is bounded with compact support, and we
will refer to the pair \[
\left(  K^{\alpha},\left\{  \eta_{\delta,R}^{\alpha
}\right\}  _{0<\delta<R<\infty}\right)
\]
as an $\alpha$-fractional singular
integral operator, which we typically denote by $T^{\alpha}$, suppressing the
dependence on the truncations. In the event that $\alpha=0$ and $T^{0}$ is
bounded on unweighted $L^{2}\left(  \mathbb{R}^{n}\right)  $, we say that
$T = T^{0}$ is a Calder\'{o}n-Zygmund operator.

\begin{definition}
\label{def bounded}We say that an $\alpha$-fractional singular integral
operator $T^{\alpha}=\left(  K^{\alpha},\left\{  \eta_{\delta,R}^{\alpha
}\right\}  _{0<\delta<R<\infty}\right)  $ satisfies the norm inequality%
\begin{equation}
\left\Vert T_{\sigma}^{\alpha}f\right\Vert _{L^{2}\left(  \omega\right)  }%
\leq\mathfrak{N}_{T^{\alpha}} \left ( \sigma, \omega \right ) \left\Vert f\right\Vert _{L^{2}\left(
\sigma\right)  },\ \ \ \ \ f\in L^{2}\left(  \sigma\right)  .
\label{two weight'}%
\end{equation}
provided $\mathfrak{N}_{T^{\alpha}} \left ( \sigma, \omega \right )$ is the best constant $\mathfrak{N}$ for which
\[
\left\Vert T_{\sigma,\delta,R}^{\alpha}f\right\Vert _{L^{2}\left(
\omega\right)  }\leq \mathfrak{N}
\left\Vert f\right\Vert _{L^{2}\left(  \sigma\right)  },\ \ \ \ \ f\in
L^{2}\left(  \sigma\right)  ,0<\delta<R<\infty.
\]
\end{definition}

\begin{description}
\item[Independence of Truncations] \label{independence}In the presence of the
classical Muckenhoupt condition $A_{2}^{\alpha}$, the norm inequality
(\ref{two weight'}) is independent of the choice of truncations used,
including \emph{nonsmooth} truncations as well - see \cite[Section 2.1]{LaSaShUr3}.
\end{description}

Now we introduce the testing conditions for Calder\'on-Zygmund operators.

\begin{definition}
Given an $\alpha$-fractional singular integral operator $T^{\alpha
}=\left(  K^{\alpha},\left\{  \eta_{\delta,R}^{\alpha}\right\}  _{0<\delta
<R<\infty}\right)  $, define the testing constants
\begin{align*}
\mathfrak{T}_{T^{\alpha}}\left(  \sigma,\omega\right)  ^{2} \equiv\sup
_{Q\in\mathcal{P}_{n}}\frac{1}{\left\vert Q\right\vert _{\sigma}}\int
_{Q}\left\vert T^{\alpha}\left(  \mathbf{1}_{Q}\sigma\right)  \right\vert
^{2}d\omega\, , \qquad\mathfrak{T}_{T^{\alpha,\ast}}\left(
\omega,\sigma\right)  ^{2} \equiv\sup_{Q\in\mathcal{P}_{n}}\frac{1}{\left\vert
Q\right\vert _{\omega}}\int_{Q}\left\vert T^{\alpha,\ast}\left(
\mathbf{1}_{Q}\omega\right)  \right\vert ^{2}d\sigma.
\end{align*}

We also define the dyadic testing constants by
\begin{align}\label{eq:defn_dyadic_testing}
\mathfrak{T}_{T^{\alpha}} ^{\mathcal{D}} \left(  \sigma,\omega\right)  ^{2} \equiv\sup
_{Q\in\mathcal{D}^0}\frac{1}{\left\vert Q\right\vert _{\sigma}}\int
_{Q}\left\vert T^{\alpha}\left(  \mathbf{1}_{Q}\sigma\right)  \right\vert
^{2}d\omega<\infty\, , \qquad\mathfrak{T}_{T^{\alpha,\ast}} ^{\mathcal{D}^0} \left(
\omega,\sigma\right)  ^{2} \equiv\sup_{Q\in\mathcal{D}}\frac{1}{\left\vert
Q\right\vert _{\omega}}\int_{Q}\left\vert T^{\alpha,\ast}\left(
\mathbf{1}_{Q}\omega\right)  \right\vert ^{2}d\sigma<\infty.
\end{align}
We say $T^{\alpha}$ satisfies the (dyadic) testing conditions if both (dyadic) testing constants for each admissible truncation are finite, and the constants are bounded uniformly over all admissible truncations. 
\end{definition}

The following $T1$ theorem, whose proof we include in Appendix \ref{section:pf_T1}, is a corollary of  \cite[Theorem 2]{SaShUr10}. In particular, if $\tau$ is sufficiently small, it can be applied to the measure pair $(\sigma_{\tau}, \omega_{\tau})$. Recall that a Calder\'on-Zygmund operator $T^{\alpha}$ is $(1+\delta)$-smooth if in addition to having a kernel $K^{\alpha}$ satisfying \eqref{size and smoothness}, we also have
\[
|\nabla K^{\alpha}(x,y) - \nabla K^{\alpha} (x', y)| \leq C_{CZ} \left ( \frac{|x-x'|}{|x-y|} \right )^{\delta} |x-y|^{\alpha-(n+1)}
\]
whenever
\[
\frac{|x-x'|}{|x-y|} \leq \frac{1}{2} \, .
\]
\begin{theorem}[$T1$ theorem for flat doubling measures]\label{thm:starting_T1_thm}
Suppose $\sigma$ and $\omega$ are doubling measures with doubling constant at most $2^{n+\epsilon}$ for some $\epsilon \in (0,1)$, and let $T$ be a $(1+\delta)$-smooth Calder\'on-Zygmund operator of fractional order $0$. Then 
   \[
   \mathfrak{N}_{T} \left ( \sigma, \omega \right ) \lesssim \sqrt{A_2 \left (\sigma, \omega \right )} + \mathfrak{T}_{T} (\sigma, \omega)  + \mathfrak{T}_{T^*} (\omega , \sigma) \, . 
   \] 
\end{theorem}

Finally, we record an estimate from \cite{Saw6} that will be used in proving
the next lemma.

\begin{lemma}
[{\cite[Lemma 23]{Saw6}}]\label{doubling halo}If $\mu$ is a doubling measure and $P$ is a
cube, then for every $\delta\in(0,\frac{1}{2})$ we have%
\[
\left\vert \left\{  x\in P:\operatorname*{dist}\left(  x,\partial P\right)
<\delta\ell\left(  P\right)  \right\}  \right\vert _{\mu}\lesssim\frac{1}%
{\ln\frac{1}{\delta}} \left |P \right |_{\mu}.
\]

\end{lemma}

\begin{lemma}
\label{dyadic_testing_R2} For all $\tau>0$ sufficiently small, the second Riesz transform $R_{2}$
satisfies the norm inequality for the measure pair $\left(  \sigma_{\tau
},\omega_{\tau}\right)  $, i.e.,
\[
\mathfrak{N}_{R_2} \left (\sigma_{\tau}, \omega_{\tau} \right ) \lesssim 1 \, .
\]
\end{lemma}

\begin{proof}
Let $\tau$ be sufficiently small so that the doubling constants for $\sigma$ and $\omega$ are at most $2^{n + \frac{1}{2}}$, so that Theorem \ref{thm:starting_T1_thm} applies. Fix a \emph{dyadic} cube $Q\in\mathcal{D}$. If $Q$ has side length at
most $1$, then $Q$ is contained in one of the cubes $Q_{\alpha}$, where we
have already shown that the testing conditions for $(\sigma,\omega)$ hold in
Proposition \ref{Riesz Nazarov}. In particular we have the following
inequality that will be used repeatedly below,%
\begin{equation}
\int_{Q_{\alpha}}\left\vert R_{2}\left(  \mathbf{1}_{Q_{\alpha}}\sigma_{\tau
}\right)  \right\vert ^{2}d\omega_{\tau}=a_{\alpha}^{3}\int_{Q_{\alpha}%
}\left\vert R_{2}\left(  \mathbf{1}_{Q_{\alpha}}\sigma\right)  \right\vert
^{2}d\omega\leq C_{\ast}a_{\alpha}\left\vert Q_{\alpha}\right\vert _{\sigma
}=C_{\ast}\left\vert Q_{\alpha}\right\vert _{\sigma_{\tau}}.\ \ \ \ \ \alpha
\in\mathbb{Z}^{n}\text{.} \label{testing size 1}%
\end{equation}

So suppose $Q$ has side length $2^{k}$ with $k\geq1$ for some $k\in\mathbb{N}%
$. Then $Q$ is a finite pairwise disjoint union of cubes $Q_{\beta}$, say
$Q=\bigcup_{\beta:\left\vert \beta\right\vert \leq2^{k}}Q_{\beta}$, where
$\left\vert \beta\right\vert \equiv\max\left\{  \beta_{1},\beta_{2}%
,...,\beta_{n}\right\}  $. We will suppose that $Q=\left[  0,2^{k}\right]
^{n}$ as the general case follows the same argument. Finally we note that
$a_{\alpha}\approx\frac{1}{\left(  1+\left\vert \alpha\right\vert \right)
^{\tau}}$. Now we write%
\begin{align}
&  \int_{Q}\left\vert R_{2}\left(  \mathbf{1}_{Q}\sigma_{\tau}\right)
\right\vert ^{2}d\omega_{\tau}=\sum_{\substack{\alpha_1, \alpha_2 , \alpha_3  \in \mathbb{Z}^n \\ 0 \leq |\alpha_j| \leq 2^k}}
\int_{Q_{\alpha_{1}}}R_{2}\left(  \mathbf{1}_{Q_{\alpha_{2}}}%
\sigma_{\tau}\right)  \ R_{2}\left(  \mathbf{1}_{Q_{\alpha_{3}}}\sigma_{\tau
}\right)  d\omega_{\tau}\label{we write}\\
&  \lesssim  \sum_{\substack{\alpha_1, \alpha_2 , \alpha_3  \in \mathbb{Z}^n \\ 0 \leq |\alpha_j| \leq 2^k}}  \int
_{Q_{\alpha_{1}}}\frac{\left\vert R_{2}\left(  \mathbf{1}_{Q_{\alpha_{2}}%
}\sigma\right)  \right\vert \ \left\vert R_{2}\left(  \mathbf{1}%
_{Q_{\alpha_{3}}}\sigma\right)  \right\vert }{\left(  1+\left\vert \alpha
_{2}\right\vert \right)  ^{\tau}\left(  1+\left\vert \alpha_{3}\right\vert
\right)  ^{\tau}}\frac{d\omega}{\left(  1+\left\vert \alpha_{1}\right\vert
\right)  ^{\tau}},\nonumber
\end{align}
We split the sum into several different configurations of $(\alpha_1, \alpha_2, \alpha_3)$, which we consider separately. In what follows, we will \emph{not} specify the configurations considered explicitly within the sum, instead opting to mention in words which configuration we sum over before estimating the sum. 

First assume that we only sum over the configuration of multi-indices $(\alpha_1, \alpha_2, \alpha_3)$ satisfying $\left\vert \alpha_{2}-\alpha_{1}\right\vert \geq2$ and
$\left\vert \alpha_{3}-\alpha_{1}\right\vert \geq2$, so that what we need to
bound is%
\[
\sum_{\substack{\alpha_1, \alpha_2 , \alpha_3  \in \mathbb{Z}^n \\ 0 \leq |\alpha_j| \leq 2^k}} \frac{\left(  1+\left\vert
\alpha_{2}-\alpha_{1}\right\vert \right)  ^{-n}\left(  1+\left\vert \alpha
_{3}-\alpha_{1}\right\vert \right)  ^{-n}}{\left(  1+\left\vert \alpha
_{2}\right\vert \right)  ^{\tau}\left(  1+\left\vert \alpha_{3}\right\vert
\right)  ^{\tau}}\frac{\left\vert Q_{\alpha_{2}}\right\vert _{\sigma
}\left\vert Q_{\alpha_{3}}\right\vert _{\sigma}\left\vert Q_{\alpha_{1}%
}\right\vert _{\omega}}{\left(  1+\left\vert \alpha_{1}\right\vert \right)
^{\tau}},
\]
where we suppress the specified conditions $\left\vert \alpha_{2}-\alpha
_{1}\right\vert \geq2$ and $\left\vert \alpha_{3}-\alpha_{1}\right\vert \geq2$
in the sum. Summing first over $\alpha_{3}$ and using $\left\vert
Q_{\alpha_{3}}\right\vert _{\sigma}=A$ we see that the above term is dominated
by%
\begin{align*}
&  \sum_{\substack{\alpha_1, \alpha_2 , \alpha_3  \in \mathbb{Z}^n \\ 0 \leq |\alpha_j| \leq 2^k}}\frac{\left(  1+\left\vert
\alpha_{2}-\alpha_{1}\right\vert \right)  ^{-n}\left(  1+\left\vert \alpha
_{3}-\alpha_{1}\right\vert \right)  ^{-n}}{\left(  1+\left\vert \alpha
_{2}\right\vert \right)  ^{\tau}\left(  1+\left\vert \alpha_{3}\right\vert
\right)  ^{\tau}}\frac{A\left\vert Q_{\alpha_{2}}\right\vert _{\sigma
}\left\vert Q_{\alpha_{1}}\right\vert _{\omega}}{\left(  1+\left\vert
\alpha_{1}\right\vert \right)  ^{\tau}}\\
&  \leq A\sum_{\substack{\alpha_1, \alpha_2  \in \mathbb{Z}^n \\ 0 \leq |\alpha_j| \leq 2^k}}\left[  \sum_{\substack{\alpha_3  \in \mathbb{Z}^n \\ 0 \leq |\alpha_3| \leq 2^k}}\frac{\left(  1+\left\vert \alpha_{3}-\alpha_{1}\right\vert
\right)  ^{-n}}{\left(  1+\left\vert \alpha_{3}\right\vert \right)  ^{\tau}%
}\right]  \frac{\left(  1+\left\vert \alpha_{2}-\alpha_{1}\right\vert \right)
^{-n}}{\left(  1+\left\vert \alpha_{2}\right\vert \right)  ^{\tau}}%
\frac{\left\vert Q_{\alpha_{2}}\right\vert _{\sigma}\left\vert Q_{\alpha_{1}%
}\right\vert _{\omega}}{\left(  1+\left\vert \alpha_{1}\right\vert \right)
^{\tau}}\\
&  =A\sum_{\substack{\alpha_1, \alpha_2  \in \mathbb{Z}^n \\ 0 \leq |\alpha_j| \leq 2^k}}\left[  \left\{  \sum_{\substack{\alpha_3  \in \mathbb{Z}^n \\ |\alpha_3| < \frac 1 2 |\alpha_1|}}+\sum_{\substack{\alpha_3  \in \mathbb{Z}^n \\  \frac 1 2 |\alpha_1| \leq |\alpha_3| \leq 2 |\alpha_1|}}+\sum_{\substack{\alpha_3  \in \mathbb{Z}^n \\  2 |\alpha_1| < |\alpha_3| }}\right\}  \frac{\left(  1+\left\vert \alpha_{3}%
-\alpha_{1}\right\vert \right)  ^{-n}}{\left(  1+\left\vert \alpha
_{3}\right\vert \right)  ^{\tau}}\right] \\
&  \ \ \ \ \ \ \ \ \ \ \ \ \ \ \ \ \ \ \ \ \ \ \ \ \ \ \ \ \ \ \times
\frac{\left(  1+\left\vert \alpha_{2}-\alpha_{1}\right\vert \right)  ^{-n}%
}{\left(  1+\left\vert \alpha_{2}\right\vert \right)  ^{\tau}}\frac{\left\vert
Q_{\alpha_{2}}\right\vert _{\sigma}\left\vert Q_{\alpha_{1}}\right\vert
_{\omega}}{\left(  1+\left\vert \alpha_{1}\right\vert \right)  ^{\tau}}\\
&  \lesssim A\sum_{\substack{\alpha_1, \alpha_2  \in \mathbb{Z}^n \\ 0 \leq |\alpha_j| \leq 2^k}}\left[  \frac{\ln\left(  2+\left\vert \alpha
_{1}\right\vert \right)  }{\left(  1+\left\vert \alpha_{1}\right\vert \right)
^{\tau}}\right]  \frac{\left(  1+\left\vert \alpha_{2}-\alpha_{1}\right\vert
\right)  ^{-n}}{\left(  1+\left\vert \alpha_{2}\right\vert \right)  ^{\tau}%
}\frac{\left\vert Q_{\alpha_{2}}\right\vert _{\sigma}\left\vert Q_{\alpha_{1}%
}\right\vert _{\omega}}{\left(  1+\left\vert \alpha_{1}\right\vert \right)
^{\tau}}.
\end{align*}
Now summing over $\alpha_{1}$, using that $\left\vert Q_{\alpha_{1}%
}\right\vert _{\omega}=B$ and that $AB \leq A_{2} (\sigma, \omega)$, we obtain
in a similar way that the final line above is at most a constant times%
\begin{align*}
&  A_{2}(\sigma,\omega)\sum_{\substack{\alpha_2  \in \mathbb{Z}^n \\ 0 \leq |\alpha_2| \leq 2^k}}\left[  \frac{\ln\left(  2+\left\vert \alpha_{2}\right\vert \right)
}{\left(  1+\left\vert \alpha_{2}\right\vert \right)  ^{3\tau}}\right]
\left\vert Q_{\alpha_{2}}\right\vert _{\sigma}=A_{2}(\sigma,\omega
)\sum_{\substack{\alpha_2  \in \mathbb{Z}^n \\ 0 \leq |\alpha_2| \leq 2^k}}\left[  \frac{\ln\left(
2+\left\vert \alpha_{2}\right\vert \right)  }{\left(  1+\left\vert \alpha
_{2}\right\vert \right)  ^{2\tau}}\right]  \left\vert Q_{\alpha_{2}%
}\right\vert _{\sigma_{\tau}}\\
&  \leq CA_{2}(\sigma,\omega)\sum_{\substack{\alpha_2  \in \mathbb{Z}^n \\ 0 \leq |\alpha_2| \leq 2^k}}\left\vert Q_{\alpha_{2}}\right\vert _{\sigma_{\tau}}=CA_{2}%
(\sigma,\omega)\left\vert Q\right\vert _{\sigma_{\tau}}\ ,
\end{align*}
where we used that $AB\leq A_{2}(\sigma_{\tau},\omega_{\tau})$.

The relatively simple case we just proved is case (6) in the following
exhaustive list of cases, which we delineate based on the relationship of the
indices $\alpha_{2}$ and $\alpha_{3}$ to the distinguished index $\alpha_{1}%
$:
\begin{align*}
&  (1)\ \alpha_{1}=\alpha_{2}=\alpha_{3}.\\
&  (2)\ \alpha_{1}=\alpha_{2}\text{ and }Q_{\alpha_{1}},Q_{\alpha_{3}}\text{
are separated},\\
&  (3)\ \alpha_{1}=\alpha_{3}\text{ and }Q_{\alpha_{1}},Q_{\alpha_{2}}\text{
are separated},\\
&  (4)\ Q_{\alpha_{1}},Q_{\alpha_{2}}\text{ are adjacent and }Q_{\alpha_{1}%
},Q_{\alpha_{3}}\text{ are separated},\\
&  (5)\ Q_{\alpha_{1}},Q_{\alpha_{3}}\text{ are adjacent and }Q_{\alpha_{1}%
},Q_{\alpha_{2}}\text{ are separated},\\
&  (6)\ Q_{\alpha_{1}},Q_{\alpha_{2}}\text{ are separated and }Q_{\alpha_{1}%
},Q_{\alpha_{3}}\text{ are separated}\\
&  (7)\ \left\{
\begin{array}
[c]{ccc}%
\alpha_{1}=\alpha_{2} & \text{ and } & Q_{\alpha_{1}},Q_{\alpha_{3}}\text{ are
adjacent}\\
\alpha_{1}=\alpha_{3} & \text{ and } & Q_{\alpha_{1}},Q_{\alpha_{2}}\text{ are
adjacent}\\
Q_{\alpha_{1}},Q_{\alpha_{2}}\text{ are adjacent} & \text{ and } &
Q_{\alpha_{1}},Q_{\alpha_{3}}\text{ are adjacent}%
\end{array}
\right.  ,
\end{align*}
where we say that $Q_{\alpha_{1}},Q_{\alpha_{2}}$ are \emph{separated} if
$\left\vert \alpha_{1}-\alpha_{2}\right\vert \geq2$, and of course
$Q_{\alpha_{1}},Q_{\alpha_{2}}$ are adjacent if and only if $\left\vert
\alpha_{1}-\alpha_{2}\right\vert =1$.

In the first of these seven cases, the right side of (\ref{we write})
equals
\[
\sum_{\substack{\alpha  \in \mathbb{Z}^n \\ 0 \leq |\alpha| \leq 2^k}}\int_{Q_{\alpha}}\left\vert
R_{2}\left(  \mathbf{1}_{Q_{\alpha}}\sigma_{\tau}\right)  \right\vert
^{2}d\omega_{\tau}\leq C_{\ast}\sum_{\left\vert \alpha\right\vert =1}^{2^{k}%
}\left\vert Q_{\alpha}\right\vert _{\sigma_{\tau}}=C_{\ast}\left\vert
Q\right\vert _{\sigma_{\tau}},
\]
independent of $\tau\in(0,1)$ by (\ref{testing size 1}).

In the second of these cases, we will use the separation between
$Q_{\alpha_{1}}$ and $Q_{\alpha_{3}}$, as well as the fact that%
\begin{equation}
\label{fact}\begin{aligned} \left\vert \int_{Q_{\alpha_{1}}}R_{2}\left( \mathbf{1}_{Q_{\alpha_{1}}}\sigma_{\tau}\right) d\omega_{\tau}\right\vert & \leq\left( \int _{Q_{\alpha_{1}}}\left\vert R_{2}\left( \mathbf{1}_{Q_{\alpha_{1}}}\sigma_{\tau}\right) \right\vert ^{2}d\omega_{\tau}\right) ^{\frac{1}{2}}\sqrt{\left\vert Q_{\alpha_{1}}\right\vert _{\omega_{\tau}}}\\ & \leq\sqrt{C_{\ast}}\sqrt{\left\vert Q_{\alpha_{1}}\right\vert _{\sigma_{\tau}}}\sqrt{\left\vert Q_{\alpha_{1}}\right\vert _{\omega_{\tau}}} \lesssim \sqrt{C_{\ast}} \frac{A B}{\left ( 1 + \left | \alpha_1 \right| \right )^{\tau}} \end{aligned}
\end{equation}
where the second inequality follows from reasoning using (\ref{testing size 1}%
), similar to the previous display. Thus recalling that $AB\leq A_{2}%
(\sigma,\omega)$, we dominate the right hand side of (\ref{we write}) using
(\ref{fact}) by
\begin{align*}
&  \sum_{\substack{\alpha_1, \alpha_3  \in \mathbb{Z}^n \\ 0 \leq |\alpha_j| \leq 2^k}}\int_{Q_{\alpha_{1}}}\frac{\left\vert R_{2}\left(  \mathbf{1}%
_{Q_{\alpha_{1}}}\sigma\right)  \right\vert \ \left(  1+\left\vert \alpha
_{3}-\alpha_{1}\right\vert \right)  ^{-n}\left\vert Q_{\alpha_{3}}\right\vert
_{\sigma}}{\left(  1+\left\vert \alpha_{1}\right\vert \right)  ^{\tau}\left(
1+\left\vert \alpha_{3}\right\vert \right)  ^{\tau}}\frac{d\omega}{\left(
1+\left\vert \alpha_{1}\right\vert \right)  ^{\tau}}\\
&  \leq A_{2} \left (\sigma, \omega \right ) \sqrt{C_{\ast}}\sum_{\substack{ \alpha_3  \in \mathbb{Z}^n \\ 0 \leq |\alpha_3| \leq 2^k}}\frac{\left\vert Q_{\alpha_{3}}\right\vert _{\sigma}}{\left(  1+\left\vert
\alpha_{3}\right\vert \right)  ^{\tau}}\ \sum_{\left\vert \alpha
_{1}\right\vert =0}^{2^{k}}\frac{\left(  1+\left\vert \alpha_{3}-\alpha
_{1}\right\vert \right)  ^{-n}}{\left(  1+\left\vert \alpha_{1}\right\vert
\right)  ^{\tau}} \\
& \leq A_{2} \left (\sigma, \omega \right )\sqrt{C_{\ast}}\sum_{\substack{ \alpha_3  \in \mathbb{Z}^n \\ 0 \leq |\alpha_3| \leq 2^k}}\frac{\left\vert Q_{\alpha_{3}}\right\vert
_{\sigma}\ln\left(  2+\left\vert \alpha_{3}\right\vert \right)  }{\left(
1+\left\vert \alpha_{3}\right\vert \right)  ^{2\tau}}\leq CA_{2}\sqrt{C_{\ast
}}\left\vert Q\right\vert _{\sigma_{\tau}}.
\end{align*}

To handle the cases where $Q_{\alpha_{1}}$ is adjacent to one of the cubes
$Q_{\alpha_{2}}$ or $Q_{\alpha_{3}}$ or both, we use Lemma
\ref{doubling halo}, i.e., that doubling measures charge halos with reciprocal log
control. Indeed, in the fourth case above, namely $\left\vert \alpha
_{1}-\alpha_{2}\right\vert =1$ and $\left\vert \alpha_{1}-\alpha
_{3}\right\vert \geq2$, we follow the same argument just used except that in
place of the testing condition in (\ref{fact}), we use%
\[
\int_{Q_{\alpha_{1}}}R_{2}\left(  \mathbf{1}_{Q_{\alpha_{2}}}\sigma_{\tau
}\right)  d\omega_{\tau}=\left\{  \int_{\left(  1-\delta\right)  Q_{\alpha
_{1}}}+\int_{Q_{\alpha_{1}}\setminus\left(  1-\delta\right)  Q_{\alpha_{1}}%
}\right\}  R_{2}\left(  \mathbf{1}_{Q_{\alpha_{2}}}\sigma_{\tau}\right)
d\omega_{\tau}\equiv I+II.
\]
We control the first term\ $I$ by $\delta$-separation between $\left(
1-\delta\right)  Q_{\alpha_{1}}$ and $Q_{\alpha_{2}}$:%
\[
\left\vert I\right\vert \leq\int_{\left(  1-\delta\right)  Q_{\alpha_{1}}%
}C\frac{1}{\delta^{n}}\left\vert Q_{\alpha_{2}}\right\vert _{\sigma_{\tau}%
}d\omega_{\tau}\leq C\frac{1}{\delta^{n}}\left\vert Q_{\alpha_{2}}\right\vert
_{\sigma_{\tau}}\left\vert Q_{\alpha_{1}}\right\vert _{\omega_{\tau}}
\]
\[
=C\frac{1}{\delta^{n}}\frac{AB}{\left(  1+\left\vert \alpha_{1}\right\vert
\right)  ^{\tau}\left(  1+\left\vert \alpha_{2}\right\vert \right)  ^{\tau}%
}\leq C\frac{1}{\delta^{n}}\frac{A_{2}}{\left(  1+\left\vert \alpha
_{1}\right\vert \right)  ^{\tau}\left(  1+\left\vert \alpha_{2}\right\vert
\right)  ^{\tau}}.
\]
We control the second term $II$ by using Lemma \ref{doubling halo}:%
\begin{align*}
\left\vert II\right\vert  &  \leq\int_{Q_{\alpha_{1}}\setminus\left(
1-\delta\right)  Q_{\alpha_{1}}}\left\vert R_{2}\left(  \mathbf{1}%
_{Q_{\alpha_{2}}}\sigma_{\tau}\right)  \right\vert d\omega_{\tau}%
\leq\mathfrak{N}_{R_{2}}(\sigma_{\tau},\omega_{\tau})\sqrt{\left\vert
Q_{\alpha_{2}}\right\vert _{\sigma_{\tau}}\left\vert Q_{\alpha_{1}}%
\setminus\left(  1-\delta\right)  Q_{\alpha_{1}}\right\vert _{\omega_{\tau}}%
}\\
&  \leq\frac{C}{\sqrt{\ln\frac{1}{\delta}}}\mathfrak{N}_{R_{2}}(\sigma_{\tau
},\omega_{\tau})\frac{\sqrt{A}\sqrt{B}}{\left(  1+\left\vert \alpha
_{1}\right\vert \right)  ^{\frac{\tau}{2}}\left(  1+\left\vert \alpha
_{2}\right\vert \right)  ^{\frac{\tau}{2}}}\leq\frac{C}{\sqrt{\ln\frac
{1}{\delta}}}\mathfrak{N}_{R_{2}}(\sigma_{\tau},\omega_{\tau})\frac
{\sqrt{A_{2}}}{\left(  1+\left\vert \alpha_{1}\right\vert \right)
^{\frac{\tau}{2}}\left(  1+\left\vert \alpha_{2}\right\vert \right)
^{\frac{\tau}{2}}}.
\end{align*}
Altogether, our replacement for (\ref{fact}) is%
\begin{equation}
\left\vert \int_{Q_{\alpha_{1}}}R_{2}\left(  \mathbf{1}_{Q_{\alpha_{2}}}%
\sigma_{\tau}\right)  d\omega_{\tau}\right\vert \leq\left(  C_{\delta}%
\sqrt{A_{2}}+\frac{C}{\sqrt{\ln\frac{1}{\delta}}}\mathfrak{N}_{R_{2}}%
(\sigma_{\tau},\omega_{\tau})\right)  \frac{\sqrt{A_{2}}}{\left(  1+\left\vert
\alpha_{1}\right\vert \right)  ^{\tau}} \label{log_absorption_R2}%
\end{equation}
since $\left\vert \alpha_{1}-\alpha_{2}\right\vert =1$. Now the previous
argument can continue using (\ref{log_absorption_R2}) in place of
(\ref{fact}), which proves the fourth case since there are just $3^{n}-1$
points $\alpha_{2}$ for each fixed point $\alpha_{1}$. Indeed, we have%
\begin{align*}
&  \sum_{\substack{ \alpha_1, \alpha_3  \in \mathbb{Z}^n \\ 0 \leq |\alpha_j| \leq 2^k}}\int_{Q_{\alpha_{1}}}\frac{\left\vert R_{2}\left(  \mathbf{1}%
_{Q_{\alpha_{1}}}\sigma\right)  \right\vert \ \left(  1+\left\vert \alpha
_{3}-\alpha_{1}\right\vert \right)  ^{-n}\left\vert Q_{\alpha_{3}}\right\vert
_{\sigma}}{\left(  1+\left\vert \alpha_{1}\right\vert \right)  ^{\tau}\left(
1+\left\vert \alpha_{3}\right\vert \right)  ^{\tau}}\frac{d\omega}{\left(
1+\left\vert \alpha_{1}\right\vert \right)  ^{\tau}},\\
&  \leq\left(  C_{\delta}\sqrt{A_{2}}+\frac{C}{\ln\frac{1}{\delta}%
}\mathfrak{N}_{R_{2}}(\sigma_{\tau},\omega_{\tau})\right)  \sum_{\substack{ \alpha_3, \in \mathbb{Z}^n \\ 0 \leq |\alpha_3| \leq 2^k}}\frac{\left\vert Q_{\alpha_{3}}\right\vert
_{\sigma}}{\left(  1+\left\vert \alpha_{3}\right\vert \right)  ^{\tau}}%
\ \sum_{\substack{ \alpha_1 , \in \mathbb{Z}^n \\ 0 \leq |\alpha_1| \leq 2^k}}\frac{\left(  1+\left\vert
\alpha_{3}-\alpha_{1}\right\vert \right)  ^{-n}}{\left(  1+\left\vert
\alpha_{1}\right\vert \right)  ^{\tau}}\\
&  \leq\left(  C_{\delta}\sqrt{A_{2}}+\frac{C}{\ln\frac{1}{\delta}%
}\mathfrak{N}_{R_{2}}(\sigma_{\tau},\omega_{\tau})\right)  \left\vert
Q\right\vert _{\sigma_{\tau}}.
\end{align*}

The third and fifth cases are symmetric to those just handled. So it remains
to consider the remaining seventh case, where one of the following three
subcases holds:%
\[
\left\{
\begin{array}
[c]{ccc}%
\alpha_{1}=\alpha_{2} & \text{ and } & \left\vert \alpha_{1}-\alpha
_{3}\right\vert =1\\
\alpha_{1}=\alpha_{3} & \text{ and } & \left\vert \alpha_{1}-\alpha
_{2}\right\vert =1\\
\left\vert \alpha_{1}-\alpha_{2}\right\vert =1 & \text{ and } & \left\vert
\alpha_{1}-\alpha_{3}\right\vert =1
\end{array}
\right.  .
\]
In all three of these subcases, there is essentially only the sum over
$\alpha_{1}$ since for each fixed $\alpha_{1}$, there are at most $3^{2n}$
pairs $\left(  \alpha_{2},\alpha_{3}\right)  $ satisfying one of the three
subcases. If both $Q_{\alpha_{2}}$ and $Q_{\alpha_{3}}$ are adjacent to
$Q_{\alpha_{1}}$, we write
\begin{align*}
&  \int_{Q_{\alpha_{1}}}R_{2}\left(  \mathbf{1}_{Q_{\alpha_{2}}}\sigma_{\tau
}\right)  \ R_{2}\left(  \mathbf{1}_{Q_{\alpha_{3}}}\sigma_{\tau}\right)
d\omega_{\tau} =\int_{Q_{\alpha_{1}}}R_{2}\left(  \mathbf{1}_{\left(
1-\delta\right)  Q_{\alpha_{2}}}\sigma_{\tau}\right)  R_{2}\left(
\mathbf{1}_{\left(  1-\delta\right)  Q_{\alpha_{3}}}\sigma_{\tau}\right)
d\omega_{\tau}\\
&  +\int_{Q_{\alpha_{1}}}R_{2}\left(  \mathbf{1}_{\left(  1-\delta\right)
Q_{\alpha_{2}}}\sigma_{\tau}\right)  \ R_{2}\left(  \mathbf{1}_{Q_{\alpha_{3}%
}\setminus\left(  1-\delta\right)  Q_{\alpha3}}\sigma_{\tau}\right)
d\omega_{\tau} +\int_{Q_{\alpha_{1}}}R_{2}\left(  \mathbf{1}_{Q_{\alpha_{2}%
}\setminus\left(  1-\delta\right)  Q_{\alpha_{2}}}\sigma_{\tau}\right)
\ R_{2}\left(  \mathbf{1}_{Q_{\alpha_{3}}}\sigma_{\tau}\right)  d\omega_{\tau
}\ .
\end{align*}
The first term of the right-hand side is handled by the $\delta$-separation
between $Q_{\alpha_{1}}$ and $\left(  1-\delta\right)  Q_{\alpha_{2}}$, as
well as between $Q_{\alpha_{1}}$ and $\left(  1-\delta\right)  Q_{\alpha_{3}}%
$, together with the $A_{2}$ condition $AB\leq1$ to obtain%
\[
\left\vert \int_{Q_{\alpha_{1}}}R_{2}\left(  \mathbf{1}_{\left(
1-\delta\right)  Q_{\alpha_{2}}}\sigma_{\tau}\right)  R_{2}\left(
\mathbf{1}_{\left(  1-\delta\right)  Q_{\alpha_{3}}}\sigma_{\tau}\right)
d\omega_{\tau}\right\vert \leq C\frac{1}{\delta^{2n}}\int_{Q_{\alpha_{1}}%
}\left\vert Q_{\alpha_{2}}\right\vert _{\sigma_{\tau}}\left\vert Q_{\alpha
_{3}}\right\vert _{\sigma_{\tau}}d\omega_{\tau}
\]
\[C\frac{1}{\delta^{2n}%
}\left\vert Q_{\alpha_{3}}\right\vert _{\sigma_{\tau}}AB\leq C\frac{1}%
{\delta^{2n}}\left\vert Q_{\alpha_{3}}\right\vert _{\sigma}A_{2} \left ( \sigma, \omega \right ) %
\]
and since for each fixed $\alpha_{3}$, there are at most $3^{2n}$ pairs
$\left(  \alpha_{1},\alpha_{2}\right)  $, we can sum to obtain the bound
$C\frac{1}{\delta^{2n}}\left\vert Q\right\vert _{\sigma_{\tau}}$.

To handle the terms involving a halo $Q_{\alpha_{j}}\setminus\left(
1-\delta\right)  Q_{\alpha_{j}}$ we use Lemma \ref{doubling halo} together
with the norm constant $\mathfrak{N}_{R_{2}}=\mathfrak{N}_{R_{2}}(\sigma
_{\tau},\omega_{\tau})$. For example,%
\begin{align*}
&  \left\vert \int_{Q_{\alpha_{1}}}R_{2}\left(  \mathbf{1}_{Q_{\alpha_{2}%
}\setminus\left(  1-\delta\right)  Q_{\alpha_{2}}}\sigma_{\tau}\right)
\ R_{2}\left(  \mathbf{1}_{Q_{\alpha_{3}}}\sigma_{\tau}\right)  d\omega_{\tau
}\right\vert \\
&  \leq\left(  \int_{Q_{\alpha_{1}}}\left\vert R_{2}\left(  \mathbf{1}%
_{Q_{\alpha_{2}}\setminus\left(  1-\delta\right)  Q_{\alpha_{2}}}\sigma_{\tau
}\right)  \right\vert ^{2}d\omega_{\tau}\right)  ^{\frac{1}{2}}\left(
\int_{Q_{\alpha_{1}}}\left\vert R_{2}\left(  \mathbf{1}_{Q_{\alpha_{3}}}%
\sigma_{\tau}\right)  \right\vert ^{2}d\omega_{\tau}\right)  ^{\frac{1}{2}}\\
&  \leq\mathfrak{N}_{R_{2}}\sqrt{\left\vert Q_{\alpha_{2}}\setminus\left(
1-\delta\right)  Q_{\alpha_{2}}\right\vert _{\sigma_{\tau}}}\mathfrak{N}%
_{R_{2}}\left(  \left\vert Q_{\alpha_{3}}\right\vert _{\sigma_{\tau}}\right)
^{\frac{1}{2}}=\left(  \mathfrak{N}_{R_{2}}\right)  ^{2}\frac{C}{\sqrt
{\ln\frac{1}{\delta}}}\sqrt{\left\vert Q_{\alpha_{2}}\right\vert
_{\sigma_{\tau}}\left\vert Q_{\alpha_{3}}\right\vert _{\sigma_{\tau}}},
\end{align*}
and again we can sum to obtain the bound $\left(  \mathfrak{N}_{R_{2}}\right)
^{2}\frac{C}{\ln\frac{1}{\delta}}\left\vert Q\right\vert _{\sigma_{\tau}}$
because the indices $\alpha_{j}$ are at distance one from each other. The other
terms are handled similarly and we thus obtain in this seventh case that%
\[
\sum_{\left\vert \alpha_{1}\right\vert ,\left\vert \alpha_{2}\right\vert
,\left\vert \alpha_{3}\right\vert =0}^{2^{k}}\left\vert \int_{Q_{\alpha_{1}}%
}R_{2}\left(  \mathbf{1}_{Q_{\alpha_{2}}}\sigma_{\tau}\right)  \ R_{2}\left(
\mathbf{1}_{Q_{\alpha_{3}}}\sigma_{\tau}\right)  d\omega_{\tau}\right\vert
\leq C\left(  \frac{1}{\delta^{2n}}A_{2}+\frac{\left(  \mathfrak{N}_{R_{2}%
}\right)  ^{2}}{\sqrt{\ln\frac{1}{\delta}}}\right)  \left\vert Q\right\vert
_{\sigma_{\tau}}%
\]

The cases where just one of the cubes is adjacent to $Q_{\alpha_{1}}$ are
handled similarly. Altogether we now have%
\[
\mathfrak{T}_{R_{2}}^{\mathcal{D}}\left(  \sigma_{\tau},\omega_{\tau}\right)
^{2}\equiv\sup_{Q\in\mathcal{D}}\frac{1}{\left\vert Q\right\vert
_{\sigma_{\tau}}}\int_{Q}\left\vert R_{2}\left(  \mathbf{1}_{Q}\sigma_{\tau
}\right)  \right\vert ^{2}d\omega_{\tau}\leq C_{\tau}+\frac{1}{\delta_{1}%
^{2n}}A_{2}+\frac{\left(  \mathfrak{N}_{R_{2}}\right)  ^{2}}{\sqrt{\ln\frac
{1}{\delta_{1}}}},
\]
for any choice of $\delta_{1}\in(0,1)$, where the constant $C_{\ast}$ arises
in (\ref{testing size 1}).

Now we turn to the case of a general cube $Q$. In this case we first fix
$M\in\mathbb{N}$ large to be chosen later, and write $Q$ as a union of roughly
$2^{Mn}$ dyadic subcubes $\left\{  Q_{\alpha}\right\}  _{\alpha}$ of side
length $\delta_{2}\equiv\frac{\ell\left(  Q\right)  }{2^{M}}>0$, in such a way
that the remaining portion of $Q$ is contained in the $5\delta_{2}$-halo of
$Q$. Then the above argument shows that the testing condition holds except for
the terms that arise from the halo. But by Lemma \ref{doubling halo} these
leftover terms in $\left(  \int_{Q}\left\vert R_{2}\left(  \mathbf{1}%
_{Q}\sigma_{\tau}\right)  \right\vert ^{2}\ d\omega_{\tau}\right)  ^{\frac
{1}{2}}$ are dominated by $C\frac{1}{\sqrt[4]{\ln\frac{1}{\delta_{2}}}%
}\mathfrak{N}_{R_{2}}\sqrt{\left\vert Q\right\vert _{\sigma}}$, so that
altogether we obtain that the continuous testing constant satisfies%
\begin{align}
\mathfrak{T}_{R_{2}}\left(  \sigma_{\tau},\omega_{\tau}\right)   &  \leq
C_{\delta_{2}}\mathfrak{T}_{R_{2}}^{\mathcal{D}}\left(  \sigma_{\tau}%
,\omega_{\tau}\right)  +C\frac{1}{\sqrt[4]{\ln\frac{1}{\delta_{2}}}%
}\mathfrak{N}_{R_{2}}\left(  \sigma_{\tau},\omega_{\tau}\right)
\label{pre testing}\\
&  +C_{\delta_{2}}\left(  C_{\ast}+C_{\ast}C_{\tau}+\frac{1}{\delta_{1}^{2n}%
}A_{2}+\frac{\left(  \mathfrak{N}_{R_{2}}\right)  ^{2}}{\sqrt{\ln\frac
{1}{\delta_{1}}}}\right)  ^{\frac{1}{2}}+C\frac{\mathfrak{N}_{R_{2}}\left(
\sigma_{\tau},\omega_{\tau}\right)  }{\sqrt[4]{\ln\frac{1}{\delta_{2}}}%
}\nonumber\\
&  \leq C_{\delta_{2},\tau}\sqrt{C_{\ast}}+C_{\delta_{2}}\frac{1}{\delta
_{1}^{n}}\sqrt{A_{2}}+\left(  \frac{C_{\delta_{2}}}{\sqrt[4]{\ln\frac
{1}{\delta_{1}}}}+\frac{C}{\sqrt[4]{\ln\frac{1}{\delta_{2}}}}\right)
\mathfrak{N}_{R_{2}}\left(  \sigma_{\tau},\omega_{\tau}\right)  \ .\nonumber
\end{align}
Note that the two weight norm $\mathfrak{N}_{R_2} \left (\sigma_{\tau}, \omega_{\tau} \right )$ is finite, as both weights $\sigma_{\tau},\omega_{\tau}$ are bounded step functions,
and so by the boundedness of the principal value interpretation of $R_{2}$ on
Lebesgue spaces, we have%
\[
\mathfrak{N}_{R_{2}}\left(  \sigma_{\tau},\omega_{\tau}\right)  \leq\left\Vert
\sigma_{\tau}\right\Vert _{\infty}\left\Vert \omega_{\tau}\right\Vert
_{\infty} < \infty.
\]
Thus by boundedness of maximal truncations (see e.g. \cite[Proposition 1 page
31]{Ste2}) together with the independence of truncations mentioned above, the
above arguments actually prove that (\ref{pre testing}) holds \emph{uniformly}
over all admissible truncations of $R_{2}$\rmv{ to check}, which is the hypothesis used in
\cite{SaShUr7, SaShUr10, AlSaUr}. Thus noting Definition \ref{def bounded}, we can apply Theorem \ref{thm:starting_T1_thm} to obtain%
\begin{align*}
&  \mathfrak{N}_{R_{2}}\left(  \sigma_{\tau},\omega_{\tau}\right)  \leq
C\sqrt{A_{2}\left(  \sigma_{\tau},\omega_{\tau}\right)  }+C\mathfrak{T}%
_{R_{2}}\left(  \sigma_{\tau},\omega_{\tau}\right)  +C\mathfrak{T}_{R_{2}%
}\left(  \omega_{\tau},\sigma_{\tau}\right) \\
&  \leq C\sqrt{A_{2}\left(  \sigma_{\tau},\omega_{\tau}\right)  }+2\left\{
C_{\delta_{2},\tau}\sqrt{C_{\ast}}+C_{\delta_{2}}\frac{1}{\delta_{1}^{2}}%
\sqrt{A_{2}\left(  \sigma_{\tau},\omega_{\tau}\right)  }+\left(
\frac{C_{\delta_{2}}}{\sqrt{\ln\frac{1}{\delta_{1}}}}+\frac{C}{\sqrt{\ln
\frac{1}{\delta_{2}}}}\right)  \mathfrak{N}_{R_{2}}\left(  \sigma_{\tau
},\omega_{\tau}\right)  \right\}  ,
\end{align*}
for any admissible truncation of $R_{2}$. Thus with $\delta_{2}>0$ chosen
sufficiently small that $\frac{C}{\sqrt{\ln\frac{1}{\delta_{2}}}}<\frac{1}{4}%
$, and then $\delta_{1}>0$ chosen sufficiently small that $\frac{C_{\delta
_{2}}}{\sqrt{\ln\frac{1}{\delta_{1}}}}<\frac{1}{4}$, an absorption completes
the proof that the norm inequality for $R_{2}$ holds (recall that truncations
of $R_{2}$ are \emph{a priori} bounded).
\end{proof}

We have thus proved the following special case of Theorem \ref{stab} for the
individual Riesz transforms $R_{1}$ and $R_{2}$.

\begin{proposition}
\label{individual}For every $\Gamma>1$ and $0<\tau<1$, there is a pair of
positive weights $\left(  \sigma,\omega\right)  $ in $\mathbb{R}^{n}$
satisfying%
\begin{align*}
&  \int_{\mathbb{R}^{n}}\left\vert R_{1}\left(  \mathbf{1}_{\left[
0,1\right]  ^{n}}\sigma\right)  \left(  x\right)  \right\vert ^{2}\left(
x\right)  d\omega\left(  x\right)  \geq\Gamma\int_{\left[  0,1\right]  ^{n}%
}d\sigma\left(  x\right)  ,\\
&  \int_{I}\left\vert R_{2}\mathbf{1}_{I}\sigma\left(  x\right)  \right\vert
^{2}d\omega\left(  x\right)  \leq\int_{I}d\sigma\left(  x\right)
,\ \ \ \ \ \text{for all cubes }I\in\mathcal{P}^{n},\\
&  \int_{I}\left\vert R_{2}\mathbf{1}_{I}\omega\left(  x\right)  \right\vert
^{2}d\sigma\left(  x\right)  \leq\int_{I}d\omega\left(  x\right)
,\ \ \ \ \ \text{for all cubes }I\in\mathcal{P}^{n},\\
&  \left(  \frac{1}{\left\vert I\right\vert }\int_{I}d\sigma\right)  \left(
\frac{1}{\left\vert I\right\vert }\int_{I}d\omega\right)  \leq
1,\ \ \ \ \ \text{for all cubes }I\in\mathcal{P}^{n},\\
&  1-\tau<\frac{E_{J}\sigma}{E_{K}\sigma},\frac{E_{J}\omega}{E_{K}\omega
}<1+\tau,\ \ \ \ \ \text{for arbitrary adjacent cubes }J,K \in \mathcal{P}^{n}.
\end{align*}

\end{proposition}

The argument used in proving this proposition also shows that in any two weight $T1$ theorem for doubling pairs $(\sigma, \omega)$, the testing may be carried out over only cubes in
any fixed dyadic grid $\mathcal{D}$, and here is one possible formulation of
this improvement.

\begin{theorem}
\label{single grid}Suppose $0\leq\alpha<n$, and let $T^{\alpha}$ be an
$\alpha$-fractional Calder\'{o}n-Zygmund singular integral operator on
$\mathbb{R}^{n}$ with a smooth $\alpha$-fractional kernel $K^{\alpha}$. Assume
that $\sigma$ and $\omega$ are doubling measures on $\mathbb{R}^{n}$. Finally fix a dyadic grid $\mathcal{D}$ on
$\mathbb{R}^{n}$.

If the two weight norm $\mathfrak{N}_{T^{\alpha}} (\sigma, \omega)$
satisfies
\[
\mathfrak{N}_{T^{\alpha}} \left (\sigma, \omega \right ) \leq C_{\alpha,n}\left(  \sqrt{A_{2}^{\alpha}%
}+\mathfrak{T}_{T^{\alpha}}+\mathfrak{T}_{\left(  T^{\alpha
}\right)  ^{\ast}}\right) \, ,
\]
where $A_{2}^{\alpha}$ is the classical Muckenhoupt constant and the constant $C_{\alpha,n}$ depends on the Calder\'{o}n-Zygmund kernel and the doubling constants of the measures $\sigma, \omega$,
then
\begin{equation}
\mathfrak{N}_{T^{\alpha}}\leq C_{\alpha,n} ' \left(  \sqrt{A_{2}^{\alpha}%
}+\mathfrak{T}_{T^{\alpha}}^{\mathcal{D}}+\mathfrak{T}_{\left(  T^{\alpha
}\right)  ^{\ast}}^{\mathcal{D}}\right)  , \label{bound}%
\end{equation}
where the constant $C_{\alpha,n}'$ also depends on the Calder\'{o}n-Zygmund
kernel and the doubling constants of $\sigma$ and $\omega$, and
$\mathfrak{T}_{T^{\alpha}}^{\mathcal{D}},\mathfrak{T}_{\left(  T^{\alpha
}\right)  ^{\ast}}^{\mathcal{D}}$ are the $\mathcal{D}$-dyadic testing constants.
\end{theorem}

In order to complete the proof of Theorem \ref{stab}, we need to consider
iterated Riesz transforms.

\section{Iterated Riesz transforms}

\label{section:iterated_Riesz}

Throughout Section \ref{section:action_of_Riesz} and \ref{section:full_proof},
we considered Riesz transforms of order $1$. However our results extend to
arbitrary iterated Riesz transforms of odd order in $\mathbb{R}^{n}$. We will
extend the results of Section \ref{section:action_of_Riesz} to their
appropriate analogues to make the reasoning of Section
\ref{section:full_proof} hold for the appropriate iterated Riesz transforms,
and we begin by establishing the following theorem.

\begin{theorem}
\label{pure odd}The odd order pure iterated Riesz transforms $R_{1}^{2m+1}$
are unstable on $\mathbb{R}^{n}$ for pairs of doubling measures under
$90^{\circ}$ rotations in any coordinate plane. In fact, there exists a
measure pair of doubling measures on which $R_{1}^{2m+1}$ is unbounded, and
all iterated Riesz transforms of order $2m+1$ that are \emph{not} a pure power
of $R_{1}$, are bounded.
\end{theorem}

\begin{proof}
Recall the notation $T_{\sigma} f = T(f \sigma)$. We begin first by
considering Riesz transforms of arbitrary order, even or odd. Using the
identity
\begin{equation}
R_{1}^{2}+\ldots+R_{n}^{2}=-I\,, \label{Riesz identity}%
\end{equation}
and $N\geq2$ we have for an arbitrary positive measure $\sigma$ that
\[
\left(  R_{1}^{N}\right)  _{\sigma}=\left(  R_{1}^{N-2}R_{1}^{2}\right)
_{\sigma}=-\left(  R_{1}^{N-2}\right)  _{\sigma}-\sum\limits_{j=2}^{n}\left(
R_{1}^{N-2}R_{j}^{2}\right)  _{\sigma}.
\]
Iteration then yields for $N\geq1$,
\begin{equation}
\left(  R_{1}^{N}\right)  _{\sigma}=\left\{
\begin{array}
[c]{ccc}%
\pm I_{\sigma}+\sum\limits_{k=0}^{m}\left[  \pm\sum\limits_{j=2}^{n}\left(
R_{1}^{N-2k}R_{j}^{2}\right)  _{\sigma}\right]  & \text{ if } & N=2m\text{ is
even}\\
\pm\left(  R_{1}\right)  _{\sigma}+\sum\limits_{k=0}^{m}\left[  \pm\sum
\limits_{j=2}^{n}\left(  R_{1}^{N-2k}R_{j}^{2}\right)  _{\sigma}\right]  &
\text{ if } & N=2m+1\text{ is odd}%
\end{array}
\right.  . \label{red form}%
\end{equation}

For the weight pairs $\left(  \sigma_{\tau},\omega_{\tau}\right)  $
constructed in Section \ref{section:full_proof}, and with $N=2m+1$ odd, the
second line in (\ref{red form}) yields%
\begin{align*}
\left\Vert \left(  R_{1}^{N}\right)  _{\sigma_{\tau}}\right\Vert
_{L^{2}\left(  \sigma_{\tau}\right)  \rightarrow L^{2}\left(  \omega_{\tau
}\right)  }  &  \geq\left\Vert \left(  R_{1}\right)  _{\sigma_{\tau}%
}\right\Vert _{L^{2}\left(  \sigma_{\tau}\right)  \rightarrow L^{2}\left(
\omega_{\tau}\right)  }-\sum_{k=0}^{m}\sum\limits_{j=2}^{n}\left\Vert \left(
R_{1}^{N-2k}R_{j}^{2}\right)  _{\sigma_{\tau}}\right\Vert _{L^{2}\left(
\sigma_{\tau}\right)  \rightarrow L^{2}\left(  \omega_{\tau}\right)  }\\
&  \geq\Gamma-\sum_{k=0}^{m}\sum\limits_{j=2}^{n}\left\Vert \left(
R_{1}^{N-2k}R_{j}^{2}\right)  _{\sigma_{\tau}}\right\Vert _{L^{2}\left(
\sigma_{\tau}\right)  \rightarrow L^{2}\left(  \omega_{\tau}\right)  },
\end{align*}
where $\Gamma$ is the constant in the construction of the weight pair $\left(
\sigma_{\tau},\omega_{\tau}\right)  $. Note that the operator norm dominates
the testing constant, which was shown to exceed $\Gamma$.

We now claim that the double sum of the operator norms on the right side is
bounded independently of $\Gamma$, i.e.,
\[
\sum_{k=0}^{m}\sum\limits_{j=2}^{n}\left\Vert \left(  R_{1}^{N-2k}R_{j}%
^{2}\right)  _{\sigma_{\tau}}\right\Vert _{L^{2}\left(  \sigma_{\tau}\right)
\rightarrow L^{2}\left(  \omega_{\tau}\right)  }=O\left(  1\right)  .
\]
In fact if $j\geq2$ and $R^{\alpha}=R_{1}^{\alpha_{1}}R_{2}^{\alpha_{2}%
}...R_{n}^{\alpha_{n}}$ with $\alpha_{j}>0$, then by Lemma \ref{reduction}
part (3),%
\begin{align*}
&  \limsup_{k\rightarrow\infty}\int\left\vert R_{j}R^{\alpha}s_{k}%
^{P,\hor}\left(  x\right)  \right\vert ^{2}%
dx=\limsup_{k\rightarrow\infty}\left\vert \int\left(  R_{j}s_{k}%
^{P,\hor}\right)  \left(  x\right)  \left(
R_{j}R^{2\alpha}s_{k}^{P,\hor}\right)  \left(  x\right)
dx\right\vert \\
&  \leq\sqrt{\limsup_{k\rightarrow\infty}\int\left\vert R_{j}s_{k}%
^{P,\hor}\left(  x\right)  \right\vert ^{2}dx}%
\sqrt{\limsup_{k\rightarrow\infty}\int\left\vert R_{j}R^{2\alpha}%
s_{k}^{P,\hor}\left(  x\right)  \right\vert ^{2}dx}\\
&  \leq\sqrt{\limsup_{k\rightarrow\infty}\int\left\vert R_{j}s_{k}%
^{P,\hor}\left(  x\right)  \right\vert ^{2}dx}\left\Vert
R_{j}R^{2\alpha}\right\Vert _{L^{2}\left(  \mathbb{R}^{n}\right)  \rightarrow
L^{2}\left(  \mathbb{R}^{n}\right)  }\sqrt{\left | P \right | }=0,\ \ \ \ \ \text{for all }%
N\in\mathbb{N}.
\end{align*}
Therefore the reasoning in Proposition \ref{Riesz Nazarov} and Lemma
\ref{dyadic_testing_R2} shows that iterated Riesz transforms of order $N$
which are \emph{not} pure powers of $R_{1}$ have dyadic testing constants on
the weight pairs $\left(  \sigma_{\tau},\omega_{\tau}\right)  $ that are
$O\left(  1\right)  $. Then Theorem \ref{single grid} shows that the operator
norms of such operators, including $R_{1}^{N-2k}R_{j}^{2}$, are $O\left(
1\right)  $, which proves our claim, and completes the proof of the second
assertion of the theorem. The first assertion regarding $R_{1}^{2m+1}$ now
follows from the fact that a rotation in the $\left(  x_{1},x_{j}\right)
$-plane interchanges $R_{1}^{2m+1}$ and $R_{j}^{2m+1}$.
\end{proof}

The key to our proof of Theorem \ref{pure odd} is the construction of weight
pairs $\left(  \sigma_{\tau},\omega_{\tau}\right)  $ satisfying the inequality%
\begin{equation}
\Vert\left(  R_{1}^{N}\right)  _{\sigma_{\tau}}\Vert_{L^{2}(\sigma_{\tau
})\rightarrow L^{2}(\omega_{\tau})}\geq\Gamma\text{ for }\Gamma\text{
arbitrarily large,} \label{key ineq}%
\end{equation}
when $N$ is odd. In fact, the inequality (\ref{key ineq}) actually
\emph{fails} for the weight pairs we construct when $N$ is even. Indeed, from
the first line in (\ref{red form}), and the fact that the proof of Theorem
\ref{pure odd} shows that
\[
\sum_{k=0}^{m}\sum\limits_{j=2}^{n}\left\Vert \left(  R_{1}^{N-2k}R_{j}%
^{2}\right)  _{\sigma_{\tau}}\right\Vert _{L^{2}\left(  \sigma_{\tau}\right)
\rightarrow L^{2}\left(  \omega_{\tau}\right)  } = O(1) \, ,
\]
we get
\[
\left\Vert \left(  R_{1}^{N}\right)  _{\sigma_{\tau}}\right\Vert
_{L^{2}\left(  \sigma_{\tau}\right)  \rightarrow L^{2}\left(  \omega_{\tau
}\right)  }\leq\left\Vert I_{\sigma_{\tau}}\right\Vert _{L^{2}\left(
\sigma_{\tau}\right)  \rightarrow L^{2}\left(  \omega_{\tau}\right)  }+ O(1).
\]
The right hand side of the above display is bounded since the operator norm of
$I_{\sigma_{\tau}}$ is bounded by $A_{2}(\sigma_{\tau},\omega_{\tau})$:
indeed, when $\sigma$ and $\omega$ are weights, we have $\left\Vert \sigma
\omega\right\Vert _{\infty}\leq A_{2}(\sigma,\omega)$ by the Lebesgue
differentiation theorem, and so%
\[
\Vert I_{\sigma}f\Vert_{L^{2}(\omega)}^{2}=\int_{\mathbb{R}^{n}}f^{2}%
\sigma^{2}\omega\leq A_{2}(\sigma,\omega)\int_{\mathbb{R}^{n}}f^{2}%
\sigma=\Vert f\Vert_{L^{2}(\sigma)}^{2}\ .
\]
Moreover, it is easily shown that $\left\Vert I_{\sigma}\right\Vert
_{L^{2}(\sigma)\rightarrow L^{2}(\omega)}=A_{2}(\sigma,\omega)$ for arbitrary
weights $\sigma$ and $\omega$. Thus $R_{1} ^{N}$ must then satisfy the testing
conditions for the measure pair $(\sigma, \omega)$.

In the next subsection we show that every odd order iterated Riesz transform
$R^{\beta}=R_{1}^{\beta_{1}}R_{2}^{\beta_{2}}...R_{n}^{\beta_{n}}$ is unstable
under rotations, by showing that $R_{1}^{\beta_{1}}R_{2}^{\beta_{2}}%
...R_{n}^{\beta_{n}}$ is some rotation of $R^{\left(  \left\vert
\beta\right\vert ,0,...,0\right)  }$ whenever $\beta\neq|\beta| e_{k}$ for
some $k$. When $\beta= |\beta| e_{k}$ some $k$, then we may assume without
loss of generality that $k = 2$.

\subsection{Rotations}

Let $\beta$ be a multi-index of length $\left\vert \beta\right\vert =N$. The
symbol of the iterated Riesz transform $R^{\beta}=R_{1}^{\beta_{1}}%
R_{2}^{\beta_{2}}...R_{n}^{\beta_{n}}$ is $i^{N}\frac{\xi_{1}^{\beta_{1}}%
\xi_{2}^{\beta_{2}}...\xi_{n}^{\beta_{n}}}{\left\vert \xi\right\vert ^{N}}$.
We already know that $R^{\left(  N,0,...,0\right)  }$ is unstable, and the
following lemma will be used to show all $R^{\beta}$ are unstable.

\begin{lemma}
\label{rotate_poly} If $P\left(  \xi\right)  $ is a nontrivial homogeneous
polynomial of degree $N$ that doesn't contain the monomial $\xi_{1}^{N}$, then
there is a set of rotations of full-measure $\Lambda$, and for any rotation
$\Theta\in\Lambda$, we have $\xi=\Theta\eta$ is such that $P\left(  \Theta
\eta\right)  $ contains the monomial $\eta_{1}^{N}$.
\end{lemma}

\begin{proof}
In dimension $n=2$, we have
\[
P\left(  \xi_{1},\xi_{2}\right)  =\sum_{m=1}^{N}c_{m}\xi_{1}^{m}\xi_{2}%
^{N-m},\ \ \ \ \ \text{where not all }c_{m}=0,
\]
and the restriction of this polynomial to the unit circle cannot vanish
identically (otherwise $P$ itself would vanish identically by homogeneity, a
contradiction). Thus there is $\theta\in\left[  0,2\pi\right)  $ such that%
\[
0\neq P\left(  \cos\theta,\sin\theta\right)  =\sum_{m=1}^{N}c_{m}\cos
^{m}\theta\sin^{N-m}\theta.
\]
However, if we make the rotational change of variable, i.e.
\[
\left(
\begin{array}
[c]{c}%
\xi_{1}\\
\xi_{2}%
\end{array}
\right)  =\left[
\begin{array}
[c]{cc}%
\cos\theta & -\sin\theta\\
\sin\theta & \cos\theta
\end{array}
\right]  \left(
\begin{array}
[c]{c}%
\eta_{1}\\
\eta_{2}%
\end{array}
\right)  =\left(
\begin{array}
[c]{c}%
\eta_{1}\cos\theta-\eta_{2}\sin\theta\\
\eta_{1}\sin\theta+\eta_{2}\cos\theta
\end{array}
\right)  ,
\]
then%
\begin{align*}
P\left(  \xi_{1},\xi_{2}\right)   &  =\sum_{m=1}^{N}c_{m}\xi_{1}^{m}\xi
_{2}^{N-m}=\sum_{m=1}^{N}c_{m}\left(  \eta_{1}\cos\theta-\eta_{2}\sin
\theta\right)  ^{m}\left(  \eta_{1}\sin\theta+\eta_{2}\cos\theta\right)
^{N-m}\\
&  =\eta_{1}^{N}\sum_{m=1}^{N}c_{m}\cos^{m}\theta\sin^{N-m}\theta+\sum
_{\beta\neq\mathbf{e}_{1}:\ \left\vert \beta\right\vert =N}\eta^{\beta
}f_{\beta}\left(  \theta\right)
\end{align*}
where $\sum_{m=1}^{N}c_{m}\cos^{m}\theta\sin^{N-m}\theta\neq0$. The case
$n\geq3$\ is similar.
\end{proof}

\subsection{Completion of proofs of main Theorems \ref{stab} and \ref{prop:perturbations_rotation}}

To complete the proof of Theorem \ref{stab} we use the above Lemma, together
with Proposition \ref{individual}, and we see that any iterated Riesz
transform $R^{\beta}$ of odd order $N=\left\vert \beta\right\vert $ with
$\beta\neq\left(  N,0,...,0\right)  $,$\ $is bounded on the higher dimensional
analogue of the weight pair $\left(  \sigma,\omega\right)  $ constructed in
Proposition \ref{individual}, and can be rotated into a sum $S$ of iterated
Riesz transforms that includes $R^{\left(  N,0,...,0\right)  }$, and hence $S$
is unbounded on the weight pair $\left(  \sigma,\omega\right)  $. Since
stability under rotational change of variables is unaffected by rotation of
the operator, this completes our proof that all iterated Riesz transforms
$R^{\beta}$ of odd order are unstable under rotational changes of variable,
even when the measures are doubling with adjacency constant $\lambda_{\operatorname*{adj}}$ arbitrarily close to $1$.
This completes the proof of the main Theorem \ref{stab}.

To prove Theorem \ref{prop:perturbations_rotation}, suppose $R^{\beta}$ is an
odd order iterated Riesz transform; without loss of generality, assume that
$R^{\beta}\neq R_{1}^{|\beta|}$. Then by Lemma \ref{rotate_poly}, there is a
set $\Lambda$ of rotations of full measure such that for each $\Theta
\in\Lambda$, $\Theta$ rotates $R^{\beta}$ to $c(\Theta)R_{1}^{|\beta|}$ plus
mixed iterated Riesz transforms, where $c(\Theta)\neq0$. Then our construction
yields a weight pair $(\sigma,\omega)$ for which the norm inequality for
$R^{\beta}$ is bounded, but the norm inequality for the rotated operator can
be made arbitrarily large.

\section{Appendix}

We begin by using the counterexamples in \cite{LaSaUr} to show that the
Hilbert transform is two weight norm biLipschitz unstable on $\mathcal{S}%
_{\operatorname*{lfpB}}$. Then we demonstrate that the notion of stability
that is maximal for preserving the classical $A_{2}$ condition, is that of
\emph{biLipschitz} stability. Next, we show that all sparse bump functionals are biLipschitz stable on the pairs of doubling measures. After that, we give the details for arguments
surrounding classical doubling which were omitted from
\cite{NaVo}. And finally, we give the proof of the $T1$ theorem \ref{thm:starting_T1_thm}.

\subsection{BiLipschitz instability of the Hilbert transform for arbitrary
weight pairs}

Here we show that the Hilbert transform $H$ is two weight norm unstable under
biLipschitz transformations. We consider the measure pairs $\left(
\sigma,\omega\right)  $ and $\left(  \ddot{\sigma},\omega\right)  $
constructed in \cite{LaSaUr}, where $\left(  \sigma,\omega\right)  $ satisfies
the two weight norm inequality for $H$, while $\left(  \ddot{\sigma}%
,\omega\right)  $ does not, although it continues to satisfy the two-tailed
Muckenhoupt $\mathcal{A}_{2}$ condition. The measure $\omega$ is the standard
Cantor measure on $\left[  0,1\right]  $ supported in the middle-third Cantor
set $E$. The measures $\sigma=\sum_{k,j}s_{j}^{k}\delta_{z_{j}^{k}}$ and
$\ddot{\sigma}=\sum_{k,j}s_{j}^{k}\delta_{\ddot{z}_{j}^{k}}$ are sums of
weighted point masses located at positions $z_{j}^{k}$ and $\ddot{z}_{j}^{k}$
within the component $G_{j}^{k}$ removed at the $k^{th}$ stage of the
construction of $E$, and satisfy
\begin{equation}
0<c_{1}<\frac{\operatorname*{dist}\left(  z_{j}^{k},\partial G_{j}^{k}\right)
}{\left\vert G_{j}^{k}\right\vert },\frac{\operatorname*{dist}\left(  \ddot
{z}_{j}^{k},\partial G_{j}^{k}\right)  }{\left\vert G_{j}^{k}\right\vert
}<c_{2}<1, \label{separate}%
\end{equation}
independent of $k,j$. See \cite{LaSaUr} for notation and proofs.

It remains to construct a biLipschitz map $\Phi:\mathbb{R}\rightarrow
\mathbb{R}$ such that $\left(  \ddot{\sigma},\omega\right)  =\left(
\Phi_{\ast}\sigma,\Phi_{\ast}\omega\right)  $. For this, we first define
biLipschitz maps $\Phi:\overline{G_{j}^{k}}\rightarrow\overline{G_{j}^{k}}$ so
that $\Phi$ fixes the endpoints of $\overline{G_{j}^{k}}$ and $\ddot{z}%
_{j}^{k}=\Phi\left(  z_{j}^{k}\right)  $, and note that this can be done with
bounds independent of $k,j$ by (\ref{separate}). Now we extend the definition
of $\Phi$ to all of $\mathbb{R}$ by the identity map, and it is evident that
$\Phi$ is biLipschitz and pushes $\left(  \sigma,\omega\right)  $ forward to
$\left(  \ddot{\sigma},\omega\right)  $.

\subsection{Beyond biLipschitz maps for $A_{2}$ stability}

Here we initiate an investigation of how general a map can be, and still
preserve the two weight $A_{2}$ condition for all pairs of measures $\left(
\sigma,\omega\right)  $. We begin by defining some of the terminology we will
use in this subsection.

\begin{definition}
Let $\mu$ be a locally finite positive Borel measure on $\mathbb{R}^{n}$, and
let $\Phi:\mathbb{R}^{n}\rightarrow\mathbb{R}^{n}$ be a Borel measurable
function. We define the pushforward of the measure $\mu$ by the map $\Phi$ as
the unique measure $\Phi_{\ast}\mu$ such that
\[
\int_{E}\Phi_{\ast}\mu=\int_{\Phi^{-1}(E)}\mu,\ \ \ \ \ \text{for all Borel
sets }E\subset\mathbb{R}^{n}\textup{.}%
\]

\end{definition}

In the case $d\mu\left(  x\right)  =w\left(  x\right)  dx$ is absolutely
continuous, its pushforward for $\Phi$ sufficiently smooth is given by%
\[
\left(  \Phi_{\ast}\mu\right)  \left(  y\right)  \equiv w\left(
\Phi(y)\right)  \left\vert \det\frac{\partial\Phi}{\partial x}\left(
y\right)  \right\vert .
\]

\begin{definition}
A map $\Phi:\mathbb{R}^{n}\rightarrow\mathbb{R}^{n}$ is $A_{2}$\emph{-stable},
if there exists a constant $C>0$ such that for every pair of locally finite
positive Borel measures $\sigma,\omega$ we have
\[
A_{2}(\Phi_{\ast}\sigma,\Phi_{\ast}\omega)\leq CA_{2}(\sigma,\omega).
\]

\end{definition}

\begin{definition}
A map $\Phi:\mathbb{R}^{n}\rightarrow\mathbb{R}^{n}$ (not necessarily
invertible) is \emph{shape-preserving} if there exists $K\geq1$ such that for
every cube $Q\subset\mathbb{R}^{n}$ we can find cubes $Q_{\operatorname{small}%
}$ and $Q_{\operatorname{big}}$ with the properties,%
\[
Q_{\operatorname{small}}\subset\Phi{^{-1}(Q)}\subset Q_{\operatorname{big}%
}\,\qquad\dfrac{\ell(Q_{\operatorname{big}})}{\ell(Q_{\operatorname{small}}%
)}\leq K.
\]
We call such a set $\Phi{^{-1}(Q)}$ an \emph{almost cube}.
\end{definition}

Note that homeomorphisms on the real line are automatically shape-preserving,
as are quasiconformal maps in $\mathbb{R}^{n}$ \cite[Lemma 3.4.5]{AsIwMa}.

\begin{theorem}
\label{homeo}Let $\Phi:\mathbb{R}^{n}\rightarrow\mathbb{R}^{n}$ be
shape-preserving and Borel-measurable. Then the following two conditions are equivalent:

\begin{enumerate}
\item There exists a constant $C_{1}>0$ such that $\left\vert \Phi^{-1}\left(
Q\right)  \right\vert \leq C_{1}\left\vert Q\right\vert $ for every cube $Q$.

\item $\Phi$ is $A_{2}$-stable.
\end{enumerate}
\end{theorem}

\begin{remark}
\label{det control}If $\Phi$ is sufficiently regular that the usual change of
variables formula holds, e.g., $\Phi^{-1}$ is locally Lipschitz, then condition
$\left(  1\right)  $ becomes $\left\vert \det D\Phi^{-1}\right\vert \lesssim1$.
\end{remark}

\begin{proof}
Assume condition $\left(  1\right)  $ holds where $\Phi$ is shape-preserving
with constant $K$, and let $Q$ be an arbitrary cube in $\mathbb{R}^{n}$. Then%
\begin{align*}
&  A_{2}\left(  \Phi_{\ast}\sigma,\Phi_{\ast}\omega\right)  =\sup_{Q}\left(
\frac{\int_{Q}d\Phi_{\ast}\sigma}{\left\vert Q\right\vert }\right)  \left(
\frac{\int_{Q}d\Phi_{\ast}\omega}{\left\vert Q\right\vert }\right)  =\sup
_{Q}\left(  \frac{\int_{\Phi^{-1}\left(  Q\right)  }d\sigma}{\left\vert
Q\right\vert }\right)  \left(  \frac{\int_{\Phi^{-1}\left(  Q\right)  }%
d\omega}{\left\vert Q\right\vert }\right) \\
&  \leq C_{1}^{2}\sup_{Q}\left(  \frac{\int_{\Phi^{-1}\left(  Q\right)
}d\sigma}{\left\vert \Phi^{-1}Q\right\vert }\right)  \left(  \frac{\int
_{\Phi^{-1}\left(  Q\right)  }d\omega}{\left\vert \Phi^{-1}Q\right\vert
}\right)  \leq C_{1}^{2}K^{2n}\sup_{Q}\left(  \frac{\int
_{Q_{\operatorname{big}}}d\sigma}{\left\vert Q_{\operatorname{big}}\right\vert
}\right)  \left(  \frac{\int_{Q_{\operatorname{big}}}d\omega}{\left\vert
Q_{\operatorname{big}}\right\vert }\right)  \leq C_{1}^{2}K^{2n}A_{2}\left(
\sigma,\omega\right)  .
\end{align*}
Conversely, if condition $\left(  2\right)  $ holds, then with both measures
$\sigma$ and $\omega$ equal to Lebesgue measure, and for any cube $Q$, we
have,%
\[
\left(  \frac{\left\vert \Phi^{-1}\left(  Q\right)  \right\vert }{\left\vert
Q\right\vert }\right)  ^{2}=\left(  \frac{\int_{\Phi^{-1}\left(  Q\right)
}dx}{\left\vert Q\right\vert }\right)  \left(  \frac{\int_{\Phi^{-1}\left(
Q\right)  }dx}{\left\vert Q\right\vert }\right)  =\left(  \frac{\int_{Q}%
d\Phi_{\ast}\sigma}{\left\vert Q\right\vert }\right)  \left(  \frac{\int
_{Q}d\Phi_{\ast}\omega}{\left\vert Q\right\vert }\right)  \leq C.
\]

\end{proof}

\begin{remark}
If the pair $\left(  \Phi_{\ast}\sigma,\Phi_{\ast}\omega\right)  $ is in
$A_{2}$ for the \emph{single} choice of weights $d\sigma\left(  x\right)
=d\omega\left(  x\right)  =dx$, then the above proof shows that $\Phi$
preserves all $A_{2}$ pairs under the side assumption of shape-preservation.
\end{remark}

\begin{corollary}
Assume $\Phi:\mathbb{R}^{n}\rightarrow\mathbb{R}^{n}$ is a shape-preserving
invertible Lipschitz map with $\left\Vert D\Phi\right\Vert _{\infty}\leq1$.
Then $\Phi$ is $A_{2}$-stable if and only if $\Phi$ is biLipschitz.
\end{corollary}

\begin{proof}
By Theorem \ref{homeo} and Remark \ref{det control}, we see that $\Phi$ is
$A_{2}$-stable if and only if $\left\vert \det D\Phi\right\vert \gtrsim1$. But
then $1\leq C\left\vert \det D\Phi\right\vert \leq C^{\prime}\left\vert
D\Phi\right\vert ^{n}$, together with $\left\Vert D\Phi\right\Vert _{\infty
}\leq1$, shows that $\Phi$ is $A_{2}$-stable if and only if $\Phi$ is biLipschitz.
\end{proof}

\begin{corollary}
\label{group}Assume $\Phi:\mathbb{R}^{n}\rightarrow\mathbb{R}^{n}$ is
Borel-measurable and invertible, and that both $\Phi$ and $\Phi^{-1}$ are
shape-preserving. Then both $\Phi$ and $\Phi^{-1}$ are $A_{2}$-stable \emph{if
and only if} $\Phi$ is biLipschitz.
\end{corollary}

\begin{proof}
If both $\Phi$ and $\Phi^{-1}$ are $A_{2}$-stable, then from Theorem
\ref{homeo} we obtain that
\begin{align*}
\left\vert \Phi^{-1}\left(  Q\right)  \right\vert  &  \leq C_{1}\left\vert
Q\right\vert \text{ for every cube }Q,\\
\left\vert \Phi\left(  Q\right)  \right\vert  &  \leq C_{1}\left\vert
Q\right\vert \text{ for every cube }Q.
\end{align*}
Thus if $Q$ is a minimal cube containing both $x$ and $y$, then the almost
cube $\Phi^{-1}\left(  Q\right)  $ contains both $\Phi^{-1}\left(  x\right)  $
and $\Phi^{-1}\left(  y\right)  $, and so%
\[
\frac{\left\vert \Phi^{-1}\left(  x\right)  -\Phi^{-1}\left(  y\right)
\right\vert }{\left\vert x-y\right\vert }\lesssim\frac{\operatorname*{diam}%
\Phi^{-1}\left(  Q\right)  }{\operatorname*{diam}Q}\lesssim\frac{\left\vert
\Phi^{-1}\left(  Q\right)  \right\vert }{\left\vert Q\right\vert }\leq C_{1},
\]
and since the almost cube $\Phi\left(  Q\right)  $ contains both $\Phi\left(
x\right)  $ and $\Phi\left(  y\right)  $,%
\[
\frac{\left\vert \Phi\left(  x\right)  -\Phi\left(  y\right)  \right\vert
}{\left\vert x-y\right\vert }\lesssim\frac{\operatorname*{diam}\Phi\left(
Q\right)  }{\operatorname*{diam}Q}\lesssim\frac{\left\vert \Phi\left(
Q\right)  \right\vert }{\left\vert Q\right\vert }\leq C_{1}.
\]

\end{proof}

\subsection{Stability and sparse operators}\label{subsection:sparse_bumps}

Recall that a grid of dyadic cubes $\mathcal{S}$ is called $\eta$ -sparse,
$0<\eta<1$, if for every $Q\in\mathcal{S}$ there are subsets $E_{Q}\subset Q$
such that $\left\vert E_{Q}\right\vert \geq\eta\left\vert Q\right\vert $ and
the sets $\left\{  E_{Q}\right\}  _{Q\in\mathcal{S}}$ are pairwise disjoint.
Note that such an $\mathcal{S}$ satisfies the following $\frac{1}{\eta}%
$-Carleson condition,%
\begin{align*}
\sum_{Q^{\prime}\in\mathcal{S}:Q^{\prime}\subset Q}\left\vert Q^{\prime
}\right\vert  &  \leq\frac{1}{\eta}\sum_{Q^{\prime}\in\mathcal{S}:Q^{\prime
}\subset Q}\left\vert E_{Q^{\prime}}\right\vert \leq\frac{1}{\eta}\left\vert
Q\right\vert ,\ \ \ \ \ \text{for all }Q\in\mathcal{S},\\
\sum_{Q^{\prime}\in\mathcal{S}:Q^{\prime}\subset\Omega}\left\vert Q^{\prime
}\right\vert  &  \leq\frac{1}{\eta}\left\vert \Omega\right\vert
,\ \ \ \ \ \text{for all open sets }\Omega.
\end{align*}
Conversely, if $\mathcal{S}$ satisfies the $\Lambda$-Carleson condition,
\begin{equation}
\sum_{Q^{\prime}\in\mathcal{S}:Q^{\prime}\subset Q}\left\vert Q^{\prime
}\right\vert \leq\Lambda\left\vert Q\right\vert ,\ \ \ \ \ \text{for all }%
Q\in\mathcal{S}, \label{Lambda ineq}%
\end{equation}
then $\mathcal{S}$ is $\frac{1}{\Lambda}$-sparse, see e.g. \cite{LeNa}.

\begin{definition}
Given a sparse grid of cubes $\mathcal{S}$, we define the associated sublinear
\emph{sparse operator} $S$ by%
\begin{equation}\label{eq:def_spare_op}
Sf\left(  x\right)  \equiv\sum_{Q\in\mathcal{S}}\left(  \frac{1}{\left\vert
Q\right\vert }\int_{Q}\left\vert f\right\vert \right)  \mathbf{1}_{Q}\left(
x\right)  ,\ \ \ \ \ x\in\mathbb{R}^{n},
\end{equation}
and we say that $S$ is $\eta$-sparse if $\mathcal{S}$ is $\eta$-sparse.
\end{definition}

\begin{definition}
Let $\mathcal{U}$ be a biLipschitz invariant set of locally finite positive
Borel measures on $\mathbb{R}^{n}$. A functional $\mathcal{B}\left(
\sigma,\omega\right)  $ on pairs of measures $\left(  \sigma,\omega\right)  $
is called a \emph{sparse bump} functional on $\mathcal{U}$ if for every
$\eta\in\left(  0,1\right)  $, there exists a continuous increasing function
$\Gamma_{\eta}:\left(  0,\infty\right)  \rightarrow\left(  0,\infty\right)  $
such that for all $\eta$-sparse operators $S$,
\[
\mathfrak{N}_{S}\left(  \sigma,\omega\right)  \leq\Gamma_{\eta}\left(
\mathcal{B}\left(  \sigma,\omega\right)  \right)  ,\ \ \ \ \ \text{for all
}\left(  \sigma,\omega\right)  \in\mathcal{U}\times\mathcal{U}.
\]

\end{definition}

Obviously, no biLipschitz stable (bump) condition can characterize a
biLipschitz unstable weighted norm inequality. Here we will show that no
sparse bump functional can either. Note that it is shown in \cite{Ler} that
all (separated) Orlicz or entropy bump conditions, that are currently known to
imply boundedness of singular integrals, are sparse bump functionals on any
such $\mathcal{U}$. Here is the main result of this section.

\begin{theorem}
\label{bumps don't work}Let $\mathcal{U}_{\operatorname*{doub}}$ be the
biLipschitz invariant set of doubling measures on $\mathbb{R}^{n}$ (called
$\mathcal{S}_{\operatorname*{doub}}$ in the introduction), and let
$\mathcal{B}\left(  \sigma,\omega\right)  $ be a sparse\ bump\emph{
}functional on $\mathcal{U}_{\operatorname*{doub}}$. Then for any smooth
Calder\'{o}n-Zygmund operator $T$ that is biLipschitz unstable on pairs of
doubling weights, there is no continuous increasing function $\Gamma:\left(
0,\infty\right)  \rightarrow\left(  0,\infty\right)  $ such that%
\begin{equation}
\mathcal{B}\left(  \sigma,\omega\right)  \leq\Gamma\left(  \mathfrak{N}%
_{T}\left(  \sigma,\omega\right)  \right)  ,\ \ \ \ \ \text{for all }\left(
\sigma,\omega\right)  \in\mathcal{U}_{\operatorname*{doub}}.
\label{Gamma ineq}%
\end{equation}
In particular by Theorem \ref{stab}, we can take $T$ to be an iterated Riesz
transform of odd order.
\end{theorem}

\begin{remark}
This theorem, together with Theorem \ref{lerner thm} below, shows that no
sparse bump functional $\mathcal{B}\left(  \sigma,\omega\right)  $ can
characterize the two weight norm inequality for an iterated Riesz transform of
odd order on doubling measures.
\end{remark}

To prove Theorem \ref{bumps don't work} we will use a special case of the
groundbreaking sparse domination principle of A. Lerner. Recall that a Dini-regular Calder\'on-Zygmund operator $T$ with kernel $K$ is an operator where the kernel, rather than satisfying the size and smoothness estimates \eqref{size and smoothness}, instead satisfies
\[
|K(x,y)| \leq C_{CZ} |x-y|^{-n} \, ,
\]
\[
\left | K(x,y) - K(x',y) \right | + \left | K(y,x) - K(y,x') \right | \leq f_{CZ} \left ( \frac{|x-x'|}{|x-y|}\right ) |x-y|^{-n} \, ,
\]
where the nonnegative function $f_{CZ}$ satisfies the Dini condition
\[
\int\limits_0 ^1 f_{CZ} (t) \frac{dt}{t} < \infty \, .
\]

\begin{theorem}
[A. Lerner \cite{Ler2}]\label{lerner thm}Let $T$ be a Dini-regular
Calder\'{o}n-Zygmund operator, and let $f\in L^{1}\left(  \mathbb{R}%
^{n}\right)  $ be compactly supported. Then with $\eta_{n}=\frac{1}{2\left(
5\sqrt{n}\right)  ^{n}}$ there is an $\eta_{n}$-sparse grid $\mathcal{S}$
depending on $f$ such that%
\[
\left\vert Tf\left(  x\right)  \right\vert \leq C_{n,T} Sf (x)  ,\ \ \ \ \ \text{for }a.e.x\in
\mathbb{R}^{n} \, ,
\]
where $Sf$ is as in (\ref{eq:def_spare_op}).

\end{theorem}

Now we can give the proof of Theorem \ref{bumps don't work}.

\begin{proof}
[Proof of Theorem \ref{bumps don't work}]Suppose in order to derive a
contradiction that (\ref{Gamma ineq}) holds for some sparse\ bump\emph{
}functional $\mathcal{B}\left(  \sigma,\omega\right)  $ on $\mathcal{U}$ in
$\mathbb{R}^{n}$. Then for any BiLipschitz map $\Phi$, we have $\mathfrak{N}_{T}\left(  \Phi_{\ast}\sigma,\Phi_{\ast
}\omega\right)  =\mathfrak{N}_{\Phi_{\ast}T}\left(  \sigma,\omega\right)  $
and so if a compactly supported function $f\in L^{2}\left(  \sigma\right)  $ is chosen to be a near
extremizer for the norm $\mathfrak{N}_{\Phi_{\ast}T}\left(  \sigma
,\omega\right)  $, we have from Lerner's theorem, applied to the Dini-regular
Calder\'{o}n-Zygmund operator $\Phi_{\ast}T$, that there is an $\eta_{n}%
$-sparse operator $S$ such that
\begin{align*}
\mathfrak{N}_{\Phi_{\ast}T}\left(  \sigma,\omega\right)   &  \leq
2\frac{\left\Vert \Phi_{\ast}T\left(  f\sigma\right)  \right\Vert
_{L^{2}\left(  \omega\right)  }}{\left\Vert f\right\Vert _{L^{2}\left(
\sigma\right)  }}\leq2C_{n,T,\left\Vert \Phi\right\Vert }\frac{\left\Vert
S\left(  f\sigma\right)  \right\Vert _{L^{2}\left(  \omega\right)  }%
}{\left\Vert f\right\Vert _{L^{2}\left(  \sigma\right)  }}\\
&  \leq2C_{n,T,\left\Vert \Phi\right\Vert }\mathfrak{N}_{S}\left(
\sigma,\omega\right)  \leq C_{n,T,\left\Vert \Phi\right\Vert }\Gamma_{\eta
_{n}}\left(  \mathcal{B}\left(  \sigma,\omega\right)  \right) \\
&  \leq C_{n,T,\left\Vert \Phi\right\Vert }\Gamma_{\eta_{n}}\left(
\Gamma\left(  \mathfrak{N}_{T}\left(  \sigma,\omega\right)  \right)  \right)
,
\end{align*}
where the first line uses Theorem \ref{lerner thm}, the second line uses the definition of sparse bump
functional and the assumed inequality (\ref{Gamma ineq}). Thus two-weight norm inequalities for Dini-regular operators are biLipschitz stable, as defined in Definition \ref{defn:stability}. But by Theorem \ref{stab}, the inequality for $T$ equal an individual Riesz transform cannot be biLipschitz stable. This contradiction proves the theorem.
\end{proof}

\subsection{Modification of Transplantation to achieve Classical Doubling\label{subsection trans}}

In Section \ref{section:sprvsr_trans}, we constructed functions $v,u$ on a
cube $Q^{0}$ such that both $v,u$ are dyadically $\tau$-flat on $Q_{0}$. However, dyadic doubling does not imply continuous doubling on
$Q^{0}$. As such, we will need to modify the transplantation argument to
smooth out $v,u$ into weights $v^{\prime},u^{\prime}$ which are classically
doubling, as done in \cite{NaVo}. See also \cite{Naz, KaTr}.

We will describe how to attain $u^{\prime}$ from $u$, as the process for
$v^{\prime}$ and $v$ is identical. Recall in Proposition \ref{Riesz Nazarov}
we define $u$ by
\[
u= \left(  E_{Q^{0}}U \right)  \mathbf{1}_{Q^{0}} +\sum\limits_{t=0}^{m-1}%
\sum\limits_{Q\in\mathcal{K}_{t}}\langle U,h_{\mathcal{S}(Q)}%
^{\hor}\rangle\frac{1}{\sqrt{\left\vert \mathcal{S}%
(Q)\right\vert }}s_{k_{t+1}}^{Q,\hor}\,,
\]
where $s_{k_{t+1}}$ is constant on cubes in $\mathcal{K}_{t+1}$.

Define the grid $\widehat{\mathcal{K}}$ from $\mathcal{K}$ inductively as
follows. First set $\widehat{\mathcal{K}}_{0}\equiv\mathcal{K}_{0}$. Now given
$Q\in\widehat{\mathcal{K}}_{t}$, a cube $R\in\mathcal{K}_{t+1}$ is called a
\textit{transition cube} for $Q$ if $Q=\pi_{\mathcal{K}}R$ and $(\partial
\pi_{\mathcal{D}}R)\cap\partial Q$ is non-empty; as such define
$\widehat{\mathcal{K}}_{t+1}$ to consist of the cubes $P\in\mathcal{K}_{t+1}$
such that $\pi_{\mathcal{K}}P\in\widehat{\mathcal{K}}_{t}$ and $P$ is
\emph{not} a transition cube. Finally, set $\widehat{K} \equiv\bigcup
\limits_{t} \widehat{\mathcal{K}}_t$.

One can see that $\widehat{\mathcal{K}}$ consists of the cubes in
$\mathcal{K}$ not contained in a transition cube. This implies that if $R$ is
a transition cube, then $\pi_{\mathcal{K}} R \in\mathcal{K}$. It also implies
that no two transition cubes have overlapping interiors. Visually, the union
of the transition cubes for a cube $Q$ forms a \textquotedblleft
halo\textquotedblright\ for $Q$. Recalling that two distinct dyadic cubes in
$\mathcal{D}$ of the same size are \emph{adjacent} if their boundaries
intersect, we then note that two adjacent cubes in $\widehat
{\mathcal{K}}$ must then have the same $\mathcal{K}%
$-parent, and so are close to each other in the
tree distance of $\mathcal{K}$. The proof of the following lemma is left to
the reader, who is encouraged to draw a picture. It helps to note that in
$\mathbb{R}$, if two transition intervals $R_{1}$ and $R_{2}$ are at levels
$s$ and $s+2$, then there must be a transition interval $R$ at level $s+1$
such that $R$ lies between $R_{1}$ and $R_{2}$.

\begin{lemma}
\label{lem:adjacent_trans_cubes} Let $R_{1}\in\mathcal{K}_{s}$ be a transition cube.

\begin{enumerate}
\item If $R_{2}$ $\in\mathcal{K}_{t}$ is a transition cube such that the
interiors of $R_{2}$ and $R_{1}$ are disjoint, but not their closures, then
$t\in\{s-1,s,s+1\}$.

\item If $K\in\widehat{\mathcal{K}}_{t}$ is such that the interiors of $K$ and
$R_{1}$ are disjoint, but not their closures, then $t\in\{s-1,s\}$. And if $t
= s$, then $\pi_{\mathcal{K}} K = \pi_{\mathcal{K}} R_{1}$.
\end{enumerate}
\end{lemma}

With this in mind, given $Q\in\widehat{\mathcal{K}}_{t}$, define
\[
r_{k_{t+1}}^{Q,\hor}(x)\equiv%
\begin{cases}
s_{k_{t+1}}^{Q,\hor}(x) & \text{ if }x\text{ is not
contained in a transition cube for }Q\\
0 & \text{ otherwise }%
\end{cases}
\,.
\]
Now we may define
\begin{align*}
u_{\ell}^{\prime}  &  \equiv\left(  E_{Q^{0}}U \right)  \mathbf{1}_{Q^{0}}
+\sum\limits_{t=0}^{\ell-1}\sum\limits_{Q\in\widehat{\mathcal{K}}_{t}}\langle
U,h_{\mathcal{S}(Q)}^{\hor}\rangle\frac{1}{\sqrt
{\left\vert \mathcal{S}(Q)\right\vert }}r_{k_{t+1}}%
^{Q,\hor}\,,\quad0\leq\ell\leq m\,,\\
u^{\prime}  &  \equiv u_{m}^{\prime}\text{ and }v^{\prime}\equiv v_{m}%
^{\prime}.
\end{align*}
Given $x\in Q^{0}$ and $\ell\leq m$, if we define
\[
t(x)\equiv%
\begin{cases}
t\text{ if }x\text{ is contained in a transition cube belonging to
}\mathcal{K}_{t} \text{ for some } t< \ell\\
\ell\text{ otherwise }%
\end{cases}
,
\]
then pointwise we have
\[
u_{\ell}^{\prime}(x)=\left(  E_{Q^{0}}U \right)  \mathbf{1}_{Q^{0}} (x) %
+\sum\limits_{t=0}^{t(x)-1}\sum\limits_{Q\in\widehat{\mathcal{K}}_{t}}\langle
U,h_{\mathcal{S}(Q)}^{\hor}\rangle\frac{1}{\sqrt
{\left\vert \mathcal{S}(Q)\right\vert }}s_{k_{t+1}}%
^{Q,\hor}(x)\,,\quad0\leq\ell\leq m\,.
\]

The function $u^{\prime}$ is nearly a transplantation of $U$, as exhibited by
the following lemma, whose proof we leave to the reader. The reader should
note that for each cube contained in a transition cube, the value of $u_{\ell
}^{\prime}$ is equal to its average on the transition cube containing it.

\begin{lemma}
\label{lem:mod_trans_avgs} Let $\mathcal{K}$ be as above.

\begin{enumerate}
\item If $P \in\mathcal{K}$ is not contained in a transition cube, then $E_{P}
u^{\prime}_{\ell} = E_{\mathcal{S}(P)} U$.

\item If $P \in\mathcal{K}$ is contained in a transition cube $R$, then $E_{P}
u^{\prime}_{\ell} = E_{\mathcal{S}(\pi_{\mathcal{K}} R)} U$.

\item If $P \in\mathcal{D}$ is a cube for which $K_{t+1}\subsetneq P\subset
K_{t}$ where $K_{t+1}\in\mathcal{K}_{t+1}$ and $K_{t}\in\mathcal{K}_{t}$, then
$E_{P}u_{\ell}^{\prime}=E_{K_{t}}u_{\ell}^{\prime}$.
\end{enumerate}
\end{lemma}

\begin{remark}\label{rmk:A2_transition}
   From the above lemma, it follows that
   \[
   A_2 ^{\operatorname{dyadic}} \left ( u_{\ell} ' , v_{\ell}' \right ) \leq A_2 ^{\operatorname{dyadic}} \left ( U , V \right ) \leq 1 \, . 
   \]
\end{remark}

\begin{lemma}
\label{lem:adjacent doubling} If $P_{1},P_{2}$ are adjacent dyadic subcubes of
$Q^{0}$, then $\frac{E_{P_{1}}u^{\prime}}{E_{P_{2}}u^{\prime}}\in
(1-\tau,1+\tau)$. Similarly for $v^{\prime}$.
\end{lemma}

\begin{proof}
[Proof of Lemma \ref{lem:adjacent doubling}]Let $P_{1},P_{2}$ be adjacent
dyadic subcubes of $Q^{0}$. By Lemma \ref{lem:mod_trans_avgs} part (3), it
suffices to check the case when $P_{1},P_{2}\in\mathcal{K}$. We consider
various cases.

\textbf{Case 1}: neither $P_{1}$ nor $P_{2}$ is contained in a transition
cube, i.e., both belong to $\widehat{\mathcal{K}}$. Then $P_{1}$ and $P_{2}$
must have a common $\mathcal{K}$- parent, meaning
\[
\pi_{\mathcal{D}} \mathcal{S}(P_{1}) = \mathcal{S}(\pi_{\mathcal{K}} P_{1}) =
\mathcal{S}(\pi_{\mathcal{K}} P_{2}) = \pi_{\mathcal{D}} \mathcal{S}(P_{2})
\]
and so $\mathcal{S}(P_{1})$ and $\mathcal{S}(P_{2})$ must be equal or dyadic
siblings. By the first formula of Lemma \ref{lem:mod_trans_avgs} we get
$\frac{E_{P_{1}}u^{\prime}}{E_{P_{2}}u^{\prime}}\in(1-\tau,1+\tau)$.

\textbf{Case 2}: exactly one of the cubes, say $P_{1}$, is contained in a
transition cube $R_{1}$. Since $P_{2}$ is not in a transition cube, then the
only way for $P_{1},P_{2}$ to be adjacent is for both to have the same
$\mathcal{K}$-parent. And since $P_{2}$ is not contained in a transition cube,
then $R_{1}$ must in fact equal $P_{1}$, i.e. $P_{1}$ is a transition cube:
indeed, if $P_{1}$ were a level below $R_{1}$ in the grid $\mathcal{K}$, then
the only way $P_{2}$ can be adjacent to $P_{1}$ is by being in a transition
cube adjacent to $R_{1}$ or in $R_{1}$ itself, but the latter can't happen by
assumption on $P_{2}$.

Altogether, the above yields that $\mathcal{S}(\pi_{\mathcal{K}}%
P_{1})=\mathcal{S}(\pi_{\mathcal{K}}P_{2})=\pi_{\mathcal{D}}\mathcal{S}%
(P_{2})$. Thus by Lemma \ref{lem:mod_trans_avgs} parts (1) and (2), dyadic
$\tau$-flatness of $U$, and the fact that $P_{1}$ is a transition cube, we
have
\[
\frac{E_{P_{1}}u^{\prime}}{E_{P_{2}}u^{\prime}}=\frac{E_{\mathcal{S}%
(\pi_{\mathcal{K}}P_{1})}U}{E_{\mathcal{S}(P_{2})}U}=\frac{E_{\pi
_{\mathcal{D}}\mathcal{S}(P_{2})}U}{E_{\mathcal{S}(P_{2})}U}\in(1-\tau
,1+\tau).
\]

\textbf{Case 3:} both $P_{1}$ and $P_{2}$ are contained within transition
cubes, say $R_{1}$ and $R_{2}$ respectively. Using Lemma
\ref{lem:mod_trans_avgs}, it suffices to show the ratio
\[
\frac{E_{P_{1}}u^{\prime}}{E_{P_{2}}u^{\prime}}=\frac{E_{\mathcal{S}%
(\pi_{\mathcal{K}}R_{1})}U}{E_{\mathcal{S}(\pi_{\mathcal{K}}R_{2})}U}\,
\]
lies between $1-\tau$ and $1+\tau$. Note adjacency of $P_{1},P_{2}$ implies
$R_{1}$ and $R_{2}$ have disjoint interiors, but not closures, or are equal.

\textbf{Case 3a:} $R_{1}=R_{2}$. Then we get $\frac{E_{P_{1}}u^{\prime}%
}{E_{P_{2}}u^{\prime}}=1$.

\textbf{Case 3b:} $R_{1}$ and $R_{2}$ are of the same sidelength, but
$R_{1}\neq R_{2}$. Then both $R_{1}$ and $R_{2}$ are adjacent, and so
$\mathcal{S}(\pi_{\mathcal{K}}R_{1})$ and $\mathcal{S}(\pi_{\mathcal{K}}%
R_{2})$ must be equal or dyadic siblings. In either case, by the formula above
$\frac{E_{P_{1}}u^{\prime}}{E_{P_{2}}u^{\prime}}\in(1-\tau,1+\tau)$.

\textbf{Case 3c: }$R_{1}$ and $R_{2}$ are of different sidelengths, say
$\ell(R_{1})>\ell(R_{2})$. Since $P_{1},P_{2}$ are adjacent then $R_{1}$ and
$R_{2}$ have disjoint interiors, but not closures. It follows that if
$R_{1}\in\mathcal{K}_{t}$, then $R_{2}\in\mathcal{K}_{t+1}$ by Lemma
\ref{lem:adjacent_trans_cubes}. Thus $R_{1}$ is adjacent to $\pi_{\mathcal{K}%
}R_{2}$. In fact, since $R_{1}$ is a transition cube but $\pi_{\mathcal{K}}
R_{2}$ is not, then by Lemma \ref{lem:adjacent_trans_cubes} (2) we have
$\pi_{\mathcal{K}}R_{1} = \pi_{\mathcal{K}}^{\left(  2\right)  }R_{2}$ and so
\[
\mathcal{S}(\pi_{\mathcal{K}}R_{1})= \mathcal{S}(\pi_{\mathcal{K}} ^{(2)}
R_{2})= \pi_{\mathcal{D}}\mathcal{S}(\pi_{\mathcal{K}}R_{2}) \, .
\]
Thus
\[
\frac{E_{P_{1}}u^{\prime}}{E_{P_{2}}u^{\prime}}=\frac{E_{\mathcal{S}%
(\pi_{\mathcal{K}}R_{1})}U}{E_{\mathcal{S}(\pi_{\mathcal{K}}R_{2})}U}%
=\frac{E_{\pi_{\mathcal{D}}\mathcal{S}(\pi_{\mathcal{K}}R_{2})}U}%
{E_{\mathcal{S}(\pi_{\mathcal{K}}R_{2})}U}\in(1-\tau,1+\tau)\,.
\]
This completes the proof.
\end{proof}

Showing $u^{\prime}$ has relative adjacency constant $1+o_{\tau \to 1} \left(  1\right)  $
as $\tau\rightarrow0$ on $Q^{0}$ follows from Lemma
\ref{lem:adjacent doubling} and a standard argument, and similarly for
$v^{\prime}$.

\subsection{Proof of the T1 Theorem \ref{thm:starting_T1_thm}}\label{section:pf_T1}
   By \cite[Theorem 2]{SaShUr10}, we have
   \[
   \mathfrak{N}_{T} \left ( \sigma, \omega \right ) \lesssim \sqrt{\mathcal{A}_2 \left (\sigma, \omega \right )} + \mathfrak{T}_{T} (\sigma, \omega)  + \mathfrak{T}_{T^*} (\omega , \sigma) + \mathcal{E} (\sigma, \omega) + \mathcal{E} (\omega, \sigma) \, , 
   \]
   where the two-tailed $\mathcal{A}_2$ condition is given by
   \[
   \mathcal{A}_2 \left (\sigma, \omega \right ) \equiv \sup\limits_{Q \in \mathcal{P}^n} \left ( \int\limits_{\mathbb{R}^n} \left ( \frac{\ell \left ( Q \right )}{\left ( \ell \left ( Q \right ) + \left |x - c_Q \right | \right )^2} \right )^{n} d \sigma \right ) \left ( \int\limits_{\mathbb{R}^n} \left ( \frac{\ell \left ( Q \right )}{\left ( \ell \left ( Q \right ) + \left |x - c_Q \right | \right )^2} \right )^{n} d \omega \right ) \, ,
   \] 
   and the energy condition is defined by  
   \[
   \mathcal{E} \left (\sigma, \omega \right )^2 \equiv \sup\limits_{I = \dot \bigcup_r J_r} \frac{1}{\left |I \right |_{\sigma}}  \sum\limits_{r=1}^{\infty}  \operatorname{P} (J_r , \sigma) ^2 \left | J_r \right |_{\omega} \operatorname{E} (J_{r}, \omega ) ^2 \, ,
   \]
   where the supremum is taken over all cubes $I \in \mathcal{P}^n$ and all disjoint decompositions of $I \in \mathcal{P}^n$ into disjoint cubes $\dot{\bigcup\limits_{r}} J_r$. Within the energy condition we also have the Poisson average $\operatorname{P}(J, \sigma)$, which is defined by
   \[
   \operatorname{P}(J, \sigma) \equiv \int\limits_{\mathbb{R}^n}  \frac{\ell \left ( J \right )}{\left ( \ell \left ( J \right ) + \left |x - c_J \right | \right )^{n+1 }}  d \sigma \, ,
   \]
   and we also define
   \[
   \operatorname{E} (J_{r}, \omega ) ^2 \equiv \frac{1}{\left | J_r \right |_{\omega}} \int\limits_{J_r} \left | \frac{x-A}{\ell \left (J_r \right )}   \right |^2 d \omega\left ( x\right ) \, , \text{ and } A \equiv \frac{1}{\left | J_r \right |_{\omega}} \int\limits_{J_r} z d\omega(z) \, .
   \]
   Since $\operatorname{E} (J_{r}, \omega ) ^2 \leq 1$, then the Energy Condition is bounded by the Pivotal Condition
   \[
   \mathcal{V} (\sigma, \omega) ^2 \equiv \sup\limits_{I = \dot \bigcup_r J_r} \frac{1}{\left |I \right |_{\sigma}}  \sum\limits_{r=1}^{\infty}  \operatorname{P} (J_r , \sigma) ^2 \left | J_r \right |_{\omega}  \, ,
   \]
   
   By \cite[Theorem 4]{Gri}, if $\sigma$ and $\omega$ are doubling, then the tailed $\mathcal{A}_2$ condition is equivalent to the classical $A_2$ condition, i.e.,
   \[
   \mathcal{A}_2 \left (\sigma, \omega \right ) \lesssim A_2 (\sigma, \omega) \, ;
   \]
   see also \cite[Proposition 39]{AlLuSaUr3} for further details.

   As for the pivotal condition, a dyadic decomposition yields that the Poisson average of $\sigma$ on $Q$ is controlled by the expectation of $\sigma$ on $Q$, i.e.,
   \[
   \operatorname{P}(Q, \sigma) \lesssim \frac{\left | Q \right |_{\sigma}}{ \left | Q \right |} +\sum\limits_{k=1}^{\infty} \int\limits_{2^{k+1} Q \setminus 2^k Q}  \frac{\ell \left ( Q \right )}{\left ( \ell \left ( Q \right ) + \left |x - c_Q \right | \right )^{n+1 }}  d \sigma 
   \lesssim \sum\limits_{k=0}^{\infty} \left | 2^k Q \right |_{\sigma} \frac{\ell \left ( Q \right )}{\left ( 2^k \ell \left ( Q \right )\right )^{n+1}} 
   \]
   \[
    \lesssim \left | Q \right |_{\sigma} \sum\limits_{k=0}^{\infty} 2^{(n+\epsilon)k}  \frac{\ell \left ( Q \right )}{\left ( 2^k \ell \left ( Q \right )\right )^{n+1}} = \frac{\left | Q \right |_{\sigma}}{\left |Q \right |} \sum\limits_{k=0}^{\infty} 2^{-(1-\epsilon)k}  \lesssim \frac{\left | Q \right |_{\sigma}}{\left |Q \right |} \, ,
   \]
   where the inequality between the first and second line follows by the hypothesis on the doubling constants, and the last inequality follows because $\epsilon < 1$ implies the geometric series converges.
   Thus we can estimate
   \[
   \mathcal{V} (\sigma, \omega) ^2 \lesssim \sup\limits_{I = \dot \bigcup_r J_r} \frac{1}{\left |I \right |_{\sigma}}   \sum\limits_{r=1}^{\infty}  \operatorname{P} (J_r , \sigma) ^2  \left | J_r \right |_{\omega} \lesssim \sup\limits_{I = \dot \bigcup_r J_r} \frac{1}{\left |I \right |_{\sigma}}  \sum\limits_{r=1}^{\infty} \frac{\left | J_r \right |_{\sigma}^2}{\left |J_r \right |^2}  \left | J_r \right |_{\omega}
   \]
   \[
   \lesssim  A_2 (\sigma, \omega) \sup\limits_{I = \dot \bigcup_r J_r}  \frac{1}{\left |I \right |_{\sigma}}  \sum\limits_{r=1}^{\infty} \left | J_r \right |_{\sigma} \lesssim A_2 \left (\sigma, \omega \right ) \, .
   \] 
   Combining all the above estimates with the corresponding dual estimates yields the theorem.

   Alternatively, rather than applying \cite[Theorem 2]{SaShUr10}, one can apply \cite[Theorem 2.6 (1)]{SaShUr7}, and then in our particular situation where both measures $\sigma$ and $\omega$ are doubling, one can dispose of the weak-boundedness property using an argument similar to \cite[Lemma 2.4]{Hyt3} or in \cite[Proof of Lemma 14]{AlLuSaUr3}.

\end{document}